\RequirePackage{fix-cm}
\documentclass[smallextended]{svjour3}       
\smartqed  
\usepackage{graphicx}
\usepackage{amsfonts}
\usepackage{amssymb}
\usepackage{xcolor}
\usepackage{comment}
\usepackage{algorithm}
\usepackage{algpseudocode}
\usepackage[hidelinks]{hyperref}
\usepackage{subcaption}
\usepackage{epstopdf}
\graphicspath{ {./images/} }
\usepackage{float}
\usepackage{csvsimple, booktabs, datatool}
\usepackage{siunitx}
\usepackage{tabularx}
\usepackage{adjustbox}
\usepackage{multirow}
\usepackage{tikz}
\makeatletter
\let\cl@chapter\undefined
\makeatother
\usepackage{cleveref}
\makeatletter
\providecommand\theHALG@line{\thealgorithm.\arabic{ALG@line}}
\makeatother
\usepackage{anyfontsize}

\makeatletter
\long\def\@makecaption#1#2{%
 \captionstyle
 \ifx\@captype\fig@type
   \vskip\figcapgap
 \fi
 \setbox\@tempboxa\hbox{{\floatlegendstyle #1\floatcounterend}%
 \capstrut #2}%
 \ifdim \wd\@tempboxa >\hsize
   {\floatlegendstyle #1\floatcounterend}\capstrut #2\par
 \else
   \hbox to\hsize{\leftlegendglue\unhbox\@tempboxa\hfil}%
 \fi
 \ifx\@captype\fig@type\else
   \vskip\tabcapgap
 \fi}
 \makeatother

\spnewtheorem{ass}{Assumption}{\bf}{\it}

\def\<#1,#2>{\langle #1,#2 \rangle}

\DeclareMathOperator{\trace}{trace}

\DeclareMathOperator{\argmax}{argmax}

\newcommand{\mC}{\mathcal{C}}
\newcommand{\cS}{\mathcal{S}}
\newcommand{\cD}{\mathcal{D}}
\newcommand{\cF}{\mathcal{F}}

\newcommand{\cU}{\mathcal{U}}
\newcommand{\cY}{\mathcal{Y}}
\newcommand{\R}{\mathbb{R}}
\newcommand{\bI}{\mathbf{I}}
\newcommand{\bzero}{\mathbf{0}}
\newcommand{\bone}{\mathbf{1}}
\newcommand{\diag}{\operatorname{diag}}

\algnewcommand{\IIf}[1]{\State\algorithmicif\ #1\ \algorithmicthen}
\algnewcommand{\EndIIf}{\unskip\ \algorithmicend\ \algorithmicif}


\tikzstyle{nodes1} = [circle, rounded corners, minimum width=1cm, minimum height=1cm,text centered, draw=black, fill=red!30]
\tikzstyle{arrow} = [thick,->,>=stealth]

\begin{document}

\title{Globally Solving Concave Quadratic {Programs} via Doubly Nonnegative Relaxation
}

\titlerunning{Globally Solving Concave QPs via DNN Relaxation}        

\author{Zheng Qu \and Tianyou Zeng \and Yuchen Lou 
}

\authorrunning{Z. Qu, T. Zeng, Y. Lou} 

\institute{Z. Qu \at
    {School of Mathematical Sciences, Shenzhen University}. \\
    \email{{quzheng2003@gmail.com}}           
    \and
    T. Zeng \at
    Department of Mathematics, The University of Hong Kong.\\
    \email{logic@connect.hku.hk}
    \and
    Y. Lou \at
    Department of Industrial Engineering and Management Sciences, Northwestern University. \\
    \email{yuchenlou2026@u.northwestern.edu}
}

\date{
}

\maketitle
\begin{abstract}
We consider the problem of maximizing a convex quadratic function over a bounded polyhedral set. We design a new framework based on SDP relaxations and cutting plane methods for solving the associated reference value problem. The major novelty is a new way to generate valid cuts through the doubly nonnegative (DNN) relaxation. We establish various theoretical properties of the DNN relaxation, including its equivalence with the Shor relaxation of an equivalent quadratically constrained problem, the strong duality, and the generation of valid cuts from an approximate solution of the DNN relaxation  returned by an arbitrary SDP solver. Computational results on both real and synthetic data demonstrate the efficiency of the proposed method and its ability to solve high-dimensional problems with dense data. In particular, our new algorithm successfully solves in 3 days the reference value problem arising from computational biology for a dataset containing more than 300,000 instances of dimension 78. In contrast, CPLEX or Gurobi is estimated to require years of computational time for the same dataset on the same computing platform.

\end{abstract}

\section{Introduction}
\label{sec:intro}

\subsection{Background}\label{subsec:bac}
We consider the following problem of maximizing a convex quadratic function over a  polyhedral set:
\begin{equation}
    \label{prob:original_max_qp0}
    \begin{aligned}
        \Phi^*(\cF):=  \max_{y\in \R^{k}}  &  \qquad   \Phi(y)  \\
        ~ \textrm{s.t.} &\qquad   y\in \cF.
    \end{aligned}
\end{equation}
Here,
$$\Phi(y) \equiv y^\top Q y + 2 d^\top y+\nu$$ is a convex quadratic function, where
$Q\in \R^{k\times k}$ is a positive semidefinite matrix, $d\in \R^k$ is a vector and $\nu\in \R$ is a scalar. The feasible region $\cF$ is a polyhedral set written in the following form:
\begin{align}\label{a:F}
    \cF:=\{ y\in \R^{k}: F^\top y\leq w,\,\,y\geq 0\} ,
\end{align}
where
$F\in \R^{k\times m}$ and $w\in \R^m$. We further assume that $\cF$ has nonempty interior and is bounded.
Note that problem~\eqref{prob:original_max_qp0} is equivalent to minimizing a concave quadratic function over a polyhedral set.
For this reason, we shall refer to problem~\eqref{prob:original_max_qp0} as a concave quadratic program (QP).

Concave QP has important applications in many areas, including  engineering \cite{guisewite1990minimum,momoh1995economic}, industry \cite{baesens2003benchmarking,burkard1998quadratic}  and medical diagnosis \cite{fung2003disputed,mangasarian1995breast}.
More recently, concave QP finds new applications in computational biology~\cite{JIAO2021559,zhao2020new}.
The biology problem concerned is the detection of undiscovered protein/genome sequence and a key step in the approach proposed in~\cite{JIAO2021559,zhao2020new} consists in solving the following problem:
\begin{align}
    \label{a:possddf0}
    \mathrm{determine~whether~~}
    \nu_R > \Phi^*(\cF) \enspace \mathrm{~or~~}  \Phi^*(\cF)\geq \nu_R \enspace ?
\end{align}
Here, $\nu_R$ is a certain prescribed reference value. For convenience, we shall refer to~\eqref{a:possddf0} as the \textit{reference value problem} associated with the concave QP problem~\eqref{prob:original_max_qp0}.
In this real application, $m$ is fixed to $22$ and the number of known protein/genome sequences of a certain species is equal to $m+k$.
Depending on the  species considered, the dimension $k$  or the number of reference value problems  may be very large.
For example, we have access to a dataset with 317,584 instances of dimension $k=78$, and a dataset with 3 instances of dimension $k=819$.
The new protein detection problem requires one to solve the reference value problem for \textbf{all} the instances in the dataset.

We randomly selected 100 instances from the 317,584 instances of dimension $k=78$, and employed two off-the-shelf commercial solvers \texttt{CPLEX} and \texttt{Gurobi} to solve them.
The time limit was set as 600 seconds for both solvers.
The results show that \texttt{CPLEX} fails to give an answer to the reference value problem~\eqref{a:possddf0} on 92 instances, and \texttt{Gurobi} fails on 48 instances.
The  performance of both \texttt{CPLEX} and \texttt{Gurobi} clearly does not make the grade, as the computational time of solving all 317,584 instances using \texttt{CPLEX} or \texttt{Gurobi} is estimated to be \textbf{at least 3 years}\footnote{If half of the instances require at least 600 seconds to solve on average, in total we need at least $600\times 317584/2=95275200$ seconds, which is roughly 3 years.  }.
Moreover, the 3 instances of dimension $k=819$ cannot be handled by
\texttt{CPLEX} nor \texttt{Gurobi} due to the large dimension.

Motivated by the challenges raised above, we develop in this paper a new method for solving the reference value problem~\eqref{a:possddf0} associated with the concave QP problem~\eqref{prob:original_max_qp0}.
This new method can also be naturally extended to a global solver of~\eqref{prob:original_max_qp0}.
We focus on the efficiency of the algorithm  for large dimensional problems.
In particular, the algorithm should be highly efficient so that the 317,584 instances of dimension $k=78$ can be solved in a reasonable amount of time.
Moreover, the algorithm should be able to deal with high-dimensional problems so that  instances of dimension up to $k=819$ can be handled as well.

\subsection{Related work}

Problem~\eqref{prob:original_max_qp0} falls into the class of general nonconvex QP problems.   There are three major techniques to find a global solution of general QP problems: reformulation, branching, and relaxation~\cite{liuzzi2022computational,nohra2021spectral}.
QP problems can be reformulated in different ways, including bilinear reformulation~\cite{konno1976cutting,konno1976maximization}, KKT reformulation~\cite{burer2008finite,chen2012globally}, completely positive reformulation~\cite{Burer2012}, and mixed-integer LP reformulation~\cite{gondzio2021global,xia2020globally}.
After certain reformulation of the QP according to the problem structures, branch-and-bound (B\&B) algorithms are employed in nearly all the state-of-the-art global QP solvers.
There exist various branching strategies based on different reformulations~\cite{chen2012globally,xia2020globally}. Importantly, the efficiency of the bounding step  crucially depends on the quality of the relaxations utilized.
Well-known relaxations include, for example, McCormick relaxation \cite{mccormick1976computability}, semidefinite programming (SDP) relaxation \cite{luo2010semidefinite,nesterov1998semidefinite},  convex quadratic relaxation by separable programming or DC (difference of convex functions) programming \cite{horst2013handbook}, and doubly nonnegative relaxation~\cite{Burer2012}.

Regarding the concave QP problem~\eqref{prob:original_max_qp0}, it has also been extensively studied in the literature and  is known to be an NP-hard problem \cite{pardalos1991quadratic}.
Traditional methods mainly include enumerative methods~\cite{cabot1970solving}, cutting plane methods~\cite{konno1976maximization,Tuy1964concave}, successive approximation methods (see Chapter~6 in \cite{horst2013global}), and branch-and-bound methods~\cite{burer2008finite,chen2012globally}.
Details and references for early work on concave QP can be found in the survey paper of Pardalos and Rosen \cite{pardalos1986methods} and in the book of Horst and Tuy \cite{horst2013global}.
More recently, Zamani~\cite{zamani2019new} proposed a method which combines the cutting plane method and the branch-and-bound method for concave QP.
Telli et al.~\cite{telli2020successive} proposed  to approximate the global optimum by solving a sequence of linear programs constructed from a local approximation set.
Hladik et al.~\cite{HladikMilan20,HladikMilanZamani21} proposed new bounds for concave QP based on factorization.

\subsection{Review on the doubly nonnegative relaxation}
In this section, we recall the concept of the doubly nonnegative (DNN) relaxation of a QP.
We first reformulate the original QP~\eqref{prob:original_max_qp0} into an equivalent form with linear equality constraints.
Let $n=m+k$. One can lift the feasible region $\cF$ into $\R^n$ and
 rewrite~\eqref{prob:original_max_qp0} in the following form:
 \begin{equation}
    \label{prob:original_max_qps}
    \begin{aligned}
          \max_{x\in \R^n}  &  \qquad   x^\top \begin{pmatrix}
            Q & \bzero  \\
            \bzero^\top & 0
        \end{pmatrix} x + 2 \begin{pmatrix} d^\top &\bzero^{{\top}}\end{pmatrix} x +\nu \\
        ~ \textrm{s.t.} &\qquad    \begin{pmatrix} F^\top & \bI\end{pmatrix} x=w \\ & \qquad\enspace  x\geq 0.
    \end{aligned}
        \end{equation}
In problem~\eqref{prob:original_max_qps} above,  $\bf{0}$ denotes the zero vector or matrix of appropriate dimension,  and  $\bf{I}$  denotes the identity matrix of appropriate dimension.
A square matrix is completely positive (CP) if it can be written as $BB^\top$ where $B$ is a matrix with  nonnegative elements only.
It is known, given in the theorem below, that problem~\eqref{prob:original_max_qps} can be reformulated as a linear program over the  cone of CP matrices.
\begin{theorem}[\cite{Burer2012}]
    The quadratic program~\eqref{prob:original_max_qps} is equivalent to
    \begin{equation}
        \label{prob:cpequiv}
        \begin{aligned}
            \max_{\substack{X\in \cS^{n}\\ x\in \R^{n}}} \quad & \<\begin{pmatrix}
            Q & \bzero  \\
            \bzero^\top & 0
        \end{pmatrix}, X> + 2 \begin{pmatrix} d^\top &\bzero^{{\top}}\end{pmatrix}  x+\nu \\
            \mathrm{s.t.} \quad & \begin{pmatrix} F^\top & \bI\end{pmatrix}x = w \\
            \quad & \diag\left(\begin{pmatrix} F^\top & \bI\end{pmatrix} X \begin{pmatrix} F \\ \bI\end{pmatrix}\right)=w\circ w \\
            \quad & \begin{pmatrix}
                X & x \\ x^\top  & 1
            \end{pmatrix} \in \mC_{n+1}.
        \end{aligned}
    \end{equation}
\end{theorem}
Here, $\circ$ denotes the operation of the element-wise product of two vectors, $\cS^{n}$ is the space of $n$-by-$n$ symmetric matrices and
$\mC_{n+1}$ is the
  cone of $(n+1)$-by-$(n+1)$ completely positive  matrices:
$$
\mC_{n+1}=\left\{BB^\top: B \in \R_+^{(n+1)\times \ell}, \ell\in \mathbb{N} \right\}.
$$
The CP cone $\mC_{n+1}$ can be approximated from the outside by a hierarchy of cones  of linear- and semidefinite-representable cones $\cD_0 \supset \cD_{1}\supset \dots\supset \mC_{n+1}$, with $\cD_0$ corresponding to the cone of doubly nonnegative (DNN) matrices, i.e., the matrices that are positive semidefinite and elementwisely nonnegative.
Replacing $\mC_{n+1}$ by $\cD_0$, we obtain the following \textit{doubly nonnegative relaxation} of~\eqref{prob:original_max_qps}:
\begin{equation}
    \label{prob:dnnrelax}
    \begin{aligned}
        \bar\Phi(\cF):=\max_{\substack{X\in \cS^{n}\\ x\in \R^{n}}} \quad & \<\begin{pmatrix}
            Q & \bzero  \\
            \bzero^\top & 0
        \end{pmatrix}, X> + 2 \begin{pmatrix} d^\top &\bzero^{{\top}}\end{pmatrix} x+\nu \\
        \textrm{s.t.} \quad & \begin{pmatrix} F^\top & \bI\end{pmatrix} x = w \\
        \quad & \diag\left(\begin{pmatrix} F^\top & \bI\end{pmatrix} X \begin{pmatrix} F \\ \bI\end{pmatrix}\right)=w\circ w \\
        \quad & \begin{pmatrix}
            X & x \\ x^\top  & 1
        \end{pmatrix} \geq 0,\enspace
        \begin{pmatrix}
            X & x \\ x^\top  & 1
        \end{pmatrix} \succeq 0.
    \end{aligned}
\end{equation}
It is easy to see that:
\begin{align}\label{a:barF}
\bar\Phi(\cF) \geq \Phi^*({\cF}).
\end{align}
In the following we refer to~\eqref{prob:dnnrelax} as the DNN relaxation of~\eqref{prob:original_max_qps}, and $\bar \Phi(\cF)$ as the DNN bound of $\Phi^*({\cF})$.
 Note that~\eqref{prob:dnnrelax} is a semidefinite program (SDP) and can be solved numerically by several existing SDP solvers, e.g.~\cite{sun2020sdpnal+,TohToddTutuSDPT3,WenGoldfardYin10}.
There are also efficient solvers specially dedicated to approximate the DNN bound $\bar\Phi(\cF)$, e.g.~\cite{KimKojimaToh16}.

\subsection{Approach and contribution}
\label{sec:approach_contribution}

We summarize in this subsection the main routine of our proposed algorithm and its novelty.

A vector $\bar y\in \cF$ is a Karush–Kuhn–Tucker (KKT) point of~\eqref{prob:original_max_qp0}  if it satisfies
\begin{equation}\label{eq:sdfsew1}
    \begin{aligned}
        0 ~ = ~ \max    &\qquad  \<\nabla \Phi(\bar y), y-\bar y> \\
        ~ \textrm{s.t.} & \qquad    y \in \cF.
    \end{aligned}
\end{equation}
A vector $\bar y\in \cF$ is said to be a \textit{KKT vertex} if $\bar y$ is both a vertex of the polyhedral set $\cF$ and a KKT point of~\eqref{prob:original_max_qp0}. Note that any optimal solution of~\eqref{prob:original_max_qp0} must be a KKT point, and at least one optimal solution of~\eqref{prob:original_max_qp0} is a vertex of $\cF$.

In the proposed method, we start by searching for a  KKT vertex $\bar y$ of~\eqref{prob:original_max_qp0}.
If $\Phi({\bar y}) \geq \nu_R$, we conclude that $\Phi^*(\cF) \geq \nu_R $ and the reference value problem~\eqref{a:possddf0} is solved.
If $\Phi({\bar y}) < \nu_R$, we compute the DNN bound $\bar\Phi(\cF)$ defined by~\eqref{prob:dnnrelax}.
If
\begin{align}\label{a:barPvR}\bar \Phi(\cF)< \nu_R,
\end{align}
we conclude that  $\Phi^*(\cF) < \nu_R $ and the reference value problem is solved.
If instead $\bar \Phi(\cF)\geq  \nu_R$, we propose to rely on the classical cutting plane method to proceed.

We will iteratively add valid cuts to the original problem~\eqref{prob:original_max_qp0} so that the feasible region is reduced successively until we find an answer to~\eqref{a:possddf0}.
The major novelty of our approach is the computation of valid cuts from the DNN relaxation.
Let $\bar y$ be a KKT vertex of the QP problem~\eqref{prob:original_max_qp0} such that $\Phi({\bar y})<\nu_R$.
If there is a region $\Delta \ni \bar y$  such that
\begin{align}\label{a:opsd}
    \nu_R >  \Phi^*(\cF\cap \Delta):= \max \{\Phi(y): y \in  \cF\cap\Delta\},
\end{align}
then we only need to consider the restricted region $\cF \backslash \Delta$ instead of $\cF$ to find an answer to~\eqref{a:possddf0}.
We shall search for a region $\Delta \ni \bar y$ such that~\eqref{a:opsd} holds and a valid cut can be added to describe the restricted region $\cF\backslash \Delta$.
This cut will exclude a region which contains the KKT point {$\bar y$}.
Therefore, we can continue to deal with a concave QP problem with a strictly smaller feasible region $\cF\backslash \Delta$, and the same process is repeated until an answer to~\eqref{a:possddf0} is obtained.

It is interesting to note that~\eqref{a:opsd} is in fact a reference value problem defined in~\eqref{a:possddf0}.
Directly verifying~\eqref{a:opsd} being difficult, it is common to replace $\Phi^*(\cF\cap \Delta)$ in~\eqref{a:opsd} by a computable upper bound of it.
In~\cite{Tuy1964concave}, Tuy proposed to generate a valid cut using the following computable upper bound of $\Phi^*(\cF\cap \Delta)$:$$\max \{\Phi({y}): {y}\in \Delta\}.$$
In~\cite{konno1976maximization}, Konno proposed  another computable upper bound of $\Phi^*(\cF\cap \Delta)$:
\begin{align}\label{a:sfdsfertts}
    \bar \Phi^K(\cF\cap \Delta):=\max\{\Psi({y},{\tilde y}): {y}\in \Delta, {\tilde  y} \in \cF \},
\end{align}
where $\Psi({y},{\tilde y}):={y^\top Q \tilde y+ d^\top y+d^\top \tilde y+\nu}$.
In this paper,
we propose to employ the DNN bound $\bar \Phi(\cF)$ and validate~\eqref{a:opsd} by the following condition:
\begin{align}\label{a:posdfw}
\bar \Phi(\cF\cap \Delta)< \nu_R.
\end{align}
We show that the cut generated by our method is always deeper than Konno's cut, by proving in \Cref{thm:phi-val-comp} that
\begin{align}\label{a:psdfwe}
 \bar \Phi(\cF\cap \Delta)\leq  \bar \Phi^K(\cF\cap \Delta) .
\end{align}

The computation of {the} DNN bounds in~\eqref{a:barPvR} and~\eqref{a:posdfw} clearly plays {a} decisive role in the overall performance of the procedure described above.
With the goal of developing an efficient and robust concave QP solver in the large-scale setting, we also tackle the following two problems:
\begin{enumerate}
    \item  It is known that the SDP problems in the form of~\eqref{prob:dnnrelax} do not have an interior point {(see Proposition 8.3 and the last paragraph of Section 8.2.2 in~\cite{Burer2012})}.
    However, the theoretical convergence of the existing popular SDP solvers (e.g.~\cite{sun2020sdpnal+,TohToddTutuSDPT3,WenGoldfardYin10}) is established with the assumption of Slater's condition, i.e., the existence of an interior point.
\item
Any numerical solver using finite floating-point arithmetic is incapable of returning the exact value of $\bar\Phi(\cF)$.
Let $\nu$ be the  value returned by  a numerical SDP solver (e.g. \texttt{MOSEK} and { \texttt{SDPNAL+} \cite{sun2020sdpnal+}}) when solving~\eqref{prob:dnnrelax}.
Then $\nu$ is an approximation (whatever the precision is) of $\bar\Phi(\cF)$.
In particular, $\nu< \nu_R$ does not necessarily imply~\eqref{a:barPvR}.
A similar problem also occurs for computing $\bar \Phi(\cF\cap \Delta)$ and verifying~\eqref{a:posdfw}.
\end{enumerate}

To address the first issue, we prove in~\Cref{lemma:equiv-shor} that the DNN relaxation~\eqref{prob:dnnrelax} has an equivalent SDP formulation~\eqref{prob:SQP3Shor}, which corresponds to  {the Shor} relaxation applied to an equivalent quadratically constrained quadratic program (QCQP) of~\eqref{prob:original_max_qp0}.
We then show in~\Cref{prop:sd} that under some mild assumptions, Slater's condition holds for~\eqref{prob:SQP3Shor}, and hence all the above-mentioned SDP solvers are guaranteed to converge when applied to solve~\eqref{prob:SQP3Shor}.
For the second issue, we establish an explicit formula in~\Cref{prop:inexcs} for computing a valid upper bound of $\Phi^*(\cF)$ from any inexact primal-dual solution of the SDP problem~\eqref{prob:SQP3Shor}.
This result is crucial to allow the use of arbitrary SDP solvers for the computation of the DNN bound, including in particular those solvers with medium accuracy specially designed for large-scale SDPs.

We demonstrate the effectiveness of our approach through  extensive computational experiments on both real and synthetic instances.
We compare the performance of our algorithms with two of the most powerful commercial solvers \texttt{CPLEX} and \texttt{Gurobi} and an academic open-source software \texttt{quadprogIP} \cite{xia2020globally}.
Our algorithm successfully solves the earlier-mentioned 317,584 instances of dimension {$k=78$} within \textbf{3 days} using 32 processors in parallel.
For the three instances of dimension {$k=819$}, our algorithm is able to solve them in \textbf{a few minutes} on a standard laptop.
On a set of randomly generated instances of dimension 100 to 500, our algorithm also exhibits superior performance compared with other solvers.

The paper is organized as follows.
In~\Cref{sec:DNN}, we study the properties of the DNN relaxation~\eqref{prob:dnnrelax}, show its connection with {the Shor} relaxation, and discuss the computation of valid bounds from inexact SDP solutions.
In~\Cref{sec:cuts}, we review the cutting plane methods for the reference value problem, propose new cuts based on {the} DNN relaxation, and show its connection with Konno's cut.
In~\Cref{sec:cuttingplane} we describe our algorithm in detail.
We report numerical results in~\Cref{sec:experiments}.
In~\Cref{sec:conclusion} we conclude.
For clarity of the presentation, some details are moved into the Appendix.

\subsection{Notations}\label{subsec:notations}

Throughout the paper we let $[n]:=\{1,\ldots,n\}$.
We use $\bf{0}$ to denote {the} zero vector or matrix of appropriate dimension, $\bf{1}$ to denote {the} vector of all ones,  and  $\bf{I}$ to denote the identity matrix of appropriate dimension.
For any vector $x \in \R^n$ and matrix $X \in \R^{m \times n}$, we use $x_i\in \R$ to denote the $i$-th entry of $x$,  $X_{i, j}\in \R$ to denote the $(i, j)$-th entry of $X$ and $X_i\in \R^m$ to  denote the $i$-th column vector of $X$.
We adopt a MATLAB-like notation to represent the submatrix of $X\in \R^{m\times n}$, for example $X_{1:k, n}$ is the submatrix of $X$ formed by the elements at the intersection of  the first $k$ rows and the $n$-th  column of $X$.

For any vector, $x\geq 0$ (resp. $x>0$) means that $x$ is nonnegative (resp. positive), i.e. all the elements in $x$ are nonnegative (resp. positive).
For any two vectors $x$ and $y$ of the same dimension, $x\geq y$ means that $x-y\geq 0$.

For any matrix $X$,  $X\geq 0$ (resp. $X>0$) means that $X$ is nonnegative (resp. positive), i.e. all the elements in $X$ are nonnegative (resp. positive), and $X\succeq 0$ (resp. $X\succ 0$) means that $X$ is positive semidefinite (resp. positive definite).
We denote by $\cS^n$ the set of $n$-by-$n$ symmetric matrices.
For two matrices $X, Y\in \cS^n$, $X \succeq Y$ means that $X - Y {\succeq 0}$, and $X \geq Y$ means that  $X - Y { \geq 0}$.
We use $\mathrm{diag}(h)$ to represent the diagonal matrix with the diagonal vector being $h$.
We consider the Frobenius inner product $\left\langle \cdot, \cdot \right\rangle$ in the space of matrices, i.e.,  $\left\langle A, B \right\rangle := \mathrm{tr}(A^\top B)$.

\section{Doubly Nonnegative Relaxation}\label{sec:DNN}

The  DNN relaxation~\eqref{prob:dnnrelax}  provides an upper bound for the optimal value of~\eqref{prob:original_max_qp0}.
However, it is  known from~\cite[Prop. 8.3]{Burer2012} that the feasible region of the SDP problem~\eqref{prob:dnnrelax} has no interior.
The lack of interior points can be a serious defect from both theoretical and computational aspects.
In particular, for problems with no interior feasible point, the strong duality may not hold and convergence of  existing SDP solvers is not always guaranteed; see~\cite{Pataki19} for more discussion.

To resolve this issue, we propose to consider the standard Shor relaxation of the equivalent QCQP problem (see~\eqref{prob:SQP3Shor} below).
We prove the strong duality for the SDP problem obtained from the Shor relaxation and its equivalence with the SDP problem~\eqref{prob:dnnrelax} obtained from the DNN relaxation.
We also establish a formula for computing a valid upper bound of $\Phi^*(\cF)$ from an approximate solution returned by any arbitrary SDP solver, which is crucial to adapt the algorithm to a wide range of inexact SDP solvers.

\subsection{Equivalence of the DNN relaxation with the Shor relaxation}
Adding redundant constraints to~\eqref{prob:original_max_qp0},  we obtain the following {equivalent} QCQP problem:
\begin{equation}
    \label{prob:SQP3}
    \begin{aligned}
         \max_{y\in {{\R^k}}} \enspace &
        y^\top Q y + 2 d^\top y +\nu \\
        \textrm{s.t.}
        \quad & F^\top y\leq w \\ \quad & y\geq 0 \\
        \quad & (F^\top y-w) (F^\top y-w)^\top \geq 0\\
        \quad & y y^\top  \geq 0 \\
        \quad & (F^\top y-w)y^\top \leq 0.
    \end{aligned}
\end{equation}
It is easy to see that the optimal value of~\eqref{prob:SQP3} is equal to $\Phi^*(\cF)$.
The standard Shor relaxation of the QCQP problem~\eqref{prob:SQP3} yields the following SDP problem:
\begin{equation}
        \label{prob:SQP3Shor}
        \begin{aligned}
            \max_{{Y\in {\cS^{k}}, y\in {\R^{k}}}} \enspace &
            \<Q, Y> + 2 d^\top y  +\nu  \\
            \textrm{s.t.}
            \quad & F^\top y\leq w \\ \quad & y\geq 0 \\
            \quad & F^\top Y F-w y^\top F-F^\top y w^\top +ww^\top \geq 0\\
            \quad & Y \geq 0
            \\ \quad & w y^\top -F^\top Y \geq 0\\
            \quad &\begin{pmatrix}
                Y & y \\ y^\top & 1
            \end{pmatrix}  \succeq 0 .
        \end{aligned}
    \end{equation}
\begin{lemma}
    \label{lemma:equiv-shor}
    The DNN relaxation of the program~\eqref{prob:original_max_qps}, as defined by~\eqref{prob:dnnrelax}, is equivalent to the Shor relaxation of the  QCQP problem~\eqref{prob:SQP3}, as defined by~\eqref{prob:SQP3Shor}. In particular, the optimal value of~\eqref{prob:SQP3Shor} is equal to $\bar \Phi(\cF)$.
\end{lemma}
\begin{proof}
    Let $(Y,y)$  be a feasible solution to~\eqref{prob:SQP3Shor}.
    Let {$X\in \cS^{n}$ and $x\in \R^n$} such that
    $$
     {X:=\begin{pmatrix}
        Y &  y w^\top - Y F  \\
        w y^\top - F^\top Y & \enspace F^\top Y F  -F^\top y w^\top -w y^\top F  +w w^\top
    \end{pmatrix}, \enspace x:=\begin{pmatrix}
    y \\
    w-F^\top y
    \end{pmatrix}.}
    $$
   Since $(Y,y)$ satisfies all the constraints of~\eqref{prob:SQP3Shor}, we have $X\geq 0$ and $x\geq 0$.
    It can be directly checked  that  $$ \begin{pmatrix}{F^\top} & \enspace \bI \end{pmatrix}x=w, $$
    $$
    \begin{pmatrix}
        F^\top &  \bI
    \end{pmatrix}  X  \begin{pmatrix}
        F  \\  \bI
    \end{pmatrix} = w w ^\top,
    $$
 and
    $$
    X=\begin{pmatrix}
        y \\ w-F^\top y
    \end{pmatrix} \begin{pmatrix}
        y \\ w-F^\top y
    \end{pmatrix}^\top + \begin{pmatrix}
        \bI \\ -F^\top
    \end{pmatrix} \left(Y- y y^\top \right) \begin{pmatrix}
        \bI \\ -F^\top
    \end{pmatrix} ^\top \succeq 0.
    $$
    Therefore, $(X,x)$ is a feasible solution of~\eqref{prob:dnnrelax}
    and
    \begin{align}\label{a:sbdfs} \langle \begin{pmatrix}
            Q & \bzero  \\
            {\bzero^\top} & 0
        \end{pmatrix},X\rangle+ 2 \begin{pmatrix} d^\top & {\bzero^\top}\end{pmatrix} x+\nu=\<Q, Y>+2 d^\top y+\nu.\end{align}
    Now let {$(X,x)$} be any feasible solution of~\eqref{prob:dnnrelax} such that
    $${
    \begin{pmatrix} X & x \\ x^\top & 1
   \end{pmatrix}
   }=\sum_{\ell=1}^r  \alpha_{\ell} \begin{pmatrix}{r^{\ell}}  \\ s^{\ell} \\ 1\end{pmatrix} \begin{pmatrix}{r^{\ell}} \\ s^{\ell}  \\ 1 \end{pmatrix}^\top+  \sum_{\ell=r+1}^t  \begin{pmatrix}{r^{\ell}}  \\ s^{\ell} \\ 0\end{pmatrix} \begin{pmatrix}{r^{\ell}} \\ s^{\ell} \\ 0 \end{pmatrix}^\top,
    $$
    where $\alpha_1,\ldots \alpha_r\geq 0$ and
    $\sum_{\ell=1}^r \alpha_{\ell} = 1$.
    By \cite[Prop. 8.3]{Burer2012}, we know that
    $$
    s^{\ell} = \left\{ \begin{array}{ll}w-F^\top {r^{\ell}}, &\enspace \forall \ell\in [r],  \\
        -F^\top {r^{\ell}} ,& \enspace \forall \ell\in \{r+1,\ldots, t\}.
    \end{array}\right.
    $$
    Let
    $$
    Y=\sum_{\ell=1}^r \alpha_{\ell} {r^{\ell}} ({r^{\ell}})^\top +\sum_{\ell=r+1}^t {r^{\ell}} ({r^{\ell}})^\top,\enspace y=\sum_{\ell=1}^r \alpha_{\ell} {r^{\ell}}.
    $$
    Then
    $$
    \begin{aligned}
        {
    \begin{pmatrix} X & x \\ x^\top & 1
   \end{pmatrix}
   }&=\sum_{\ell=1}^r  \alpha_{\ell}\begin{pmatrix}{r^{\ell}}  \\ w-F^\top {r^{\ell}} \\ 1\end{pmatrix} \begin{pmatrix}{r^{\ell}} \\ w-F^\top {r^{\ell}} \\ 1 \end{pmatrix}^\top+  \sum_{\ell=r+1}^t  \begin{pmatrix}{r^{\ell}}  \\ -F^\top {r^{\ell}} \\ 0\end{pmatrix} \begin{pmatrix}{r^{\ell}} \\ -F^\top {r^{\ell}} \\ 0 \end{pmatrix}^\top\\
        &=\begin{pmatrix}
            Y &  y w^\top - Y F & y \\
            w y^\top - F^\top Y & \enspace F^\top Y F  -F^\top y w^\top -w y^\top F  +w w^\top & w-F^\top y  \\
            y^\top & w^\top -y^\top F & 1
        \end{pmatrix}.
    \end{aligned}
    $$
    And so $(Y, y)$ is a feasible solution to~\eqref{prob:SQP3Shor} satisfying~\eqref{a:sbdfs}.
    \qed
\end{proof}
\begin{remark}
The proof of~\Cref{lemma:equiv-shor} shows how to construct an optimal solution of~\eqref{prob:SQP3Shor} from an optimal solution of~\eqref{prob:dnnrelax}, and vice versa.
\end{remark}

\begin{remark}
    The size of the largest SDP matrix in~\eqref{prob:dnnrelax}  is $n+1$, while the size of the largest SDP matrix in~\eqref{prob:SQP3Shor} is  {$k+1$}.
    This already suggests that~\eqref{prob:SQP3Shor} could be computationally more favorable than~\eqref{prob:dnnrelax}.
    Moreover, we will show in the next subsection that {the} strong duality holds for~\eqref{prob:SQP3Shor}.
\end{remark}

\subsection{Strong duality}
We consider the SDP relaxation  in the form of~\eqref{prob:SQP3Shor}.
Let us express~\eqref{prob:SQP3Shor} in the following abstract way:
\begin{equation}
    \label{prob:SQP3Shorabs}
    \begin{aligned}
        \max_{\substack{\hat Y\in {\cS^{k+1}}}} \enspace &
        \< \hat Q,  \hat Y>  \\
        \textrm{s.t.}
        \quad &  \<W^{(i,j)}, \hat Y>\leq 0,\enspace 0\leq i<j\leq {k+m}\\
        \quad & \<W^0, \hat Y>=1\\
        \quad &\hat Y  \succeq 0.
    \end{aligned}
\end{equation}
Here,
$$ \hat Q:= \begin{pmatrix}
    Q & d \\
    d^\top & \nu
\end{pmatrix},\enspace \enspace
W^0 := \begin{pmatrix}
    \bzero & \bzero \\
    \bzero^\top & 1
\end{pmatrix},
$$
and $\{W^{(i,j)}: 0\leq i<j \leq k+m\}$ are matrices for representing the polyhedral constraints in~\eqref{prob:SQP3Shor}:
\small
\begin{equation*}
    W^{(i,j)} :=
    \left\{ \begin{array}{ll}
        \begin{pmatrix} \bzero & -e_{j} \\ -e_{j}^\top & 0
        \end{pmatrix}\enspace &  i=0,\enspace 1\leq j\leq k \\
        \begin{pmatrix} \bzero & F_{j-k} \\ F_{j-k}^\top & - 2w_{j-k} \end{pmatrix} \enspace & i=0, \enspace   k < j \leq k+m  \\
        \begin{pmatrix}- e_i  e^\top_j-  e_j e_i^\top & \bzero \\ \bzero^\top & 0 \end{pmatrix}\enspace &  1\leq i<j\leq  k\enspace
        \\ \begin{pmatrix}
            F_{j-k} e_i^\top+e_i F_{j-k}^\top  & -w_{j-k} e_i\\ -w_{j-k} e_i^\top & 0
        \end{pmatrix} \enspace &  1\leq i\leq  k< j\leq m+k\\
        \begin{pmatrix}
            -F_{j-k} F^\top_{i-k}- F_{i-k}F^\top_{j-k} & w_{i-k} F_{j-k}+ w_{j-k}  F_{i-k}  \\
            w_{i-k}  F^\top_{j-k}+ w_{j-k}  F^\top_{i-k}  & -2 w_{i-k} w_{j-k}
        \end{pmatrix}\enspace &  k+1\leq  i<j \leq  m+k.
    \end{array}\right.
\end{equation*}
\normalsize
Here, $e_1,\ldots, e_{k}$ are the standard basis vectors of $\R^{{k}}$.
The dual of~\eqref{prob:SQP3Shor} can thus be written as:
\begin{equation}
    \label{prob:SQP3Shorabsdual}
    \begin{aligned}
        \min \enspace &
        {\lambda_0}  \\
        \textrm{s.t.}
        \enspace &  \hat Q+\sum_{i<j} \lambda_{i,j} W^{(i,j)}-{\lambda_0} W^0 \preceq 0\\
        \quad  &\lambda_{i,j}\leq 0,\enspace  0\leq i<j\leq {k+m}.
    \end{aligned}
\end{equation}

Now we establish the strong duality between \eqref{prob:SQP3Shorabs} and its dual \eqref{prob:SQP3Shorabsdual} under the following assumption.
\begin{ass}
\label{ass:bound_and_interior}
    $\cF$ is bounded and has  nonempty interior.
\end{ass}
\begin{lemma}\label{l:psdf}
    {If $\cF$ has nonempty interior}, there exists a solution $(Y, y)$ strictly feasible to~\eqref{prob:SQP3Shor}.
\end{lemma}
\begin{proof}
    Let $y^0\in \R^{{k}}$ such that
    $F^\top y^0 < w$ and $y^0 >0$.  Let $\epsilon>0$ and $y^i=y^0+\epsilon e_i$ for each $i\in [{k}]$.
    Choose $\epsilon>0$  sufficiently small so that
    $\{y^{1},\ldots, y^{{k}}\}$ are all in the set $\{y\in \R^{{k}}: F^\top y \leq w, y\geq 0\}$.
    Let
    $$
    Y^i =y^i (y^i)^\top,\enspace \forall i\in \{0\} \cup [{k}].
    $$
    Then for each $i$, $(Y^i, y^i)$ is a feasible solution to~\eqref{prob:SQP3Shor}. Let
    $$
    Y=\frac{1}{{k}+1}\sum_{i=0}^{{k}} Y^i,\enspace y= \frac{1}{{k}+1}\sum_{i=0}^{{k}}  y^i.
    $$
 Clearly $(Y, y)$ is  feasible to~\eqref{prob:SQP3Shor}.
 Since the vectors $
    \left\{\begin{pmatrix}y^0 \\1\end{pmatrix},\ldots, \begin{pmatrix}y^{k} \\1\end{pmatrix}\right\}
    $ are linearly independent, the rank of the matrix $\begin{pmatrix}
                Y & y \\ y^\top & 1
            \end{pmatrix}$ is ${k}+1$. It follows that $(Y, y)$ is strictly feasible to~\eqref{prob:SQP3Shor}.
    \qed
\end{proof}
\begin{lemma}
    \label{l:wsdgseg}
   {If $\cF$ is bounded and there is at least one $y\in \cF$ such that $y>0$}, the dual SDP problem~\eqref{prob:SQP3Shorabsdual} is strictly feasible.
\end{lemma}
\begin{proof}
    The conditions imply the existence of some {$t_*> 0$} such that
    \begin{equation}\label{eq:tstar}
        \begin{aligned}
            t_*= \max_{y\in \R^{{k}}} \enspace &
            \bone^\top y  \\
            \mathrm{s.t.}
            \quad & F^\top y\leq w  \\
            \quad& y\geq 0 .
        \end{aligned}
    \end{equation}
     By {the} strong duality of \eqref{eq:tstar}, there is $a\in \R_{+}^{m}$ such that
    $$
    {h}:=Fa\geq  \bone,\enspace a^\top w =t_*.
    $$
We have
    $$
    \begin{aligned}
        &	\sum_{i=1}^{{k}}	\sum_{j={k}+1}^{{{k+m}}} a_{j-k} W^{(i,j)} \\
        &=\sum_{i=1}^{{k}}	\sum_{j=1}^{m} a_j \begin{pmatrix}
            F_j e_i^\top+e_i  F_j^\top  & - w_j e_i\\ -  w_j e_i^\top & 0
        \end{pmatrix}
        \\ & =  \sum_{i=1}^{{k}}  \begin{pmatrix}
            {h} e_i^\top+e_i {h}^\top  & - t_* e_i\\ -t_* e_i^\top & 0
        \end{pmatrix}
        \\&
        = \begin{pmatrix}
            {h} \bone^\top+ \bone {h}^\top  &   -t_*\bone \\
            -t_*\bone^\top   & 0
        \end{pmatrix}.
    \end{aligned}
    $$
    Thus
    $$
    \begin{aligned}
        &\sum_{i=1}^{{k}}	\sum_{j=k+1}^{{k+m}} a_{j-k} W^{(i,j)}
        +\sum_{1\leq i<j\leq {k}} ({h}_{i}+{h}_{j}) W^{(i,j)}\\
        &=\begin{pmatrix}
            {h} \bone^\top+ \bone {h}^\top  &   -t_*\bone \\
            -t_*\bone^\top   & 0
        \end{pmatrix} + \sum_{1\leq i<j\leq {k}} ({h}_{i}+{h}_{j})\begin{pmatrix}-e_i e^\top_j-e_j e_i^\top & \bzero \\ \bzero^\top & 0 \end{pmatrix}\\
        &=\begin{pmatrix}
            2\diag({h})   &   -t_*\bone \\
            -t_*\bone^\top   & 0
        \end{pmatrix}.
        \end{aligned}
    $$
    It follows that
    $$
    \begin{aligned}
        &\sum_{i=1}^{{k}}	\sum_{j=k+1}^{{k+m}} a_{j-k} W^{(i,j)}
        +\sum_{1\leq i<j\leq {k}} ({h}_{i}+{h}_{j}) W^{(i,j)}
        + {k} t_*^2 W^0
        \\
        &=\begin{pmatrix}
            2\diag({h})   &   -t_*\bone \\
            -t_*\bone^\top   & {k} t_*^2
        \end{pmatrix} .
        \end{aligned}
    $$
    Since
    \[(-t_* \bone)^\top (2\diag({h}))^{-1} (-t_* \bone)=\frac{1}{2} t_*^2 \left( \sum_{i=1}^{{k}} {h}^{-1}_i \right) \leq \frac{ {k}  t_*^2 }{2},\]
    {and $t_*>0$}, we have
    $$
    \sum_{i=1}^{{k}}	\sum_{j={k}+1}^{{k+m}} a_{j-k} W^{(i,j)}
        +\sum_{1\leq i<j\leq {k}} ({h}_{i}+{h}_{j}) W^{(i,j)}
        + {k} t_*^2 W^0 \succ 0.
    $$
    Therefore, if we take
    \begin{align*}
        & \lambda_{i, j} = -{\rho} a_{j-k} , \quad 1 \leq i \leq k < j \leq {k+m}, \\
        & \lambda_{i, j} = -{\rho} ({h}_i + {h}_j), \quad 1 \leq i < j \leq {k}, \\
        & {\lambda_0 = \rho k t_*^2},
    \end{align*}
    for sufficiently large ${\rho>0}$, and take other $\lambda_{i, j}$ (where $i = 0$ or $i > k$) to be 0, we see that there exist $\{\lambda_{i,j}: 0\leq i<j\leq n\} \subset \R_{-}$ and ${\lambda_0} \in \R $ such that
    $$
    \hat Q+\sum_{i<j} \lambda_{i,j} W^{(i,j)}-{\lambda_0} W^0  \prec 0.
    $$
    \qed
\end{proof}

By Lemma~\ref{l:psdf} and~\ref{l:wsdgseg}, we establish Slater's condition and it straightforwardly {implies the} strong duality for the SDP relaxation~\eqref{prob:SQP3Shor}.
\begin{proposition}
    \label{prop:sd}
    Under Assumption~\ref{ass:bound_and_interior}, both
    the SDP problem~\eqref{prob:SQP3Shor} and its dual SDP problem~\eqref{prob:SQP3Shorabsdual}   have an optimal solution and the strong duality holds.
\end{proposition}

As a result, convergence of state-of-the-art SDP algorithms, such as SDPT3 \cite{TohToddTutuSDPT3} and SDPNAL+~\cite{sun2020sdpnal+} which require the existence of an interior point in the feasible region, is guaranteed.

\subsection{Valid DNN bound from inexact SDP solution}\label{sec:vdnnbi}
Hereinafter, we will assume that Assumption~\ref{ass:bound_and_interior} holds. By~\Cref{lemma:equiv-shor}, in order to compute the DNN bound $\bar \Phi(\cF)$, we can solve the SDP problem~\eqref{prob:SQP3Shor}. With the strong duality property of~\eqref{prob:SQP3Shor}, we rest assured for the convergence of a wide range of SDP solvers.
We mention in particular the interior point type methods and the augmented Lagrangian type methods.
There are many powerful {commercial and academic} solvers that implement these methods: {\texttt{MOSEK}, \texttt{SDPT3}{~\cite{TohToddTutuSDPT3}}, \texttt{SDPNAL+}{~\cite{sun2020sdpnal+}}}, etc.

However, it is important to note that any numerical solver using finite floating-point arithmetic is incapable of returning  the exact value of $\bar \Phi (\cF)$.
Indeed, such numerical solvers terminate when a solution with sufficiently small {constraint violation} and primal dual optimality gap is found.
The value returned by the numerical SDP solver is only an approximation of the true value $\bar\Phi(\cF)$.
Additional care needs to be taken in order to get a valid upper bound of $\Phi^*(\cF)$.

To be more precise, consider applying any suitable SDP numerical solver to the primal SDP problem~\eqref{prob:SQP3Shorabs} and its dual SDP problem~\eqref{prob:SQP3Shorabsdual}. Let $\epsilon>0$ and assume that the pair of primal dual solution ${(\hat Y;  \lambda)}$ returned by the SDP solver satisfies the following {constraint violation} and  duality  gap condition{s}
\begin{equation}\label{eq:pdgapep}
    \left\{
    \begin{array}{ll}
        \<W^{(i,j)}, \hat Y>\leq \epsilon,  \enspace 0\leq i<j\leq {k+m} \\
        \<W^0, \hat Y>=1 \\
        \hat Y \succeq 0 \\
        \hat Q+\sum_{i<j} \lambda_{i,j} W^{(i,j)}-{\lambda_0} W^0 \preceq \epsilon \bI  \\
        {\lambda_{i,j}\leq 0, \enspace 0\leq i<j\leq {k+m}}\\
        |\<\hat Q,\hat Y>-{\lambda_0}  | \leq \epsilon.
    \end{array}
    \right.
\end{equation}
As long as $\epsilon>0$, there is no guarantee that the approximate primal optimal value $\<\hat Q,\hat Y>$ or the approximate dual optimal value {$\lambda_0$} is larger than $ \Phi^*(\cF)$.
We next show how to obtain a valid upper
bound of $\Phi^*(\cF)$ from $(\hat Y;  \lambda)$ satisfying~\eqref{eq:pdgapep}.
\begin{lemma}\label{l:oplfds}
    For any feasible solution $\hat Y$ of~\eqref{prob:SQP3Shorabs}, we have
    \begin{align}\label{eq:oplfds}
    \trace(\hat Y) \leq 1+t^2_*,
    \end{align}
    where $t_* > 0$ is the constant defined in~\eqref{eq:tstar}.
\end{lemma}
\begin{proof}
    Let
    $$
    \hat Y=\begin{pmatrix}
        Y & y \\ y^\top & 1
    \end{pmatrix}
    $$
    be a feasible solution of~\eqref{prob:SQP3Shorabs}.
    In the proof of~\Cref{l:wsdgseg}, we have shown that
    $$
    \begin{aligned} &
        \sum_{i=1}^{{k}}	\sum_{j={k}+1}^{{k+m}} a_{j-k} W^{(i,j)}
        &= \begin{pmatrix}
            {h} \bone^\top+ \bone {h}^\top  &   -t_*\bone \\
            -t_*\bone^\top   & 0
        \end{pmatrix}
    \end{aligned},
    $$
    where $a_{j-{k}}\geq 0$ for all $j\in\{{k}+1,\ldots,{k+m}\}$ and ${h}\geq \bone$. Taking inner product with $\hat Y$ on both sides, we obtain
    $$ \begin{aligned} 
        \sum_{i=1}^{{k}}	\sum_{j={k}+1}^{{k+m}} a_{j-k} \langle W^{(i,j)}, \hat Y\rangle
        =\langle \begin{pmatrix}
            {h} \bone^\top+ \bone {h}^\top  &   -t_*\bone \\
            -t_*\bone^\top   & 0
        \end{pmatrix}, \hat Y \rangle=2\left({h}^\top Y \bone-t_* \bone ^\top y\right).
    \end{aligned}
    $$
    Therefore, since $\langle W^{(i, j)}, \hat{Y} \rangle \leq 0$ {and $a_{j-k}\geq 0$}, we have
    $$
     {h}^\top Y \bone \leq t_* \bone^\top y \leq t^2_*.
    $$
    Since  ${h}\geq \bone$, we have
    $\bI\leq \bone {h}^\top$. Finally, since $Y\geq 0$,  it follows that $$\trace(Y)=\langle Y, \bI \rangle \leq \langle Y, \bone {h}^\top \rangle ={h}^\top Y\bone\leq  t_*^2.$$
    \qed
\end{proof}

\begin{proposition}\label{prop:inexcs}
    Let $\epsilon>0$ and let ${(\hat Y ; \lambda)}$ be a pair of  approximate  primal and dual solutions satisfying the gap conditions in~\eqref{eq:pdgapep}. Then,
    \begin{align}\label{a:inexcs}
    \bar\Phi(\cF) \leq \epsilon\left(t^2_*+1\right)+{\lambda_0}.
    \end{align}
\end{proposition}
\begin{proof}
    Let $(\hat Y^*, \lambda^*)$ be an optimal primal dual solution pair. Then,
    \begin{equation*}
        \begin{aligned}
            \bar \Phi(\cF)&=\<\hat Q, \hat Y^*> \\
            &=\<\hat Q+\sum_{i<j} \lambda_{i,j} W^{(i,j)}-{\lambda_0} W^0, \hat Y^*>-\sum_{i<j} \lambda_{i,j}\<W^{(i,j)}, \hat Y^*>+{\lambda_0} \<W^0, \hat Y^*> \\
            & \overset{\eqref{eq:pdgapep}}{\leq} \epsilon \trace(\hat Y^*)+{\lambda_0}
            \\
            & \overset{\eqref{eq:oplfds}}{\leq} \epsilon \left(1+t_*^2\right)+{\lambda_0}.
        \end{aligned}
    \end{equation*}
    {Here the first equality follows from the fact that $\bar \Phi(\cF)$ is the optimal value of~\eqref{prob:SQP3Shorabs}. The second equality can be verified directly. The third inequality follows from the conditions in~\eqref{eq:pdgapep} and the positive semidefiniteness of $\hat Y^*$. The last inequality follows from~\eqref{eq:oplfds}.}
    \qed
\end{proof}
\Cref{prop:inexcs} enables us to obtain a valid upper bound of $\Phi^*(\cF)$ from any approximate primal-dual solution ${(\hat Y;\lambda)}$ returned by any SDP solver. The correction term to be added is $\epsilon(t^2_*+1)$, where $\epsilon$ depends on the constraint violation and the duality gap, and $t_*$ is defined in~\eqref{eq:tstar}.

\begin{remark}\label{rem:tstar}
Note that the exact value of $t_*$, which is the optimal value of the LP problem~\eqref{eq:tstar},  may not be computable (due to numerical tolerances and finite machine precision).  However, formula~\eqref{a:inexcs} remains valid if one replaces $t_*$ by any number larger than it.  An explicit upper bound is readily available  if $F$ has at least one column with positive elements only. Indeed, if $F_{1,1},\ldots,F_{k,1}$ are all positive, then
$$
        \begin{aligned}
            t_*\leq  \max_{y\in \R^{{k}}} \enspace &
            \bone^\top y  \\
            \mathrm{s.t.}
            \quad & F_{1,1}y_1+\cdots+F_{k,1}y_k\leq w_1 \\
            \quad& y_1,\ldots,y_k\geq 0 .
        \end{aligned}
$$
The optimal value of the above LP problem is known to be $ {w_1}\left({\min(F_{1,1},\cdots,F_{k,1})}\right)^{-1}$.
\end{remark}

\subsection{Upper bound of the DNN bound}

In this subsection, we establish an upper bound for the DNN bound $\bar\Phi(\cF)$. This result will be useful in the later section when we {compare different cuts}.

\begin{proposition}
    \label{prop:sfsdfsdf}
    Let $\Lambda \in \R_{+}^{{k}\times m}$  and $P\in \R_+^{{k}\times {k}}\cap \cS^{{k}}$  be  nonnegative   matrices  such that
    \begin{align}
        \label{a:cDe}
        -Q-P+\Lambda F^\top+ F\Lambda^\top  \succeq 0.
    \end{align}
    Then,
    $$
    \bar\Phi\left(\cF\right)  \leq  {\nu} + \max \left\{  2\left(d+ \Lambda w \right) ^\top y: F^\top y \leq w, y\geq 0 \right\}.
    $$
\end{proposition}

\begin{proof}
Let $(Y,y)$ be a feasible solution of{~\eqref{prob:SQP3Shor}}.
Since $P\geq 0$ and $Y\geq 0$, we have
\begin{align}\label{a:PY}
\<-P,Y>\leq 0.
\end{align}
Since
\begin{equation}\label{eq:Lambdageq}
\Lambda\geq 0,\enspace wy^\top -F^\top Y\geq 0,
\end{equation}
we have
\begin{align}\label{a:piwer}
\<\Lambda F^\top+ F\Lambda^\top,Y>=2\langle \Lambda, F^\top Y\rangle \leq 2 \langle\Lambda ,wy^\top\rangle=2\left(\Lambda w\right)^\top y.
\end{align}
Here, the inequality follows from~\eqref{eq:Lambdageq}.
Hence,
\begin{equation}\label{eq:dfsfer}
\begin{aligned}
\<Q,Y>+2 d^\top y & {~\leq~}  \<-P+\Lambda F^\top+ F\Lambda^\top,Y>+2d^\top y \\
& {~\leq~}  2\left(d+ \Lambda w \right) ^\top y .
\end{aligned}
\end{equation}
Here, the first inequality follows from~\eqref{a:cDe} and the fact that $Y\succeq 0$, and the second inequality follows from~\eqref{a:PY} and~\eqref{a:piwer}.
If $(Y^*,y^*)$ is an optimal solution of~\eqref{prob:dnnrelax} (the existence follows from~\Cref{prop:sd}), then
\begin{align*}
\bar \Phi(\cF)&=\<Q,Y^*>+2 d^\top y^* +\nu\\ &\leq {\nu}+2\left(d+ \Lambda w \right) ^\top y^*  \\& \leq   {\nu}+ \max \left\{  2\left(d+ \Lambda w \right) ^\top y: F^\top y \leq w, y\geq 0 \right\}.
\end{align*}
Here, the first inequality follows from~\eqref{eq:dfsfer}, and the second inequality uses the fact that $y^*$ satisfies $F^\top y^*\leq w$ and $y^*\geq 0$.
\qed
\end{proof}

\section{Generation of Valid Cuts}\label{sec:cuts}

In this section we review the idea of cutting plane algorithms and present classical as well as novel methods for the generation of valid cuts at any given KKT {vertex}.

Recall that we have, at our disposal, a reference value $\nu_R$ and the task is to determine whether
\begin{align}
    \label{a:possddf}
    \nu_R > \Phi^*(\cF) \enspace \mathrm{~or~~~}  \Phi^*(\cF)\geq \nu_R \enspace ?
\end{align}
For ease of presentation, we assume without loss of generality that the KKT vertex of interest is the origin point $\textbf{0}\in\R^k$ and the following three conditions hold:
\begin{subequations}\label{esdf}
\begin{align}
&w\geq 0,\\
&d\leq 0, \\
&\nu <\nu_R.
\end{align}
\end{subequations}
The first condition implies that $\textbf{0}$ is a vertex. The second condition implies that $\textbf{0}$ is a KKT point. The third condition means that the objective value at $\textbf{0}$ is strictly smaller than the reference value $\nu_R$.
 For details on reducing the general case to the case satisfying~\eqref{esdf}, see Appendix~\ref{sec:mp}.

\subsection{Cutting plane method}

Under the three conditions in~\eqref{esdf},
    in order to generate a valid cutting plane at the origin point $\textbf{0}\in \R^k$, we search for a nonnegative vector $\theta\in \R^{k}_+$ such that
\begin{equation}
    \label{eq:caQPcut}
    \begin{aligned}
   \nu_R> \max_{y\in \R^{k}} \quad &   y^T Q y + 2 d^\top y + \nu  \\
        \textrm{s.t.} \quad &  F^\top y\leq w\\&   y\geq 0\\
        & \theta^\top y\leq 1.
    \end{aligned}
\end{equation}
If~\eqref{eq:caQPcut} holds, then the cut {$\{y\in \R^{k}: \theta^\top y \geq 1\}$} is said to be \textit{valid} and
$$
 \Phi^*({\cF_{\theta}})\leq
\Phi^*(\cF) \leq \max \left(\nu_R,   \Phi^*({\cF_{\theta}})\right),
$$
where ${\cF_{\theta}}$ is defined as the intersection of the original feasible region $\cF$ and the cut {$\{y\in \R^{k}: \theta^\top y \geq 1\}$}:
$$
\cF_{\theta}:= \cF\cap {\{y\in \R^{k}: \theta^\top y \geq 1\}}.
$$
It is important to note that ${\textbf{0}}\notin {\cF_{\theta}}$ and hence ${\cF_{\theta}}$ is strictly included in $\cF$.
In order to determine~\eqref{a:possddf}, it suffices to determine whether
\begin{align}\label{a:possddf2}
    \nu_R >   \Phi^*({\cF_{\theta}}) \enspace \mathrm{~or~~~}   \Phi^*({\cF_{\theta}})\geq \nu_R \enspace ?
\end{align}
Thus the problem is reduced with a strictly smaller feasible region ${\cF_{\theta}}$. Repeating this procedure of adding valid cuts, the feasible region is reduced at each iteration until we find an answer to~\eqref{a:possddf2}.

In the following, we present different methods to find a nonnegative vector $\theta$ satisfying~\eqref{eq:caQPcut}.

\subsection{Tuy's and Konno's cut revisited}\label{subsec:konno}
We  give a brief review of Tuy and Konno's work \cite{konno1976cutting,konno1976maximization,Tuy1964concave}.
For ease of discussion, we denote for any $\theta\in \R_+^k$ that
$$
\Delta(\theta):=\left \{y\in \R^{k}: \sum_{i=1}^{k} y_i \theta_i \leq 1, y\geq 0\right\}.
$$
We start with a basic result which is essential in both Tuy and Konno's argument.
\begin{lemma}\label{l:dw0}
Let $f:\R^{k}\rightarrow \R$ be a convex function and $\theta>0$.
Then,
$$
\max \left\{f(y): y\in \Delta(\theta)\right\}=\max \left\{f\left(\bf{0}\right),f(\theta^{-1}_1 e_1),\ldots,f(\theta^{-1}_ke_k)\right\}.
$$
\end{lemma}
\begin{proof}
When $\theta>0$, $\Delta(\theta)$ is a simplex with vertices $\left\{{\bf{0}}, \theta^{-1}_1 e_1,\ldots,\theta^{-1}_k e_k\right\}$. The statement follows from the fact that  the maximum value of $f$ is attained at  one vertex of $\Delta(\theta)$.
\qed
\end{proof}

When $\theta\geq 0$ contains zero elements, $\Delta(\theta)$ is not a simplex and it is not enough to directly examine the function value at the vertices of $\Delta(\theta)$.
To handle this issue, we present below a generalization of~\Cref{l:dw0} which covers all the case{s} of $\theta\geq 0$.
\begin{lemma}\label{l:dw}
Let $f:\R^{k}\rightarrow \R$ be a convex function and $\gamma \geq f(\bf{0})$. Assume that for any $i\in [k]$, there is $\eta>0$ such that $f(\eta^{-1}e_i)\leq \gamma$. 
Define the vector $\theta\in \R_+^k$ as follows:
$$
\theta_i:=\inf\{\eta> 0: f(\eta^{-1} e_i)\leq \gamma \} ,\enspace i=1,\ldots,k.
$$
Then
$$
\max \left\{f(y): y\in \Delta(\theta)\right\}\leq \gamma.
$$
\end{lemma}

\begin{proof}
Let $\rho>0$. For each $i\in [k]$, if $f(\rho^{-1}e_i)> \gamma$, then $f(\eta^{-1} e_i)>\gamma$ for any $0<\eta \leq \rho$. This follows from the convexity of $f$ and $f(\bzero) \leq \gamma$. Hence for any $\rho>\theta_i$ we must have $f(\rho^{-1}e_i)\leq \gamma$.

Fix any $y\in \Delta(\theta)$.
For any $t\geq \max((\sum_{i=1}^k y_i)^2,1)$,  we have $\theta + \frac{1}{t}\mathbf{1} > \mathbf{0}$ and
\begin{equation*}
    \begin{aligned}
        &\sum_{i=1}^k\left(\theta_i + \frac{1}{t}\right)\left(1-\frac{1}{\sqrt{t}}\right)y_i  \\
        & =
        \left(1 - \frac{1}{\sqrt{t}}\right) \left(\sum_{i=1}^k \theta_i y_i + \frac{1}{t}  \sum_{i=1}^k y_i \right)\\
        & \leq \left(1 - \frac{1}{\sqrt{t}}\right) \left(1 + \frac{1}{t}  \sum_{i=1}^k y_i \right) \\
        &\leq
 \left(1 - \frac{1}{\sqrt{t}}\right) \left(1 + \frac{1}{\sqrt{t}}  \right)\leq 1,
    \end{aligned}
\end{equation*}
hence $\left(1 -
\frac{1}{\sqrt{t}} \right) y \in \Delta\left(\theta + \frac{1}{t} \mathbf{1}\right)$.
By~\Cref{l:dw0} we know that for sufficiently large $t>0$:
$$
f\left(\left(1-\frac{1}{\sqrt{t}}\right) y\right) \leq  \max \left( f(\bzero), \max_{i=1,\ldots,k} f\left(\frac{t}{t\theta_i+1}e_i\right )\right)  \leq \gamma.
$$
Hence, $f(y)\leq \gamma$ by the continuity of $f$.
\qed
\end{proof}

Tuy~\cite{Tuy1964concave} proposed to use the following value
\begin{equation}
    \label{a:posgewTuy}
    \begin{aligned}
        {\max  \left\{ \Phi(y) : y\in \Delta(\theta)\right\}  }
    \end{aligned}
\end{equation}
as an upper bound of $\max  \left\{ \Phi(y) :   y\in \cF\cap  \Delta(\theta)\right\}$. Let {$0< \delta<\nu_R-\nu$} and define
\begin{equation}\label{eq:tuytaui}
   {\tau_i:= \inf\left\{\eta > 0: \Phi(\eta^{-1} e_i) \leq \nu_R- {\delta}\right\},\enspace  i=1,\ldots,k.}
\end{equation}
It can be checked that
\begin{align}\label{a:theta}
    \tau_i=\frac{Q_{ii}}{-d_i+\sqrt{d^2_i+Q_{ii}\left( \nu_R-\delta-\nu\right)}},\enspace i=1,\dots, k.
\end{align}
 {Note that $\tau_i$ in~\eqref{a:theta} is well-defined for any $\delta<\nu_R-\nu$. If $Q_{ii}=0$ for some $i$, then $\tau_i=0$.  }
Applying~\Cref{l:dw} with $f=\Phi$ and $\gamma=\nu_R-\delta$, we obtain
$$
\max  \left\{ \Phi(y) :   y\in \cF\cap  \Delta(\tau)\right\} \leq \max  \left\{ \Phi(y) : y\in \Delta(\tau)\right\} \leq \nu_R-\delta.
$$
Thus the cut  $\{y \in \R^{k}: \tau^\top y \geq 1\}$ is a valid cut, known as Tuy's cut.

Next we recall the method proposed by Konno in~\cite{konno1976cutting},
for improving any valid cut including in particular Tuy's cut.  Given Tuy's cut $\tau$ defined in~\eqref{a:theta}, we next  deal with  the following program:
\begin{equation}
    \label{eq:caQPcuttau}
    \begin{aligned}
   \phi^*_{\tau,\theta}:= \max_{y\in \R^{k}} \quad &   y^T Q y + 2 d^\top y + \nu  \\
        \textrm{s.t.} \quad &  F^\top y\leq w\\&   y\geq 0\\
        & \tau^\top y \geq 1\\
        & \theta^\top y\leq 1.
    \end{aligned}
\end{equation}
Recall that  $\cF_\tau$ is the reduced region
$$
\cF_\tau=\cF\cap \{y\in \R^k: \tau^\top y\geq 1\},
$$
which is assumed to be nonempty (otherwise the reference value problem is solved).
Konno~\cite{konno1976cutting} proposed to use
\begin{align}\label{a:barphiko}
    \bar\phi^K_{\tau,\theta}:=
    \max\left\{\tilde y^\top Q  y+ d^\top \tilde y+ d^\top  y+\nu:  y\in \Delta(\theta), \tilde y\in \cF_\tau \right\}
\end{align}
as an upper bound of $\phi^*_{\tau,\theta}$.
The key observation made by Konno is that,
\begin{equation}\label{a:barphiko_g_form}
\bar\phi^K_{\tau,\theta}=\max \left\{ g(y): y\in \Delta(\theta)\right\},
\end{equation}
where $g:\R^k\rightarrow \R$ is a convex function defined by:
$$
g(y):=\max\{\tilde y^\top Q  y+ d^\top \tilde y+ d^\top  y+\nu: \tilde y\in \cF_\tau\}.
$$

Next, we define the vector $\theta_K$ which formulates Konno's cut.
The entries of $\theta_K$ for a given $0< \delta'<\nu_R-\nu$ are defined as\footnote{The perturbation parameter $\delta'$ here may be different from the parameter $\delta$ used in Tuy's cut~\eqref{eq:tuytaui}. }:
 \begin{align}\label{a:deflamdai}
 (\theta_K)_i:=\inf\left\{\eta > 0: g(\eta^{-1} e_i) \leq \nu_R- {\delta'}\right\},\enspace \forall i=1,\ldots,k.
 \end{align}
 Again, applying~\Cref{l:dw} with $f=g$ and $\gamma=\nu_R-\delta'$, together with~\eqref{a:barphiko_g_form} we obtain
 \begin{align}\label{a:ewer}
  \phi^*_{\tau,\theta}\leq \bar\phi^K_{\tau,\theta_K} \leq \nu_R-{\delta'}<\nu_R.
\end{align}
The cut $\{y\in \R^{k}: \theta_K^\top y\geq 1\}$ is a valid cut, which we will refer to as Konno's cut.

\begin{remark}
\label{rem:compute_alpha_as_LP}
Konno~\cite[Thm 3.3]{konno1976maximization} showed that the computation of each $(\theta_K)_i$ can be done by solving
 an LP program; see~\Cref{sec:konnocut} for details.
 \end{remark}

We next present a lemma which allows us to argue that Konno's cut is generally deeper than Tuy's cut for a particular choice of $\delta'$.
\begin{lemma}\label{l:posdf}
Without loss of generality, suppose that
$\tau_1=\cdots=\tau_\ell=0$ and {$\tau_{\ell+1}>0,\ldots,\tau_{k}>0$} for some $\ell\in \{0,\ldots,k\}$.
Define:$$
y^{i}\in \arg\max \{ \left(\tau_i^{-1} Q_i+d\right)^\top y: y\in \cF_\tau\}, \enspace i=\ell+1,\ldots,k,
$$
and
\begin{equation}\label{eq:omega}
\omega=\max \left\{  \Phi(y^{i}): i=\ell+1,\ldots,k\right\}.
\end{equation}
Then
\begin{equation}\label{a:wedfrre}
(\theta_K)_i=0,\enspace i=1,\ldots,\ell,
\end{equation}
and
\begin{equation}\label{eq:sewee}
g(\tau_i^{-1}e_i)\leq \frac{\nu_R+\omega}{2},\enspace i=\ell+1,\ldots,k.
\end{equation}
\end{lemma}
We defer the proof to~\Cref{sec:proofof}. Let us  elaborate on the utilization of~\Cref{l:posdf}.
If $\omega \geq \nu_R$, then the reference value problem~\eqref{a:possddf} is solved. Otherwise, as long as
\begin{align}\label{a:deltap}
\delta' \leq \frac{\nu_R-\omega}{2},
\end{align}
we know
from~\Cref{l:posdf} that
$$
g(\tau_i^{-1}e_i) \leq \nu_R-\frac{\nu_R-\omega}{2}\leq \nu_R-\delta',\enspace \forall i=\ell+1,\ldots,k.
$$
Then from~\Cref{a:deflamdai} we deduce that
$\theta_K\leq \tau$, i.e.,  Konno's cut is deeper than Tuy's cut.

In the remaining of this section,
we will present a new method to further improve Konno's cut.
Recall that the key idea in generating  Konno's cut is to give a computable upper bound on $\phi^*_{\tau,\theta}$ { defined in~\eqref{eq:caQPcuttau}}.
For finding deeper cuts, we  next
propose two computable upper  bounds of $\phi^*_{\tau,\theta}$ which are tighter than { Konno's bound} $\bar\phi^K_{\tau,\theta}$  {defined in~\eqref{a:barphiko}}.

\subsection{Tighter bound from {the} LP relaxation} \label{sec:lr-cut}

We start with  an estimate of  $\bar\phi^K_{\tau,\theta}$ {for any positive vector $\theta$.}
\begin{lemma}\label{l:ckonno}
     For any $\theta>0$ we have
    \begin{equation}
        \label{prob:konnoQP}
        \begin{aligned}
            \bar\phi^K_{\tau,\theta}\geq  \min \quad &  \beta+\nu \\
            \mathrm{s.t.} \quad &   Q  \leq  -\Lambda_{0}\tau^\top-\tau \Lambda_0^\top+\Lambda F^\top +F \Lambda^\top
            -\frac{1}{2}\left(d\theta^\top+\theta d^\top\right) \\
            & d   \leq \beta \theta +2\Lambda_0- 2\Lambda w \\
            & \Lambda_0\geq 0, \Lambda \geq 0.
        \end{aligned}
    \end{equation}
\end{lemma}
\begin{proof}
    Let
    $
    \beta=  \bar\phi^K_{\tau,\theta}-\nu.
    $
    In view of~\eqref{a:barphiko_g_form},
    $$
    g(\theta^{-1}_i e_{i})-\nu=\max_{y\in {\cF_\tau}}  \left(\theta^{-1}_i Q_{i}+ d\right )^\top y+\theta^{-1}_i d_i \leq \beta,\enspace \forall  i\in [k].
    $$
    Consequently, for each $i\in [k]$, the following system (on $y\in \R^{k}$) is infeasible:
    \begin{equation}\label{eq:sdsdf}
        \left\{
        \begin{aligned}
            & \left(\theta^{-1}_i Q_{i}+ d\right ) ^\top y > \beta-\theta^{-1}_i d_i \\
            & \tau^\top y \geq 1\\
            &F^\top y\leq w \\
            &y\geq 0
        \end{aligned}\right.
    \end{equation}
    By {Farkas' Lemma} and ${\cF_\tau} \neq \emptyset$, for each $i\in [k]$,  the infeasibility of~\eqref{eq:sdsdf} implies the feasibility of the following system  (on  ${u^i_0}\in \R$ and ${u^i}\in \R^{m}$):
    \begin{equation}\label{eq:esdfdf}
        \left\{
        \begin{aligned}
            &	\theta^{-1}_i Q_{i}+d  \leq  -{u^i_0} \tau+ F {u^i}  \\
            &\beta-\theta^{-1}_i d_i   \geq -{u^i_0} + ({u^i})^\top w\\
            &{u^i_0} \geq 0, {u^i} \geq 0
        \end{aligned}
        \right.
    \end{equation}
    Multiplying both sides of the inequalities by the positive number $\theta_i$,
    the system~\eqref{eq:esdfdf} can be equivalently written as
    \begin{equation}\label{eq:Qfi}
        \left\{
        \begin{aligned}
            & Q_{i}+\theta_i d  \leq  - {u^i_0} \tau + F {u^i} \\
            &\beta \theta_i - d_i   \geq -{u^i_0}  +({u^i})^\top w\\
            &{u^i_0} \geq 0, {u^i}  \geq 0
        \end{aligned}
        \right.
    \end{equation}
    The fact that the system~\eqref{eq:Qfi} is solvable for each $i\in [k]$  implies that the following system on $(\Lambda_0\in \R^{{k}}, \Lambda\in \R^{{k}\times m})$ is solvable:
    \begin{equation}\label{eq:l35}
        \left\{
        \begin{aligned}
            & Q  \leq  -\Lambda_{0}\tau^\top-\tau \Lambda_0^\top+\Lambda F^\top +F \Lambda^\top
            -\frac{1}{2}\left(d\theta^\top+\theta d^\top\right) \\
            & d   \leq \beta \theta+2 \Lambda_0- 2  \Lambda w \\
            & \Lambda_0\geq 0, \Lambda\geq 0
        \end{aligned}
        \right.
    \end{equation}
    To see this, just let $\Lambda_{0}=\frac{1}{2}\begin{pmatrix}{u_0^1}& \cdots & {u_0^k}\end{pmatrix}^\top$ and $\Lambda=\frac{1}{2}\begin{pmatrix}
        {u^1} & \cdots & {u^k}
    \end{pmatrix}^\top$  and use the fact that $Q=Q^\top$.
    \qed
\end{proof}
\begin{remark}
 In general the inequality in~\eqref{prob:konnoQP}  is strict since~\eqref{eq:l35} cannot imply~\eqref{eq:Qfi}.
\end{remark}

We next show that the lower bound of $\bar\phi^K_{\tau,\theta}$ provided by~\Cref{l:ckonno} is also an upper bound of $\phi^*_{\tau,\theta}$. For this,
consider the following  extended  LP program:
\begin{equation}
    \label{prob:konnoQP2e}
    \begin{aligned}
        \bar\phi^L_{\tau,\theta}:= \min \quad &  - {\alpha}+{q}^\top w+\beta+\nu \\
        \textrm{s.t.} \quad &  Q \leq   -\Lambda_0 \tau^\top- \tau \Lambda_0^\top+\Lambda F^\top +F\Lambda^\top +\left( \Lambda_{m+1}\theta^\top+\theta \Lambda_{m+1}^\top\right) \\
        & 2d-2\Lambda_0+2\Lambda w+2\Lambda_{m+1} \leq  -{\alpha} \tau+ F{q}  +\beta \theta\\
        & \Lambda_0 \geq 0,\Lambda\geq 0, \Lambda_{m+1} \geq 0
        ,  {\alpha} \geq 0, {q} \geq 0, \beta \geq 0.
    \end{aligned}
\end{equation}

\begin{proposition}\label{prop:IK}
    For any positive vector $\theta\in\R^{k}$, we have
    \begin{equation}
        \label{eq:comp1}
        \bar\phi^K_{\tau,\theta}  \geq\bar\phi^L_{\tau,\theta}\geq \phi^*_{\tau,\theta}.
    \end{equation}
\end{proposition}
\begin{proof}
    Any feasible solution to the optimization problem in~\eqref{prob:konnoQP} can be extended to a feasible solution of~\eqref{prob:konnoQP2e} by letting
    $$
    {\alpha}=0,\enspace {q}={\bf{0}}, \enspace \Lambda_{m+1}=-d/2.
    $$
    (Recall that $d\leq 0$.) Then we apply~\Cref{l:ckonno} to obtain the first inequality.

    Let $\Lambda_0,\Lambda, \Lambda_{m+1}, {\alpha}
    , {q} , \beta$ be feasible to~\eqref{prob:konnoQP2e}. Then for any $y\in {\cF_\tau} \cap \Delta(\theta)$, we have
    \begin{align*}
        y^\top Q y  & \leq -2y^\top \tau \Lambda_0^\top y+ 2 y^\top \Lambda F^\top y +2y^\top \Lambda_{m+1}\theta^\top y  \\&
        \leq  -2 \Lambda_0^\top y + 2 y^\top \Lambda w  +2y^\top \Lambda_{m+1}.
    \end{align*}
    Here, the first inequality follows from the first constraint of~\eqref{prob:konnoQP2e} and $y\geq 0$. The second inequality follows from the nonnegativity of $\Lambda_0 ^\top y$,  $\Lambda^\top y$ and $\Lambda_{m+1}^\top y$ and the fact that $\theta^\top y\leq 1$, $\tau^\top y \geq 1$ and $F^\top y \leq w$.
    Now we use the second constraint of~\eqref{prob:konnoQP2e} to obtain
    \begin{align*}
        y^\top Q y  & \leq -2d^\top y-{\alpha} \tau^\top y + {q}^\top F^\top y+\beta \theta^\top y\\
        & \leq -2d^\top y -{\alpha}+ {q}^\top w +\beta.
    \end{align*}
    Hence
    $$
     y^\top Q y+2d^\top y+\nu \leq -{\alpha}+{q}^\top w +\beta +\nu, \enspace \forall y \in {\cF_\tau} \cap \Delta(\theta),
    $$
    which implies  $\phi^*_{\tau,\theta}\leq  \bar\phi^L_{\tau,\theta} $.
    \qed
\end{proof}
Based on~\Cref{prop:IK},  the vector $\theta$  such that
$\bar\phi^L_{\tau,\theta}= \nu_R-\delta'$  leads to a cut deeper than {Konno's cut}.
Unfortunately, the search of such $\theta$ cannot be done by solving $k$ LP programs as for the search of $\theta_K$.
Nevertheless, we can try to improve Konno's cut through the standard bisection trick.
Specifically, let $\theta_K$ be the vector generating Konno's cut and choose some $ \theta <\theta_K$.
If $\bar\phi^L_{\tau,\theta}\leq \nu_R-\delta'$, then we get an improved cut.
Otherwise, we increase $ \theta$, for example to  $( \theta+\theta_K)/2$, and repeat until we get an improved cut.

\subsection{Deeper cut by {the} doubly nonnegative relaxation} \label{sec:dnnr-cut}
In this subsection, we show that a deeper cut than Konno's cut can be achieved by {the} DNN relaxation.
The main results are summarized in~\Cref{thm:phi-val-comp}.

We begin by rewriting~\eqref{eq:caQPcuttau} into the form of~\eqref{prob:original_max_qp0}.
Let $\hat F=\begin{pmatrix} F & \theta & -\tau \end{pmatrix}$ and $\hat w=\begin{pmatrix} w\\ 1\\ -1\end{pmatrix}$ so that~\eqref{eq:caQPcuttau} is written as
\begin{equation}
    \label{prob:konnoQP34}
    \begin{aligned}
        \phi^*_{\tau,\theta}=  \max \quad & y^\top Q y + 2 d^\top y +\nu  \\
        \textrm{s.t.}
        \quad & \hat F^\top y  \leq \hat w \\
        \quad & y \geq 0 .
    \end{aligned}
\end{equation}
Now~\eqref{prob:konnoQP34} takes the same form as \eqref{prob:original_max_qp0}, thus all the results on the DNN relaxation of~\eqref{prob:original_max_qp0} presented in~\Cref{sec:DNN} can be directly adapted to the DNN relaxation of~\eqref{prob:konnoQP34}.
Denote by
$
\bar\phi^D_{\tau,\theta}
$
the DNN relaxation value of~\eqref{prob:konnoQP34}, i.e.,
\begin{equation}
        \label{prob:SQP3Shor3ae}
        \begin{aligned}
            \bar\phi^D_{\tau,\theta}:=\max_{\substack{Y\in \cS^{k}\\ y\in \R^{k}}} \enspace &
            \<Q, Y> + 2 d^\top y  +\nu  \\
            \textrm{s.t.}
            \quad & \hat F^\top y\leq \hat w  \\
            \quad & y\geq 0 \\
            \quad & \hat F^\top Y \hat F-\hat w y^\top \hat F-\hat F^\top y \hat w^\top +\hat w\hat w^\top \geq 0\\
            \quad & Y \geq 0
            \\ \quad & \hat w y^\top -\hat F^\top Y \geq 0\\
            \quad &\begin{pmatrix}
                Y & y \\ y^\top & 1
            \end{pmatrix}  \succeq 0.
        \end{aligned}
    \end{equation}
    Note that $\bar\phi^D_{\tau,\theta}=\bar \Phi(\cF_{\tau}\cap \Delta(\theta))$.
Using directly~\Cref{prop:sfsdfsdf} we get the following upper bound of  $ \bar\phi^D_{\tau,\theta}  $.
\begin{corollary}
    \label{cor:sfsdfsdf}
    If there exists $\Lambda_0 ,\Lambda_{m+1} \in \R_{+}^{{k}}$, $\Lambda\in \R_{+}^{{k}\times m}$ and $P\in \R_{+}^{{k}\times {k}}\cap {\mathcal{S}^{k}}$  such that
    \begin{align}\label{a:cDe2}
        -Q-P-\Lambda_0 \tau^\top- \tau \Lambda_0^\top+\Lambda F^\top + F\Lambda^\top+\Lambda_{m+1}\theta^\top +\theta\Lambda_{m+1}^\top \succeq 0.
    \end{align}
    Then
    $$
    \bar\phi^D_{\tau,\theta}  \leq {\nu}+2 \max_{y\in {\cF_\tau} \cap \Delta(\theta)}  \left(d-\Lambda_0+ \Lambda w+\Lambda_{m+1} \right) ^\top y.
    $$
\end{corollary}
\begin{proof}
    {Let $\hat \Lambda={[ \Lambda \enspace \Lambda_{m+1} \enspace \Lambda_0]}\geq 0$.
    Then~\eqref{a:cDe2} can be written as
    $$
    -Q-P+\hat \Lambda \hat F^\top +\hat F \hat \Lambda^\top \succeq 0.
    $$
    Applying~\Cref{prop:sfsdfsdf}, we know that
    $$
    \begin{aligned}
    \bar \phi^D_{\tau,\theta} & \leq \nu+2 \max  \left\{ \left(d+  \hat\Lambda  \hat w \right) ^\top y: \hat F^\top y\leq \hat w,\enspace y\geq 0\right\}
    \\& ={\nu}+2 \max_{y\in {\cF_\tau} \cap \Delta(\theta)}  \left(d-\Lambda_0+ \Lambda w+\Lambda_{m+1} \right) ^\top y
    .
\end{aligned}
    $$
    }
    \qed
\end{proof}

\begin{theorem}
    \label{thm:phi-val-comp}
    For any positive vector $\theta\in\R_{++}^{{k}}$, we have
    \begin{equation}
        \label{eq:comp2}
        \bar\phi^K_{\tau,\theta}  \geq \bar\phi^L_{\tau,\theta} \geq \bar\phi^D_{\tau,\theta} \geq  \phi^*_{\tau,\theta}.
    \end{equation}
\end{theorem}
\begin{proof}
    The first inequality follows from~\Cref{prop:IK}.
    Let $\Lambda_0,\Lambda, \Lambda_{m+1}
    , {q},\alpha, \beta$ be feasible to {the program}~\eqref{prob:konnoQP2e}. Let
    $$
    P:=-Q-\Lambda_0 \tau^\top- \tau \Lambda_0^\top+\Lambda F^\top+ F\Lambda^\top+\Lambda_{m+1}\theta^\top +\theta\Lambda_{m+1}^\top.
    $$
Then~\eqref{a:cDe2} holds trivially. Further,
    by the first inequality constraint in~\eqref{prob:konnoQP2e}, $P\geq 0$.
    \Cref{cor:sfsdfsdf} allows to deduce:
    $$
    \bar\phi^D_{\tau,\theta}  \leq {\nu}+2 \max_{y\in {\cF_\tau} \cap \Delta(\theta)}  \left(d-\Lambda_0+ \Lambda w+\Lambda_{m+1} \right) ^\top y.
    $$
    Now we use the second constraint of~\eqref{prob:konnoQP2e} to obtain
    \begin{align*}
        &	2 \max_{y\in {\cF_\tau} \cap \Delta(\theta)}  \left(d-\Lambda_0+ \Lambda w+\Lambda_{m+1} \right) ^\top y  \\
        & \leq \max_{y\in {\cF_\tau} \cap \Delta(\theta)}  \left(-{\alpha} \tau +F{q}+\beta\theta\right) ^\top y\\
        & \leq  -{\alpha}+{q}^\top w +\beta.
    \end{align*}
    {Here, the first inequality follows from the second constraint of~\eqref{prob:konnoQP2e}, and the second inequality follows from $\tau^\top y \geq 1$, $F^\top y \leq w$ and $\theta^\top y \leq 1$ for any $y\in \cF_\tau \cap \Delta(\theta).$}
    It follows that
    $$
    \bar\phi^L_{\tau,\theta}\geq \bar\phi^D_{\tau,\theta}.
    $$
    \qed
\end{proof}
\Cref{thm:phi-val-comp} provides a foundation of the new cutting plane method that we are about to propose.
In order to search for a vector $\theta$ such that~\eqref{eq:caQPcuttau} holds, we choose a perturbation parameter {$0<\delta<\nu_R-\nu$} and start by computing {the vector $\tau$ defined in~\eqref{a:theta}. Then we compute $\omega$ defined by~\eqref{eq:omega}. If $\omega\geq \nu_R$, the reference value problem is solved. Otherwise we choose $0<\delta'\leq \frac{\nu_R-\max\left(\omega,\nu\right)}{2}$  and compute $\theta_K$ defined through~\eqref{a:deflamdai}} such that
$$
\bar \phi^K_{\tau,\theta_K}{\leq}\nu_R-{{\delta'}}.
$$
Then we choose a factor $\eta\in (0,1)$ and test whether $\theta=\eta\theta_K$ satisfies
\begin{align}\label{a:sdsefad}
\bar \phi^L_{\tau,\theta}\leq \nu_R-{{\delta'}}.
\end{align}
If~\eqref{a:sdsefad} holds, then $\theta$ yields a valid cut, which is deeper than Konno's cut.  Otherwise we further check if
\begin{align}\label{a:sdsefadfsd}
\bar \phi^D_{\tau,\theta}\leq \nu_R-{{\delta'}}.
\end{align}
If~\eqref{a:sdsefadfsd} holds, then $\theta$ yields a valid cut, which is deeper than Konno's cut. The pseudo code of this process is given in~\Cref{alg:ikc}.
One can also
use the bisection trick mentioned in the end of~\Cref{sec:lr-cut}. In any case, the core theory behind is~\Cref{thm:phi-val-comp}, which guarantees that
$$
\bar \phi^D_{\tau,\theta_K}\leq \bar \phi^L_{\tau,\theta_K}  \leq \bar \phi^K_{\tau,\theta_K} {\leq \nu_R-{\delta'}}<\nu_R,
$$
and provides a way to improve Konno's cut.

\section{Algorithm} \label{sec:cuttingplane}

In this section we summarize our method  in~\Cref{alg:CuP} and~\Cref{alg:CuP-global}. For the future reference, we name our   algorithm as \texttt{QuadProgCD}, where the letter ``C" refers to the cutting plane method and the letter ``D" refers to the doubly nonnegative relaxation.

The algorithm has two variants: 1. \texttt{QuadProgCD-R} (\Cref{alg:CuP}) is designed for solving the reference value problem~\eqref{a:possddf}, where we use the suffix \texttt{-R} to refer to the reference value problem;
2. \texttt{QuadProgCD-G} (\Cref{alg:CuP-global}) is designed for solving the global optimization problem \eqref{prob:original_max_qp0}, where the suffix \texttt{-G} refers to the global optimization mode.

\subsection{Algorithm for the reference value problem}
\Cref{alg:CuP} takes as input {the positive semidefinite matrix $Q\in \cS^k$, the vector $d\in \R^k$ and the scalar $\nu\in \R$ which define the quadratic objective function $\Phi$,  the matrix $F\in \R^{k\times m}$ and the vector $w\in \R^m$ which define the feasible region $\cF$},  the reference value $\nu_R$, a factor $\eta\in(0, 1)$, and a perturbation parameter $\delta>0$.
The algorithm solves the reference value problem~\eqref{a:possddf} by returning a lower bound $\underline v$ of $\Phi^*(\cF)$ and an upper bound $\bar v$ of $\Phi^*(\cF)$ such that $\nu_R\notin (\underline v, \bar v]$.

\Cref{alg:CuP} starts by searching for a KKT {vertex} $\bar y$ of~\eqref{prob:original_max_qp0} (\cref{algline:search-vertex}).  There are many ways to achieve a KKT {vertex} and we adopt the mountain climbing algorithm proposed by Konno in~\cite{konno1976maximization}. The details are recalled in~\Cref{sec:KKT}.

If  $\Phi(\bar y)$ is larger than or equal to $\nu_R$, then we obtain a lower bound $\underline v$ such that $\nu_R\leq \underline v\leq \Phi^*(\cF)$.
Otherwise, we compute the DNN  bound $\bar \Phi(\cF)$.
If the latter is strictly smaller than $\nu_R$, then we obtain an upper bound $\bar v$ such that $\nu_R> \bar v$ (note that if $\cF = \varnothing$, we would also reach this case because $\bar{\Phi}(\cF) = -\infty$).
Otherwise, we  proceed to the cutting plane step (from~\cref{algline:minimal-prog} to~\cref{algline:cutting-plane-step-end}).

In~\cref{algline:minimal-prog} we {make a suitable change of coordinate so that $\bar y$ is transformed to the origin point $\textbf{0}$ and compute the parameters $(Q,d,\nu,F,w)$ under the new basis}. For details of this step please check~\Cref{sec:mp}.
In~\cref{algline:updatedeltanu} we make sure that $\delta$ is smaller than $\nu_R-\nu$.
In~\cref{algline:cutting-plane-step-begin} we compute the vector $\tau$ which defines Tuy's cut. From~\cref{algline:phiKtautau-check-begin} to~\cref{algline:phiKtautau-check-end} we check whether $\omega$ defined in~\eqref{eq:omega} is {larger} than $\nu_R$. If {yes}, then we obtain a lower bound $\underline v$ such that $\nu_R\leq \underline v\leq \Phi^*(\cF)$. Otherwise, we {update $\delta$ so that it does not increase  and satisfies~\eqref{a:deltap} in~\cref{algline:updatedelta} . Then} we proceed to generate a deeper cut (from~\cref{algline:deeper-cut-begin} to~\cref{algline:deeper-cut-end}).

We provide two options for the generation of a cut deeper than Tuy's cut. The first option is to use Konno's cut as recalled  in~\Cref{subsec:konno}. The second option is to generate an even deeper cut  using the theory that we developed in~\Cref{sec:dnnr-cut}. The two algorithms for option I and II are described in~\Cref{sec:konnocut} and~\Cref{sec:bisec-cut}.

In the while loop, the lower bound $\underline v$ is nondecreasing and the upper bound $\bar v$ is nonincreasing.
When we solve the reference value problem~\eqref{a:possddf},  the while loop breaks when either $ \nu_R\leq  \underline v$ or $\nu_R> \bar v$ so that the reference value problem is solved.

\begin{algorithm}[htbp]
    \caption{\texttt{QuadProgCD-R}}
    \begin{algorithmic}[1]
        \Require {  $Q\succeq 0$, $d\in \R^k$, $\nu \in \R$, $F\in \R^{k\times m}$, $w\in \R^m$,} reference value $\nu_R\in \R$, factor $\eta\in (0,1)$, perturbation $\delta>0$.
  \State $\underline{v} \gets -\infty, \bar{v} \gets \infty$;
        \While{$\underline v < \nu_R\leq \bar v $}
        \State { $\bar y\leftarrow$ \texttt{Search\_of\_KKT\_Point}($Q,d,\nu,F,w$);} \Comment{See~\Cref{alg:MC}.} \label{algline:search-vertex}
        \If{{$\bar y ^\top Q \bar y+2 d^\top \bar y+\nu \geq \nu_R$}}
        \State {$\underline{v}\gets \bar y ^\top Q \bar y+2 d^\top \bar y+\nu $ } and break; \Comment{Terminate with $\underline{v} \geq \nu_R$.}
        \EndIf
        \State $\underline v\gets \max\left(\underline v, \bar y ^\top Q \bar y+2 d^\top \bar y+\nu \right)$; \Comment{Keep the best lower bound value ever reached.}
        \State $t\leftarrow $ { \texttt{DNN\_Bound}($Q,d,\nu,F,w$);} \label{algline: compute-dnn-bound}\Comment{Compute the DNN bound by solving~\eqref{prob:SQP3Shor}.}
        \State {$\bar v\leftarrow\min\left(\bar v, t\right)$; }\Comment{Keep the best upper bound value ever obtained.}
        \If{$\bar v < \nu_R$}
             \State  $\bar v\leftarrow  \max(\bar v, \nu_R-\delta) $ and break; \Comment{Terminate with $\bar{v} < \nu_R$.}
        \EndIf
        \State $(Q, d, \nu, F,w)\leftarrow$ { \texttt{Change\_of\_Coordinate}}($\bar y, Q, d, \nu, F,w$);\Comment{See~\Cref{alg:minimalprogram}.} \label{algline:minimal-prog}
 \State  { $\delta \leftarrow \min (\delta, (\nu_R-\nu)/2)$; }\label{algline:updatedeltanu}
  \For{$i=1,\ldots,k$}  \State $\tau_i \leftarrow\frac{Q_{ii}}{-d_i+\sqrt{d^2_i+Q_{ii}\left( \nu_R-\delta-\nu\right)}}$ \Comment{Tuy's cut.} \label{algline:cutting-plane-step-begin}
   \EndFor \label{algline:endfor}
  \For{$i=1,\ldots,{k}$} \label{algline:phiKtautau-check-begin}
  \If{$\tau_i>0$} \label{algline:iftau}
  \State $y^i\leftarrow\argmax\{ (\tau^{-1}_i Q_{i}+d)^\top y: F^\top y\leq w, \tau^\top y \geq 1, y\geq 0\}$;
  \State $\underline v \leftarrow   \max \left( \underline v, (y^i)^\top Q y^i+2 d^\top y^i+\nu \right)$;
  \If{$ \underline v \geq \nu_R$}
  \State break;\Comment{Terminate with  $\underline v\geq  \nu_R$.}
  \EndIf
  \EndIf
  \EndFor \label{algline:phiKtautau-check-end}
   \State $\delta \leftarrow \min (\delta, (\nu_R-\underline v)/2)$;  \label{algline:updatedelta}
  \State $\theta_K\leftarrow$ \texttt{Konno\_Cut}($Q,d, \nu,F,w, \nu_R,\tau,\delta$); \label{algline:deeper-cut-begin}\Comment{See~\Cref{alg:Konnocut}.}
        \State \textbf{Option I}: $\theta\leftarrow \theta_K$ ; \Comment{Use Konno's cut.}
        \State \textbf{Option II}: $\theta\leftarrow$ \texttt{DNN\_Cut}($Q,d, \nu,F,w, \nu_R,\tau,\theta_K, \eta,\delta$);\Comment{See~\Cref{alg:ikc}.}
        \label{algline:dnncut}
        \State { $F\leftarrow \begin{pmatrix}F & -\theta\end{pmatrix} $, $w\leftarrow \begin{pmatrix}w \\ -1\end{pmatrix}$;} \Comment{Add a valid cut.}
         \label{algline:cutting-plane-step-end} \label{algline:deeper-cut-end}
        \EndWhile
        \Ensure upper bound  $\bar v$, lower bound $\underline v$ such that $\nu_R\notin (\underline v, \bar v]\enspace$.
    \end{algorithmic}
    \label{alg:CuP}
\end{algorithm}

\subsection{Algorithm for global solution of the concave QP problem}\label{sec:quadprogcd-global}

Now we present \texttt{QuadProgCD-G} as in~\Cref{alg:CuP-global} for globally solving the concave QP problem~\eqref{prob:original_max_qp0}.
\texttt{QuadProgCD-G} follows a similar structure as~\Cref{alg:CuP}, with some differences that we detail below.

The input of~\Cref{alg:CuP-global} includes a gap tolerance $\epsilon > 0$.
The output is a lower bound $\underline v$ and an upper bound $\bar v$ of $\Phi^*(\cF)$ such that
\begin{equation}\label{eq:relgapdef} \dfrac{\bar{v} - \underline{v}}{\max\{\epsilon, |\underline{v}|\}} \leq \epsilon.\end{equation}
Here, we take the maximum of $\epsilon$ and $|\underline{v}|$ in the denominator to handle the pathological case $|\underline{v}| = 0$.

 There is  no longer a prescribed reference value $\nu_R$. Instead,
 \Cref{alg:CuP-global} uses a dynamic reference value, which is set to be the lower bound ${ \underline{v}+\epsilon^2}$ and  updated  throughout the iterations (\cref{algline:simul-update-1,algline:simul-update-2}).
The termination criteria are also modified to ensure global convergence (\cref{algline:termnation-criterion-g,algline:break-criterion-g,algline:break-criterion-3}).

Another difference with~\Cref{alg:CuP} is that in~\Cref{alg:CuP-global}  the perturbation parameter  $\delta$ is set to be zero.
This is because~\Cref{alg:CuP-global} uses the dynamic reference value $\nu_R = \underline{v}{+\epsilon^2}$.
As a result, there is no harm to cut off any solution with {an} objective value equal to $\nu_R$, since we must have recorded some points with the same objective value $\underline{v}$ beforehand.
Tuy's cut is modified accordingly (\cref{algline:Tuycut2}). Comparing with~\cref{algline:cutting-plane-step-begin} of~\Cref{alg:CuP}, we see that the term $\nu_R-\delta-\nu$ is replaced with $ \underline v+\epsilon^2-\nu$. This is because  $\nu_R={\underline v}+\epsilon^2$ and $\delta=0$ in~\Cref{alg:CuP-global}.  In this case, we cut the region with {an} objective value bounded by $\underline v+\epsilon^2$, which is consistent with our termination criteria~\eqref{eq:relgapdef}.

Finally, the lower bound $\underline v$ returned by~\Cref{alg:CuP-global} corresponds to the objective value of  a certain feasible solution of~\eqref{prob:original_max_qp0}.
The corresponding solution can be added to the output if needed.

\begin{algorithm}[htbp]
    \caption{\texttt{QuadProgCD-G}}
    \begin{algorithmic}[1]
    \Require { $Q\succeq 0$, $d\in \R^k$, $\nu \in \R$, $F\in \R^{k\times m}$, $w\in \R^m$,} relative gap tolerance $\epsilon>0$, factor $\eta\in (0,1)$.
    \State $\underline{v} \gets -\infty; \bar{v}\gets \infty$;
        \While{$\bar v -\underline v > \epsilon \max\left(|\underline v|,\epsilon\right)$} \label{algline:termnation-criterion-g}
        \State {$\bar y\leftarrow$ \texttt{Search\_of\_KKT\_Point}($Q,d,\nu,F,w$);} \Comment{See~\Cref{alg:MC}.}
        \State $\underline v \gets \max\left(\underline v, \bar y^\top Q \bar y+ 2d^\top \bar y +\nu\right)$; \Comment{Keep the best lower bound value ever reached.} \label{algline:simul-update-1}
        \State $t\leftarrow $ {\texttt{DNN\_Bound}($Q,d,\nu,F,w$)};  \Comment{Compute the DNN bound by solving~\eqref{prob:SQP3Shor}.} \label{algline: compute-dnn-bound2}
         \State {$\bar v\leftarrow\min\left(\bar v, t\right)$; }\Comment{Keep the best upper bound value ever reached.}
        \If{{$\bar v \leq \underline{v}+ \epsilon \max\left(|\underline v|,\epsilon\right)$}} \label{algline:break-criterion-g}
            \State  {$\bar v\leftarrow   \max\left(\bar v, \underline v+\epsilon^2\right)$} and break; \Comment{Terminate.}
        \EndIf \label{algline:endifg}
        \State $(Q, d, \nu, F,w)\leftarrow$ \texttt{Change\_of\_Coordinate}($\bar y, Q, d, \nu, F,w$);\Comment{See~\Cref{alg:minimalprogram}.}\label{algline:reduce}
  \For{$i=1,\ldots,k$} \State $\tau_i \leftarrow\frac{Q_{ii}}{-d_i+\sqrt{d^2_i+Q_{ii}\left(\underline v+ \epsilon^2-\nu\right)}}$; \label{algline:Tuycut2} \Comment{Tuy's cut.}
  \EndFor
  \For{$i=1,\ldots,{k}$}
  \If{$\tau_i>0$}
  \State $y^i\leftarrow\argmax\{ (\tau^{-1}_i Q_{i}+d)^\top y: F^\top y\leq w, \tau^\top y \geq 1, y\geq 0\}$;
  \State $\underline{v} \leftarrow   \max \left( \underline v, (y^i)^\top Q y^i+2 d^\top y^i+\nu \right)$; \label{algline:simul-update-2}
  \If{ {$\bar v-\underline v \leq \epsilon \max(|\underline v|, \epsilon)$}} \label{algline:break-criterion-3}
  \State { $\bar v \leftarrow  \max\left(\bar v,\underline v +\epsilon^2\right)$ and } break;  \Comment{Terminate.}
  \EndIf\label{algline:enfif-3}
  \EndIf
  \EndFor
  \State $\theta_K\leftarrow$ \texttt{Konno\_Cut}($Q,d, \nu,F,w, \underline{v}+\epsilon^2, \tau, 0$); \Comment{See~\Cref{alg:Konnocut}.}
    \State \textbf{Option I}: $\theta\leftarrow \theta_K$ ; \Comment{Use Konno's cut.}
    \label{algline:konnocut2}
    \State \textbf{Option II}: $\theta\leftarrow$ \texttt{DNN\_Cut}($Q,d, \nu,F,w, \underline{v}+\epsilon^2, \tau, \theta_K,\eta,0$);\Comment{See~\Cref{alg:ikc}.}
    \label{algline:dnncut2}
    \State { $F\leftarrow \begin{pmatrix}F & -\theta\end{pmatrix} $, $w\leftarrow \begin{pmatrix}w \\ -1\end{pmatrix}$;}\Comment{Add a valid cut.}
    \EndWhile
        \Ensure Upper bound  $\bar v$, lower bound $\underline v$ such that $0\leq \bar v-\underline v\leq \epsilon \max \left(|\underline v|,\epsilon\right)$.
    \end{algorithmic}
    \label{alg:CuP-global}
\end{algorithm}

\subsection{Computation of the DNN bound}
The DNN bound in~\cref{algline: compute-dnn-bound} of~\Cref{alg:CuP} and in~\cref{algline: compute-dnn-bound2} of~\Cref{alg:CuP-global} can be replaced by any valid upper bound of $\Phi^*(\cF)$. In fact, as discussed in~\Cref{sec:vdnnbi}, it is impossible to compute the exact value of $\bar\Phi(\cF)$, which corresponds to the optimal value of the SDP problem~\eqref{prob:SQP3Shor}. For high-dimensional problems, even an approximation with high accuracy  of $\bar\Phi(\cF)$ is not accessible due to {the} memory issue or {the} time limit. We rely on~\Cref{prop:inexcs} for obtaining a valid upper bound of $\Phi^*(\cF)$ from any approximate SDP solution. Similarly, when we compute the DNN cut in~\cref{algline:dnncut} of~\Cref{alg:CuP} and in~\cref{algline:dnncut2} of~\Cref{alg:CuP-global}, it suffices to find a valid upper bound of  the optimal value of the SDP problem~\eqref{prob:SQP3Shor3ae} and we again rely on~\Cref{prop:inexcs}. In this way,~\Cref{alg:CuP} and~\Cref{alg:CuP-global} are open to a wide range of SDP solvers, including in particular those that are able to provide medium accuracy solutions for high dimensional problems. This flexibility is crucial to make \texttt{QuadProgCD} competitive when the problem size increases.
The price to pay is the degradation of the quality of the upper bound. Nevertheless, as we will show in the next section through extensive numerical tests, the gain appears to outweigh  the loss and \texttt{QuadProgCD} shows superior performance compared with existing solvers especially for problems with large size.

\subsection{Convergence of the algorithm}\label{sec:conv}

As the paper is already lengthy, we  have decided to leave the study of the convergence properties of~\Cref{alg:CuP} and~\Cref{alg:CuP-global} to future research. Nonetheless, we would like to provide some brief comments on this challenging problem.

Firstly, it is important to note that the lower bound $\underline{v}$ is only guaranteed to be non-decreasing, and the upper bound $\bar{v}$ is only guaranteed to be non-increasing during the iterations. In other words, while adding a valid cut always leads to a strictly smaller feasible region, it may not always result in a strict improvement of the upper or lower bound. Interested readers may refer to the numerical experiments in~\Cref{fig:time-lbub} in the later section, which demonstrate the stagnation of both lower and upper bounds during the iterations.

In fact,  finite convergence of a ``pure cutting plane method" is hard to be ensured. The best-known result is the finite convergence of Tuy's cut, which is established under certain conditions; see~\cite[Theorem V.2]{horst2013global}. The condition basically requires that the cuts successively added should be sufficiently large.
Without imposing such conditions, even the convergence of Konno's cut is unknown, according to Chapter V Section 4.2 of~\cite{horst2013global}.  As an alternative, many papers propose the modification of the pure cutting plane method by incorporating cone-splitting, enumerative elements, or branch-and-bound techniques to ensure finite convergence; see~\cite{jacobsen1981convergence,Porembski01,VanTuy80}.

\section{Numerical Experiments}\label{sec:experiments}

In this section, we evaluate the numerical performance of the proposed algorithms \texttt{QuadProgCD-R} (\Cref{alg:CuP}) and \texttt{QuadProgCD-G} (\Cref{alg:CuP-global}), one for solving the reference value problem~\eqref{a:possddf} and the other for solving globally~\eqref{prob:original_max_qp0}.
For the sake of simplicity, when the problem type is clear,  we  use \texttt{QuadProgCD} to refer to both of the two variants.
We compare \texttt{QuadProgCD} with two commercial solvers \texttt{CPLEX}  (version 22.1.0) \footnote{https://www.ibm.com/hk-en/analytics/cplex-optimizer} and \texttt{Gurobi} (version 9.5.0) \footnote{https://www.gurobi.com}, and also with an open-source solver \texttt{quadprogIP}\footnote{https://github.com/xiawei918/quadprogIP} which is a nonconvex QP solver proposed very recently in \cite{xia2020globally}.

Our algorithm \texttt{QuadProgCD} is implemented and run with \texttt{MATLAB R2021a} (version 9.10.0.1851785).
In numerical experiments, the interfaces of \texttt{CPLEX} and \texttt{Gurobi} are called through a Python script. Unless otherwise specified,
all the tests are performed on a Windows laptop with Intel(R) Core(TM) i7-8750H CPU 2.20 GHz, 6 cores and 16GB memory.
The code of \texttt{QuadProgCD} is available at
\url{https://github.com/tianyouzeng/QuadProgCD}.

\subsection{Problem instances}

We use both real and synthetic data for experiments.
All the instances used for the comparison are uploaded  at the above address.
We describe them in detail below.

\subsubsection{Real data}

This set of data has a real application background in computational biology.  In \cite{zhao2020new}, the authors proposed an approach to detect new genome or protein sequence based on some known sequence data. A key step in their approach is to solve a set of reference value problems  which can be described as
\begin{equation}
    \label{prob:numexp_max_qp0}
    \begin{aligned}
        \max_{x\in \R^{n}}  &  \qquad   x^\top H x + 2p^\top x  \\
        ~ \textrm{s.t.} &\qquad   Ax = b \\ &\qquad x \geq 0,
    \end{aligned}
\end{equation}
where $H \in \cS^n$ is positive semidefinite, $p \in \R^n$, $A \in \R^{m \times n}$ with $m < n$, and $b \in \R^m$.
By elimination of variables, program~\eqref{prob:numexp_max_qp0} can be reduced into the form of~\eqref{prob:original_max_qp0} with $k = n - m$.
Due to the very special application background, the matrices and vectors $H$, $p$, $A$, $b$ for describing the concave QP problem~{\eqref{prob:numexp_max_qp0}} are  dense. Moreover, all the entries in   $p$, $A$, $b$ are nonnegative and $H$ is a completely positive matrix.

The data is divided into two groups: one group of dimension $n=100$ and the other group of dimension $n=841$. For both groups, the number of rows of $A$ is $m=22$.
Therefore in the reduced form~\eqref{prob:original_max_qp0}, we have $k = 78$ and $k = 819$ for each group respectively.
The reference value is fixed to be $n(n+1)(2n+1)/6$ for all the instances.
For convenience, we call the two groups BioData100 and BioData841 respectively.
For BioData100, there are 317,458 instances, while for BioData841 we have only 3 instances.

\subsubsection{Synthetic data}

We next consider QP in the format of~\eqref{prob:numexp_max_qp0} with randomly generated synthetic data.
The matrices and vectors required to define the program are generated as follows.
For $k_1,k_2 \in \mathbb{N}$ and $a, b\in \R$, we denote by $\mathcal{U}(k_1, k_2,a,b)$ a random matrix with $k_1$ rows and $k_2$ columns such that each entry is a uniform random variable in $[a,b]$ and entries are mutually independent.
We first fix the data dimension $n$, then we generate the feasible region $\cF := \{x: Ax = b, x \geq 0\}$ by the following steps:
\begin{enumerate}
    \item Choose  $m$ uniformly randomly from the integers in $[0.1n,0.5n]$;
    \item Generate $A\sim \cU(n,m,-20,20)$;
    \item Generate  $x_0 \sim \cU(n,1,0,1)$;
    \item Let $b = Ax_0/\|x_0\|$. If $\cF$ is bounded, remove redundant rows of $A$ if necessary and terminate. Otherwise go to step 1.
\end{enumerate}
Such a procedure guarantees that $\cF$ is feasible and bounded so that  Assumption~\ref{ass:bound_and_interior} is satisfied. The objective function $\Phi$ is generated as follows:
\begin{enumerate}
    \item[5.] Generate $U\sim \cU(n,n,-1,1)$,  $p\sim \cU(n,1,-10,10)$;
    \item[6.] Generate  $h \sim \cU(n,1,0,1)$.
    Let $D_0=\diag(h)$ and $H_0 = UD_0U^\top$;
    \item[7.] Generate $\alpha \sim \cU(1,1,10,11)$.
    Set $H=n\alpha H_0/\|H_0\|$ and $D = n\alpha D_0 / \|H_0\|$.
\end{enumerate}
Note that the  matrices and vectors generated in this way are dense.

The synthetic data
is divided into seven groups: CQMAX20, CQMAX50,   CQMAX100,  PCQMAX20, PCQMAX50, PCQMAX100 and PCQMAX500.  The integer number in the suffix refers to the dimension $n$ of the instance.
  Instances with prefix CQMAX are generated following the above procedure. Instances with prefix PCQMAX are generated with a small difference in step 2 and 5:   $A\sim \cU(n,m,0,20)$, $U\sim \cU(n,n,0,1)$ and $p\sim \cU(n,1,0,10)${, so that the generated instances have positive entries for each matrix or vector.}

 For problems of dimension $n\leq 100$, we solve {their corresponding reduced} concave QP problem~\eqref{prob:original_max_qp0} by \texttt{CPLEX} or \texttt{Gurobi} to get the value $\Phi^*(\cF)$ and  generate $\nu_R$ in a neighborhood of $\Phi^*(\cF)$.  For problem of dimension 500, as we have no access to  $\Phi^*(\cF)$ using \texttt{CPLEX} or \texttt{Gurobi},  we  set $\nu_R$ to be a random number larger than some known lower bound of $\Phi^*(\cF)$.

\subsection{Numerical results}

This subsection is divided into three parts.
In the first part (\Cref{subsec:numexp-ref-val-prob} and~\Cref{subsec:sfsdfsdf}), we present numerical results for solving the reference problem with~\Cref{alg:CuP}. In the second part (\Cref{subsec:gsosssw}), we show the numerical results for the global solution of~\eqref{prob:original_max_qp0} with~\Cref{alg:CuP-global}. In the third part (\Cref{subsec:comakc}), we demonstrate the difference between option I (Konno's cut) and option II (DNN cut) in~\Cref{alg:CuP-global}.

Unless otherwise specified, we let $\eta=1/2$ and $\delta=10^{-6}$.
The displayed computational time are all measured in seconds.
The SDP problems are solved  with either \texttt{MOSEK} or \texttt{SDPNAL+}.
A valid upper bound based on the  solution returned by the SDP solver   is then computed using the formula established in~\Cref{prop:inexcs}.

\subsubsection{Reference value problems with real data}
\label{subsec:numexp-ref-val-prob}

In this section, we report the computational results of \texttt{QuadProgCD-R} (\Cref{alg:CuP}) for solving the reference value problem arising from the new genome or protein detection problem.
Recall that we have two groups of instances: BioData100 with 317,458 instances  and BioData841 with 3 instances.

For BioData841, the test results are summarized in~\Cref{tab:refval-biodata841}.
\texttt{CPLEX} and \texttt{Gurobi} both failed to solve these three instances within \textbf{1000} seconds, while our algorithm \texttt{QuadProgCD} managed to solve them in less than \textbf{40} seconds each. Due to the large dimension $n=841$,   \texttt{MOSEK} fails to terminate in reasonable time limit for the inner SDP problems and we had to {adopt} \texttt{SDPNAL+} for solving the inner SDP problems.

\ExplSyntaxOn
\NewDocumentCommand { \getmin } {  }
  {
    \clist_gset:Nx \g_tmpa_clist {\csvcoliii, \csvcolv, \csvcolvii, \csvcolix}
    \clist_sort:Nn \g_tmpa_clist
      {
        \fp_compare:nNnTF {##1} > {##2}
          { \sort_return_swapped: }
          { \sort_return_same: }
      }
    \tl_gset:Nx \g_tmpa_tl { \clist_item:Nn \g_tmpa_clist {1} }
  }
  \NewDocumentCommand { \getmintwo } {  }
  {
    \clist_gset:Nx \g_tmpa_clist {\csvcolii, \csvcoliii}
    \clist_sort:Nn \g_tmpa_clist
      {
        \fp_compare:nNnTF {##1} > {##2}
          { \sort_return_swapped: }
          { \sort_return_same: }
      }
    \tl_gset:Nx \g_tmpa_tl { \clist_item:Nn \g_tmpa_clist {1} }
  }
\NewDocumentCommand { \mynum } { m }
  {
    \fp_compare:nNnTF { #1 } = { \g_tmpa_tl }
      { \fontsize{7pt}{7pt}\selectfont \textbf{\num{#1}} }
      { \fontsize{7pt}{7pt}\selectfont \num{#1} }
  }
\NewDocumentCommand{ \myrelgap } { m }
{
    \str_if_eq:eeTF { #1 } { - }
   { \fontsize{7pt}{7pt}\selectfont #1 }
   { \str_if_eq:eeTF { #1 } { N/A } { \fontsize{7pt}{7pt}\selectfont #1 } { \fontsize{7pt}{7pt}\selectfont \num{#1} } }
}
\ExplSyntaxOff

\begin{table}[htbp]
    \scriptsize
    \centering
    \caption{Comparison of three algorithms on the dataset BioData841. Note that the instance with time equal to 1000 means that the corresponding solver failed to solve it within 1000 seconds.}
    \label{tab:refval-biodata841}
    \begin{filecontents*}{tables/biodata831-refval-log.csv}
Instance,RefVal,Status,ALG TIME,CPLEX TIME,GUROBI TIME
1,6868.3442,Not attained,36.87,1000,1000
2,6868.3442,Attained,0.32,1000,1000
3,6868.3442,Not attained,36.8,1000,1000
\end{filecontents*}
\DTLloaddb{tabledata}{tables/biodata831-refval-log.csv}
    \sisetup{round-mode=places, zero-decimal-to-integer, round-precision=2, scientific-notation=true,
    detect-all}
    \begin{tabular}{ c  c c c }
        \toprule
        Instance  & \texttt{QuadProgCD} Time & \texttt{CPLEX} Time & \texttt{Gurobi} Time  \\
        \midrule
        \csvreader[head to column names]{tables/biodata831-refval-log.csv}{}{%
            \csvcoli  & \textbf{{\csvcoliv}} & \csvcolv & \csvcolvi \\
        }
        \\[-\normalbaselineskip] \bottomrule
    \end{tabular}
    \normalsize
\end{table}

For BioData100,
we randomly selected 100 instances out of the 317,458 instances  and tested the three algorithms for  the reference value problem. The results are displayed in~\Cref{fig:refval-biodata100-sdpnalp}. Here we follow the  format used in~\cite{chen2012globally} to display the comparison.
Each square represents one instance, and its $xy$-coordinates represent the wall-clock times of the two methods. The diagonal $y=x$ line is plotted in dashed line for reference. If a square is located above the diagonal,  \texttt{QuadProgCD} solves the instance faster. The far  the square  from the diagonal, the larger  the ratio between the two wall-clock times.  We set the maximum time limit to be {1000} seconds and the dotted horizontal line represents this time limit. If a square sits on the dotted horizontal line in the left (resp. right) plot, it means that \texttt{CPLEX} (resp. \texttt{Gurobi}) is not able to determine whether $\nu_R>\Phi^*(\cF)$ or $\nu_R\leq \Phi^*(\cF)$ within 1000 seconds.

From~\Cref{fig:refval-biodata100-sdpnalp}, we observe that \texttt{CPLEX} failed to solve most of the instances within \textbf{1000} seconds. \texttt{Gurobi} seems to perform better than \texttt{CPLEX} but it also failed to solve about half of the instances within the time limit. In contrast, our algorithm \texttt{QuadProgCD} is able to solve all the  100 instances with an average wall-clock time around \textbf{10} seconds.

\begin{figure}[ht]
    \centering
    \begin{subfigure}[t]{0.45\textwidth}
        \includegraphics[width=\textwidth]{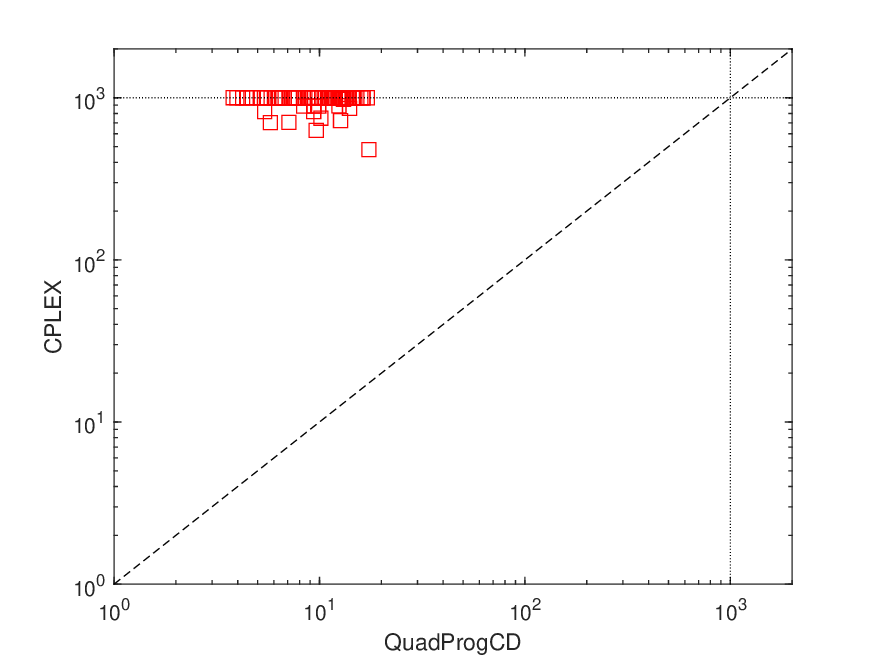}
        \caption{\texttt{QuadProgCD} vs \texttt{CPLEX} on BioData100.}
    \end{subfigure}
    \begin{subfigure}[t]{0.45\textwidth}
        \includegraphics[width=\textwidth]{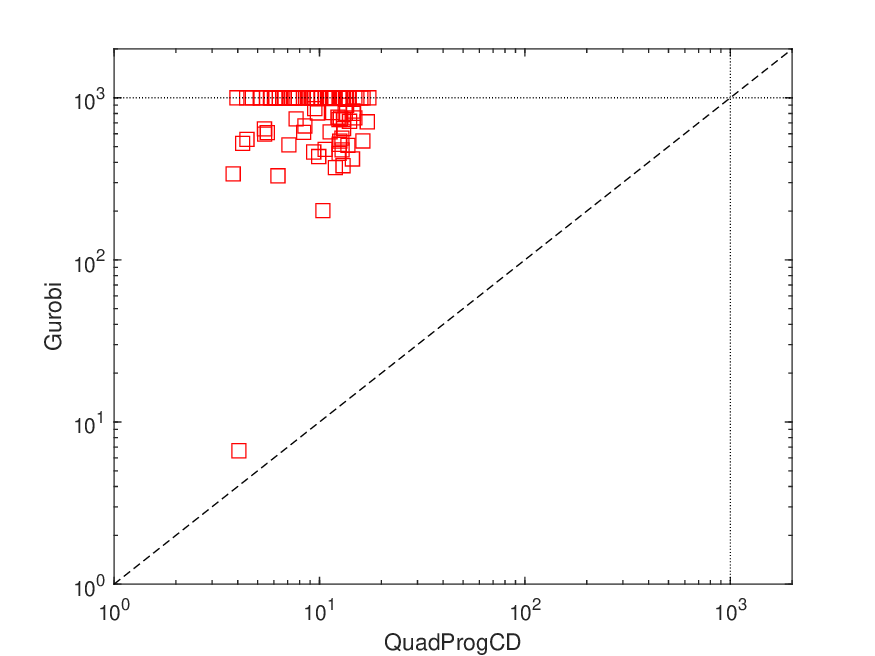}
        \caption{\texttt{QuadProgCD} vs \texttt{Gurobi} on BioData100.}
    \end{subfigure}
    \caption{Wall-clock time comparison for  solving the reference value problem on  100 instances randomly selected from BioData100, plotted in log-log scale.
   }
    \label{fig:refval-biodata100-sdpnalp}
\end{figure}

We also run our algorithm \texttt{QuadProgCD} to solve the reference value problem for all the 317,458 instances of dimension $n=100$ ($k=78$) on the HKU High Performance Computing cluster\footnote{https://hpc.hku.hk/hpc/hpc2021/} using 32 processors at the same time. The algorithm successfully solved all the instances within \textbf{3 days} (about $5\%$ of them are infeasible instances).  However, with \texttt{CPLEX} and \texttt{Gurobi}, the computational time for one single instance may take up to  several hours  and it is estimated to take \textbf{years} of computational time for solving all the  317,458 instances with these two solvers in the same computational environment.

\subsubsection{Reference value problems with synthetic data}\label{subsec:sfsdfsdf}

In this section we compare the performance of \texttt{QuadProgCD-R} (\Cref{alg:CuP}), \texttt{CPLEX} and  \texttt{Gurobi}
for the reference problem on synthetic data PCQMAX.
Each group PCQMAX20, PCQMAX50, PCQMAX100 has 100 instances. The group  PCQMAX500 has 10 instances.

\begin{figure}[ht!]
    \centering

    \begin{subfigure}[t]{0.45\textwidth}
        \includegraphics[width=\textwidth]{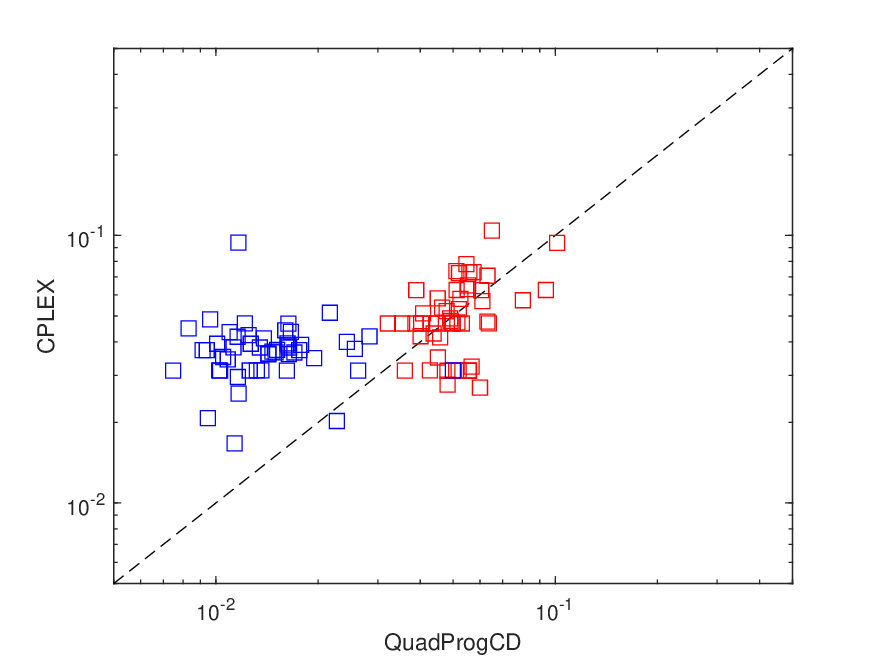}
        \caption{\texttt{QuadProgCD} vs \texttt{CPLEX} on PCQMAX20.}
    \end{subfigure}
    \begin{subfigure}[t]{0.45\textwidth}
        \includegraphics[width=\textwidth]{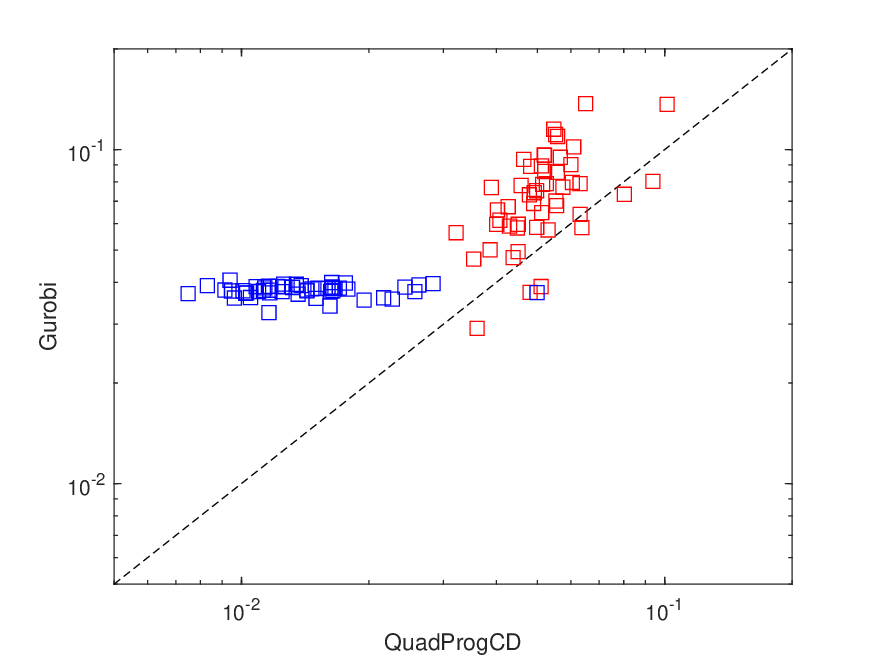}
        \caption{\texttt{QuadProgCD} vs \texttt{Gurobi} on PCQMAX20.}
    \end{subfigure}
    \hfill
    \begin{subfigure}[t]{0.45\textwidth}
        \includegraphics[width=\textwidth]{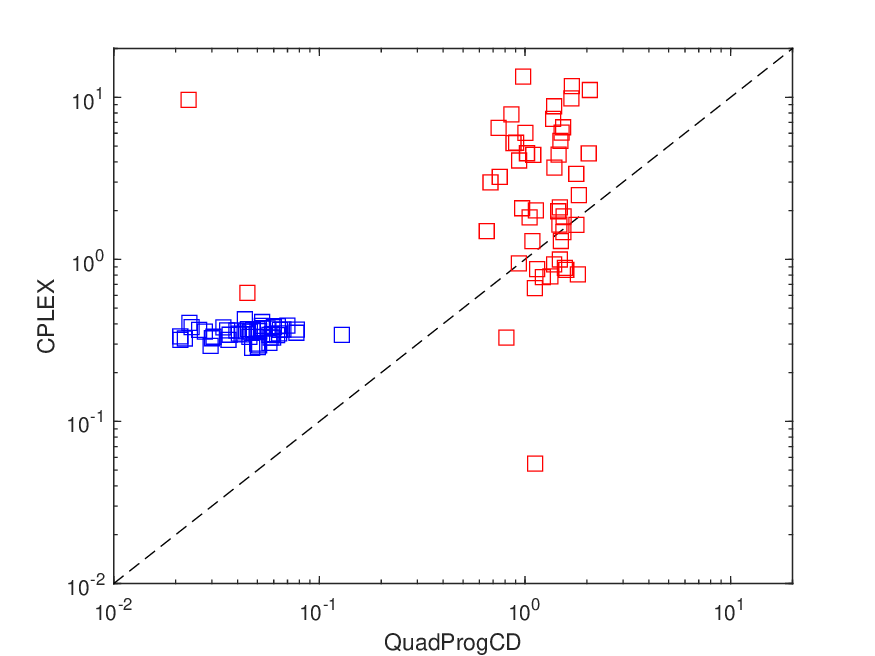}
        \caption{\texttt{QuadProgCD} vs \texttt{CPLEX} on PCQMAX50.}
    \end{subfigure}
    \begin{subfigure}[t]{0.45\textwidth}
        \includegraphics[width=\textwidth]{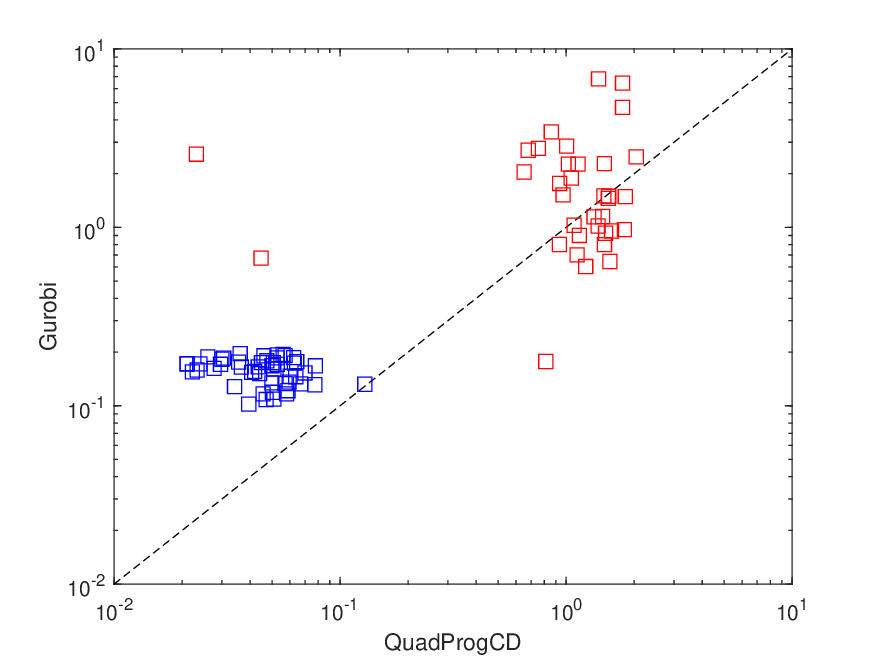}
        \caption{\texttt{QuadProgCD} vs \texttt{Gurobi} on~PCQMAX50.}
    \end{subfigure}
    \hfill
    \begin{subfigure}[t]{0.45\textwidth}
        \includegraphics[width=\textwidth]{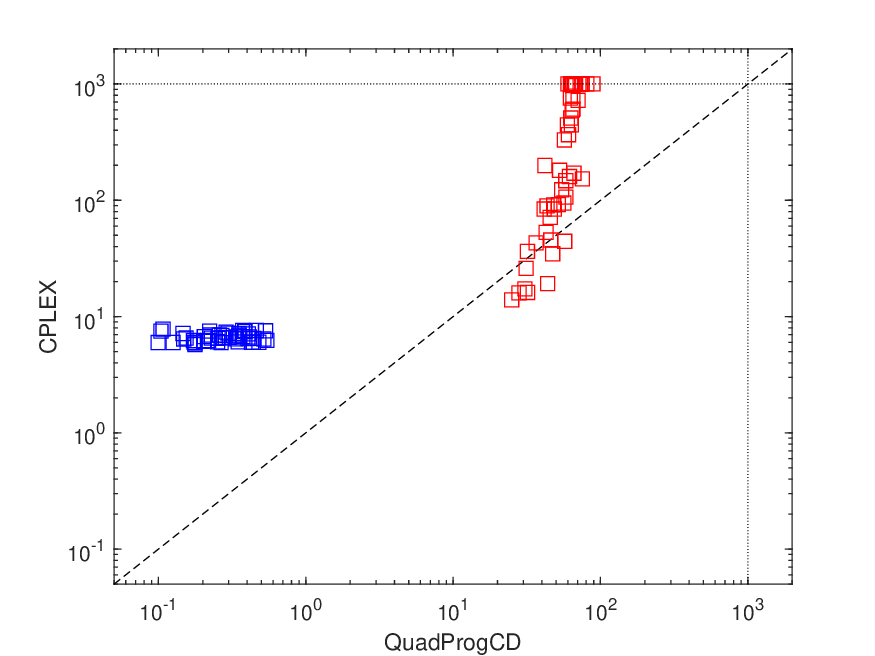}
        \caption{\texttt{QuadProgCD} vs \texttt{CPLEX} on PCQMAX100.}
  \label{subfig5:QC100M}
    \end{subfigure}
    \begin{subfigure}[t]{0.45\textwidth}
        \includegraphics[width=\textwidth]{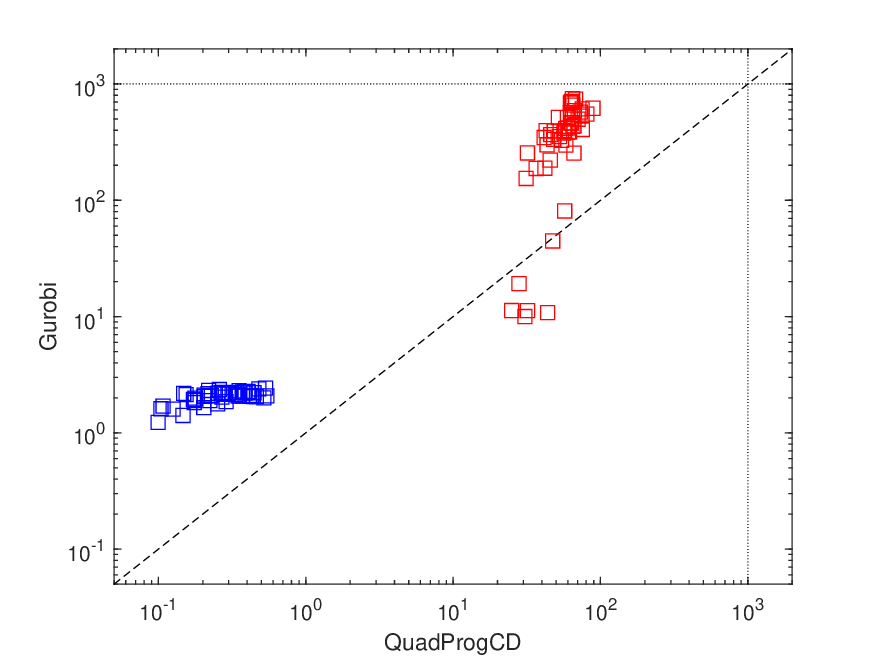}
        \caption{\texttt{QuadProgCD} vs \texttt{Gurobi} on PCQMAX100.}
  \label{subfig6:QG100M}
    \end{subfigure}
    \caption{Wall-clock time comparison for solving the reference value problem on PCQMAX, plotted in log-log scale, where the SDP problem in \texttt{QuadProgCD} is solved  by \texttt{MOSEK}.  }
    \label{fig:refval-pcqmax-mosek}
\end{figure}
For instances  of
 relatively low dimension from the dataset PCQMAX20, PCQMAX50 and PCQMAX100, it is possible to solve the doubly nonnegative relaxation using an interior point method.
The results  with \texttt{MOSEK} as the SDP solver are  summarized in~\Cref{fig:refval-pcqmax-mosek}.   The red
squares stand for instances for which {$\nu_R > \Phi^*(\cF)$}, while the blue squares stand for  instances for which {$\nu_R \leq \Phi^*(\cF)$}.
It can be observed  that \texttt{QuadProgCD} outperforms significantly \texttt{CPLEX} and \texttt{Gurobi} in most of the cases. Moreover, the superior performance of \texttt{QuadProgCD} is more and more remarkable when the dimension increases.

For instances of relatively large dimension from the dataset PCQMAX100 and PCQMAX500, we adopt \texttt{SDPNAL+} as the SDP solver.  The results are displayed in~\Cref{fig:refval-pcqmax-sdpnalp}.   By comparing~\Cref{subfig5:QC100M}-\ref{subfig6:QG100M} with~\Cref{subfig1:QC100}-\ref{subfig2:QC100}, we conclude that \texttt{QuadProgCD} with medium-accuracy SDP solver \texttt{SDPNAL+} can be faster than \texttt{QuadProgCD} with high-accuracy SDP solver \texttt{MOSEK} on the dataset PCQMAX100.
For the dataset PCQMAX500, we set the maximum time limit to be \textbf{3600} seconds. It can be seen from~\Cref{subfig3:QC100} and~\Cref{subfig4:QC100}
that
\texttt{CPLEX} and \texttt{Gurobi} failed on all the 10 instances within the time limit \textbf{3600} seconds. In contrast,  \texttt{QuadProgCD} solved all the 10 instances of dimension 500 all within \textbf{1000} seconds.

\begin{figure}[ht]
    \centering
    \begin{subfigure}[t]{0.45\textwidth}
        \includegraphics[width=\textwidth]{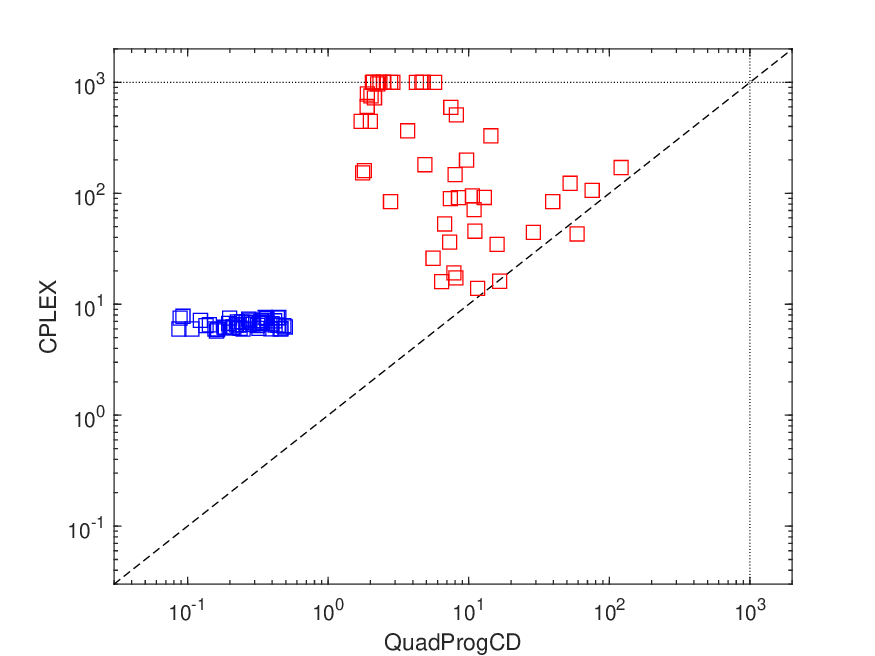}
        \caption{\texttt{QuadProgCD} vs \texttt{CPLEX} on PCQMAX100.}
  \label{subfig1:QC100}
    \end{subfigure}
    \begin{subfigure}[t]{0.45\textwidth}
        \includegraphics[width=\textwidth]{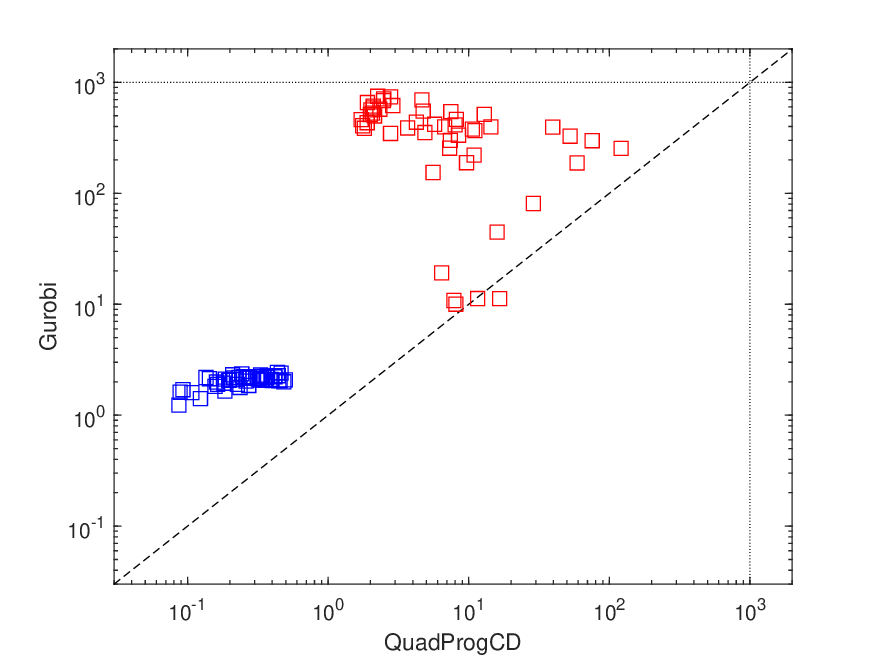}
        \caption{\texttt{QuadProgCD} vs \texttt{Gurobi} on PCQMAX100.}
    \label{subfig2:QC100}
    \end{subfigure}
    \hfill
    \begin{subfigure}[t]{0.45\textwidth}
        \includegraphics[width=\textwidth]{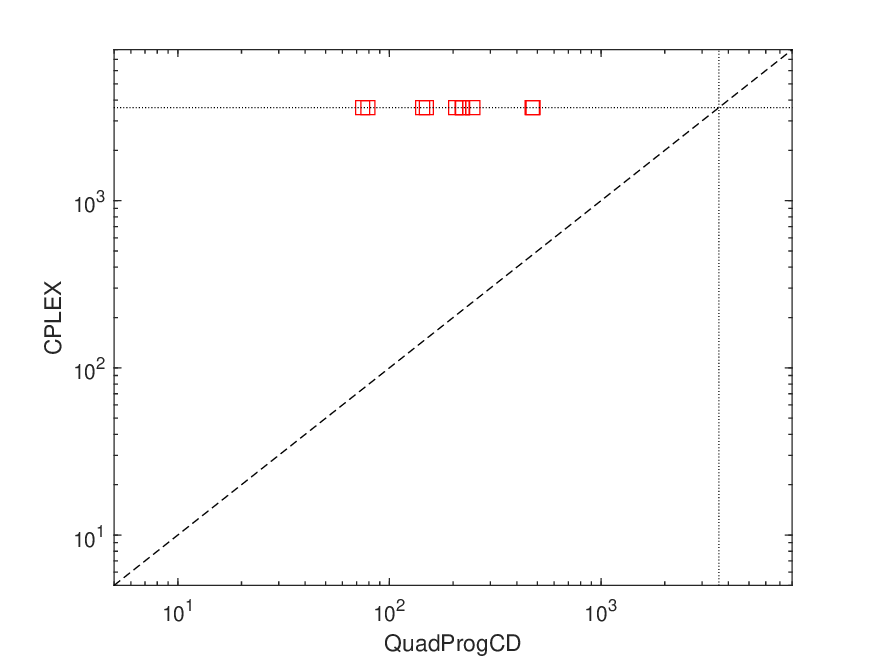}
        \caption{\texttt{QuadProgCD} vs \texttt{CPLEX} on PCQMAX500.}
   \label{subfig3:QC100}
    \end{subfigure}
    \begin{subfigure}[t]{0.45\textwidth}
        \includegraphics[width=\textwidth]{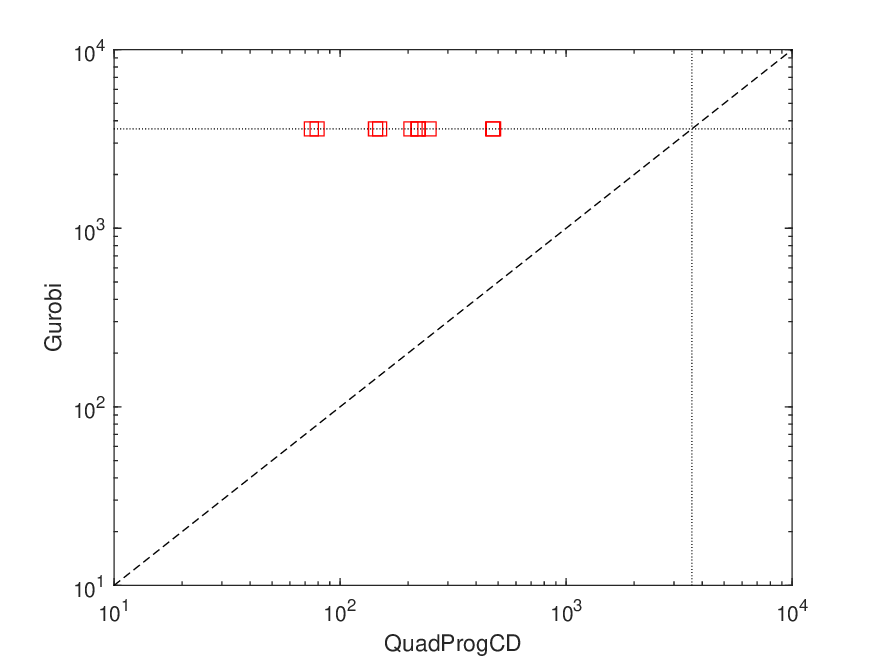}
        \caption{\texttt{QuadProgCD} vs \texttt{Gurobi} on PCQMAX500.}
   \label{subfig4:QC100}
    \end{subfigure}
    \caption{Wall-clock time comparison for solving the reference value problem on PCQMAX, plotted in log-log scale, where the SDP problem in \texttt{QuadProgCD} is solved inexactly by \texttt{SDPNAL+}. }
    \label{fig:refval-pcqmax-sdpnalp}
\end{figure}

\subsubsection{Global solution}\label{subsec:gsosssw}

In this section, we compare  \texttt{QuadProgCD-G} (\Cref{alg:CuP-global})
with three other global solvers  \texttt{CPLEX},  \texttt{Gurobi}
and \texttt{quadprogIP} for globally solving~\eqref{prob:original_max_qp0}.  The relative gap tolerance $\epsilon$ in the termination criteria~\eqref{eq:relgapdef} is set to be $\epsilon=10^{-6}$.
The test is performed on  samples from  the real dataset BioData100 and from the synthetic dataset PCQMAX20, PCQMAX50, PCQMAX100, CQMAX20, CQMAX50 and CQMAX100.

Let us examine the performance of the four algorithms on real dataset  displayed in~\Cref{tab:conv-comparison-biodata}.
We   randomly selected 20 instances from the 317,458 instances of the dataset BioData100.
\texttt{QuadProgCD} reached the gap tolerance $\epsilon=10^{-6}$ in \textbf{hundreds} of seconds for 19 instances.
\texttt{CPLEX} reached the gap tolerance $\epsilon=10^{-6}$ for only 4  instances within the time limit which is set to be \textbf{3600} seconds.
\texttt{Gurobi}  reached the gap tolerance $\epsilon=10^{-6}$ for only 6 instances within the time limit.
\texttt{quadprogIP} reached the gap tolerance $\epsilon=10^{-6}$ for only 6 instances within the time limit.

\begin{table}[htbp]
    \scriptsize
    \centering
    \caption{Comparison of four algorithms on 20 instances  from the real dataset  BioData100 for  reaching the relative gap tolerance $\epsilon=10^{-6}$.
    The shortest wall-clock time in each row is in bold font.
    A hyphen (-)  indicates that a gap  less than
    $10^{-6}$ is reached, and N/A represents that the algorithm fails to obtain an upper or lower bound before the time limit. If a gap less than $10^{-6}$ is not reached before the time limit, then the achieved gap is displayed in the column RelGap.}\label{tab:conv-comparison-biodata}
    \begin{filecontents*}{tables/conv-overall-log-biodata.csv}
Instance,CPLEX RELGAP,CPLEX TIME,GUROBI RELGAP,GUROBI TIME,ALG RELGAP,ALG TIME,QPIP RELGAP,QPIP TIME
257424,-,2269.926274,-,2233.715635,-,258.8795676,N/A,3600
84896,-,705.7146907,-,3045.896477,-,130.1826247,N/A,3600
141783,6.94E-02,3600,1.93E-04,3600,-,203.6604042,N/A,3600
123044,1.00E-01,3600,2.31E-01,3600,-,124.273166,N/A,3600
303256,2.46E-02,3600,1.86E-03,3600,-,47.8580344,-,151.6796147
312988,-,1486.210683,9.54E-05,3600,5.85E-05,3600,N/A,3600
229644,8.78E-02,3600,2.55E-01,3600,-,57.9178167,N/A,3600
259490,1.25E-01,3600,8.64E-04,3600,-,47.2092524,N/A,3600
143339,8.62E-02,3600,-,1759.158039,-,124.8301708,N/A,3600
95618,2.58E-02,3600,1.71E-04,3600,-,334.6436299,N/A,3600
23659,1.45E-01,3600,2.97E-01,3600,-,47.3291375,-,901.6060303
284341,9.83E-02,3600,6.51E-03,3600,-,434.1143218,N/A,3600
79834,1.39E-01,3600,1.06E-03,3600,-,113.4105172,-,36.2283682
88552,3.13E-02,3600,1.69E-01,3600,-,55.557618,-,143.7981923
305293,2.82E-03,3600,8.69E-04,3600,-,772.2975014,-,124.4661383
149864,3.55E-02,3600,6.69E-02,3600,-,54.0157272,N/A,3600
81426,5.47E-02,3600,7.20E-05,3600,-,672.995079,N/A,3600
274863,-,1093.108322,-,3220.615957,-,125.3375384,N/A,3600
82857,3.83E-02,3600,-,2335.981299,-,124.2082958,-,1056.025466
211213,2.65E-02,3600,-,2076.430024,-,204.0090821,N/A,3600
\end{filecontents*}
\csvreader[before reading=\footnotesize\sisetup{round-mode=places, zero-decimal-to-integer, round-precision=2,
    detect-all}
    \begin{adjustbox}{max width=0.98\columnwidth},
    after reading=\end{adjustbox},
    tabular=c c c c c c c c c,
    table head=\toprule
    & \multicolumn{2}{c}{\texttt{QuadProgCD}} & \multicolumn{2}{c}{\texttt{CPLEX}} & \multicolumn{2}{c}{\texttt{Gurobi}} & \multicolumn{2}{c}{\texttt{quadprogIP}} \\
    \cmidrule(l{1pt}r{1pt}){2-3}
    \cmidrule(l{2pt}r{2pt}){4-5}
    \cmidrule(l{2pt}r{2pt}){6-7}
    \cmidrule(l{2pt}r{2pt}){8-9}
    Instance & RelGap & Time & RelGap & Time & RelGap & Time & RelGap & Time \\
    \midrule,
    head to column names,
    before line=\getmin,
    late after last line=\\\bottomrule]{tables/conv-overall-log-biodata.csv}{}{%
        \csvcoli & \myrelgap{\csvcolvi} & \mynum{\csvcolvii} & \myrelgap{\csvcolii} & \mynum{\csvcoliii} & \myrelgap{\csvcoliv} & \mynum{\csvcolv} & \myrelgap{\csvcolviii} & \mynum{\csvcolix}
    }
    \normalsize
\end{table}

If we examine the performance of the four algorithms on the synthetic dataset in~\Cref{tab:conv-comparison-pcqmax}, we observe that \texttt{QuadProgCD} has comparable short running time with \texttt{CPLEX} for low dimensional instances with $n = 20$.
However,  when the dimension increases to $n=100$, \texttt{QuadProgCD} outperforms the other three algorithms by  a significant margin in most of the cases.

\begin{table}[htbp]
    \centering
    \caption{Comparison of four algorithms on synthetic data PCQMAX and CQMAX for reaching the relative gap tolerance $\epsilon=10^{-6}$.
    The shortest wall-clock time in each row is in bold font.
    A hyphen (-)  indicates that a gap  less than $10^{-6}$ is reached, and N/A represents that the algorithm fails to obtain an upper or lower bound before the time limit. If a gap less than $10^{-6}$ is not reached before the time limit, then the achieved gap is displayed in the column RelGap.
    }\label{tab:conv-comparison-pcqmax}
    \begin{filecontents*}{tables/conv-overall-log-pcqmax.csv}
Instance,CPLEXRELGAP,CPLEXTIME,GUROBIRELGAP,GUROBITIME,ALGRELGAP,ALGTIME,QPIPRELGAP,QPIPTIME
pcqmax20-1,-,0.077926636,-,0.1953058,-,0.084298,-,0.9782176
pcqmax20-2,-,0.046879292,-,0.0693194,-,0.0952535,-,0.1289332
pcqmax20-3,-,0.071738482,-,0.0997216,-,0.1666293,-,0.9792528
pcqmax20-4,-,0.062495947,-,0.1123852,-,0.1510153,-,1.0691831
pcqmax20-5,-,0.031250238,-,0.0568511,-,0.033267,-,0.1317045
pcqmax20-6,-,0.031252146,-,0.0884545,-,0.0675604,-,0.4220494
pcqmax20-7,-,0.046873093,-,0.0691847,-,0.3052108,-,0.1306902
pcqmax20-8,-,0.078123569,-,0.1214232,-,0.0899705,-,0.4366649
pcqmax20-9,-,0.046875477,-,0.0797991,-,0.1010845,-,1.1886969
pcqmax20-10,-,0.062499762,-,0.0842106,-,0.0780854,-,1.0827565
pcqmax50-1,-,4.266331434,-,41.4740529,-,1.383947,N/A,3600
pcqmax50-2,-,13.10608816,-,34.784434,-,1.0565707,N/A,3600
pcqmax50-3,-,9.629838228,-,42.3842014,-,2.6206526,N/A,3600
pcqmax50-4,-,11.77099085,-,35.9663166,-,2.5294132,N/A,3600
pcqmax50-5,-,4.794362068,-,40.0823852,-,2.4474242,N/A,3600
pcqmax50-6,-,6.054872513,-,60.7604537,-,2.572411,N/A,3600
pcqmax50-7,-,0.547174692,-,0.803034,-,0.4964934,-,31.8256517
pcqmax50-8,-,4.397613287,-,52.3270276,-,0.9353234,N/A,3600
pcqmax50-9,-,1.147114992,-,1.8473981,-,0.6297216,-,73.3573752
pcqmax50-10,-,1.531248331,-,7.8403943,-,2.2618807,N/A,3600
pcqmax100-1,-,76.38130641,-,337.7678537,-,53.7902512,N/A,3600
pcqmax100-2,-,3380.712233,-,3427.900164,-,96.9907435,N/A,3600
pcqmax100-3,-,94.41522813,-,550.0305605,-,122.8324409,N/A,3600
pcqmax100-4,-,45.88021374,-,298.552144,-,48.3722704,N/A,3600
pcqmax100-5,-,100.632632,-,358.3314005,-,65.9600376,N/A,3600
pcqmax100-6,1.87E-02,3600,-,2576.42592,-,178.4713965,N/A,3600
pcqmax100-7,-,856.3310874,-,802.3725374,-,148.2731482,N/A,3600
pcqmax100-8,-,86.75073242,-,479.7624949,-,77.3603257,N/A,3600
pcqmax100-9,-,2738.927632,-,2261.273073,-,106.6357727,N/A,3600
pcqmax100-10,-,48.93938994,-,367.1972787,-,92.3716982,N/A,3600
\end{filecontents*}
\csvreader[before reading=\footnotesize\sisetup{round-mode=places, zero-decimal-to-integer, round-precision=2,
    detect-all}
    \fontsize{7pt}{7pt}\selectfont
    \begin{adjustbox}{max width=0.98\columnwidth},
    after reading=\end{adjustbox},
    tabular=c c c c c c c c c,
    table head=\toprule
    &   \multicolumn{2}{c}{\texttt{QuadProgCD}} & \multicolumn{2}{c}{\texttt{CPLEX}} & \multicolumn{2}{c}{\texttt{Gurobi}} & \multicolumn{2}{c}{\texttt{quadprogIP}} \\
    \cmidrule(l{1pt}r{2pt}){2-3}
    \cmidrule(l{2pt}r{2pt}){4-5}
    \cmidrule(l{2pt}r{2pt}){6-7}
    \cmidrule(l{2pt}r{2pt}){8-9}
    Instance & RelGap & Time & RelGap & Time & RelGap & Time & RelGap & Time  \\
    \midrule,
    head to column names,
    before line=\getmin,
    after line=\vspace{0.05em},
    late after last line=\\\midrule
    ]{tables/conv-overall-log-pcqmax.csv}{}{%
        \fontsize{7pt}{7pt}\selectfont
        \csvcoli & \myrelgap{\csvcolvi} & \mynum{\csvcolvii} & \myrelgap{\csvcolii} & \mynum{\csvcoliii} & \myrelgap{\csvcoliv} & \mynum{\csvcolv} & \myrelgap{\csvcolviii} & \mynum{\csvcolix}
    }
     \begin{filecontents*}{tables/conv-overall-log-cqmax.csv}
Instance,CPLEXRELGAP,CPLEXTIME,GUROBIRELGAP,GUROBITIME,ALGRELGAP,ALGTIME,QPIPRELGAP,QPIPTIME
cqmax20-1,-,0.070000172,-,134.7665638,-,0.0910087,-,1.518418
cqmax20-2,-,0.030999899,-,445.8354607,-,0.0804868,-,0.3782909
cqmax20-3,-,0.068000555,-,73.4603881,-,0.0842091,-,0.9440549
cqmax20-4,-,0.039957523,-,10.3782213,-,0.1142753,-,0.5284464
cqmax20-5,-,0.027778149,-,231.4460463,-,0.0713189,-,1.3308431
cqmax20-6,-,0.058190346,-,11.3837546,-,0.0765711,-,0.6827231
cqmax20-7,-,0.083062649,-,90.8845141,-,0.1635302,-,1.2049425
cqmax20-8,-,0.043985367,-,1932.812167,-,0.0689742,-,0.5726752
cqmax20-9,-,0.046477795,-,10.624554,-,0.0768817,-,1.3800718
cqmax20-10,-,0.055152416,-,26.9147426,-,0.0703728,-,1.4541159
cqmax50-1,-,1.281252623,N/A,3600,-,1.1410006,N/A,3600
cqmax50-2,-,1.293955564,N/A,3600,-,2.4016705,N/A,3600
cqmax50-3,-,1.859595299,N/A,3600,-,2.5963157,N/A,3600
cqmax50-4,-,1.607500792,N/A,3600,-,1.2743577,N/A,3600
cqmax50-5,-,3.438647985,N/A,3600,-,3.0210261,N/A,3600
cqmax50-6,-,5.582187414,N/A,3600,-,2.8013003,N/A,3600
cqmax50-7,-,3.160559177,N/A,3600,-,2.3849646,N/A,3600
cqmax50-8,-,5.994004488,N/A,3600,-,3.2921643,N/A,3600
cqmax50-9,-,1.966127396,N/A,3600,-,0.9982021,N/A,3600
cqmax50-10,-,2.236162901,N/A,3600,-,0.8352248,N/A,3600
cqmax100-1,-,322.5144405,N/A,3600,-,109.8486405,N/A,3600
cqmax100-2,-,108.9635839,N/A,3600,-,39.8347097,N/A,3600
cqmax100-3,-,253.2178051,N/A,3600,-,107.9149113,N/A,3600
cqmax100-4,-,514.2454553,N/A,3600,-,116.2019592,N/A,3600
cqmax100-5,-,328.379163,N/A,3600,-,30.9002975,N/A,3600
cqmax100-6,-,922.9043643,N/A,3600,-,187.5822146,N/A,3600
cqmax100-7,-,540.5184886,N/A,3600,-,101.4140215,N/A,3600
cqmax100-8,7.18E-05,3600,N/A,3600,-,309.4762202,N/A,3600
cqmax100-9,2.38E-03,3600,N/A,3600,-,196.8236096,N/A,3600
cqmax100-10,-,696.9730282,N/A,3600,-,115.4833173,N/A,3600
\end{filecontents*}
\csvreader[before reading=\footnotesize\sisetup{round-mode=places, zero-decimal-to-integer, round-precision=2,
    detect-all}
    \fontsize{7pt}{7pt}\selectfont
    \begin{adjustbox}{max width=0.98\columnwidth},
    after reading=\end{adjustbox},
    tabular=c c c c c c c c c,
    table head=
    Instance & RelGap & Time & RelGap & Time & RelGap & Time & RelGap & Time  \\
    \midrule,
    head to column names,
    before line=\getmin,
    after line=\vspace{0.05em},
    late after last line=\\\bottomrule
    ]{tables/conv-overall-log-cqmax.csv}{}{%
        \fontsize{7pt}{7pt}\selectfont
        \csvcoli & \myrelgap{\csvcolvi} & \mynum{\csvcolvii} & \myrelgap{\csvcolii} & \mynum{\csvcoliii} & \myrelgap{\csvcoliv} & \mynum{\csvcolv} & \myrelgap{\csvcolviii} & \mynum{\csvcolix}
    }
    \normalsize
\end{table}

Finally we demonstrate the long time behavior of the packages \texttt{QuadProgCD}, \texttt{CPLEX} and \texttt{Gurobi} on six selected instances from the real dataset BioData100.
The plot of time versus relative gap is shown in~\Cref{fig:time-relgap}. It is easy to conclude that \texttt{QuadProgCD} outperforms notably
\texttt{CPLEX} and \texttt{Gurobi} both in terms of the convergence speed and in terms of the  accuracy that can ever be reached.
\begin{figure}[htbp]
    \centering
    \begin{subfigure}[t]{0.49\textwidth}
        \includegraphics[width=\textwidth]{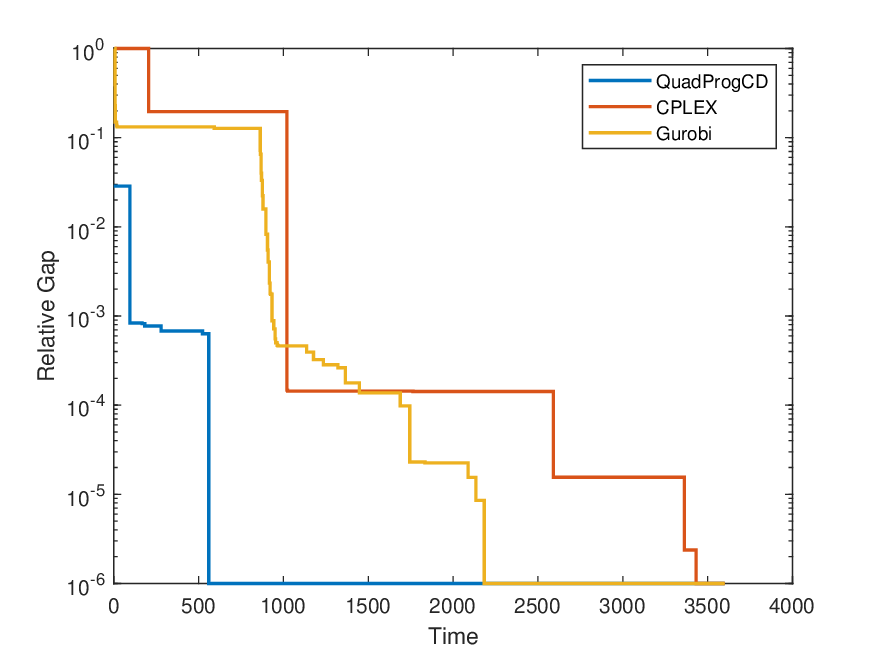}
        \caption{BioData instance 257424.}
        \label{fig:time-relgap-biodata-257424}
    \end{subfigure}
    \begin{subfigure}[t]{0.49\textwidth}
        \includegraphics[width=\textwidth]{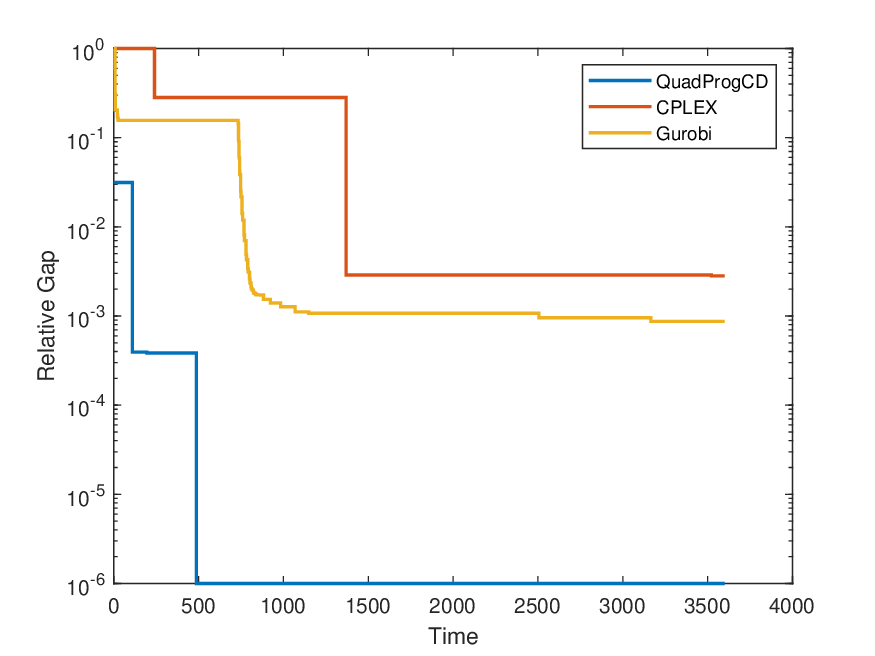}
        \caption{BioData instance 305293.}
        \label{fig:time-relgap-biodata-305293}
    \end{subfigure}
    \hfill
    \begin{subfigure}[t]{0.49\textwidth}
        \includegraphics[width=\textwidth]{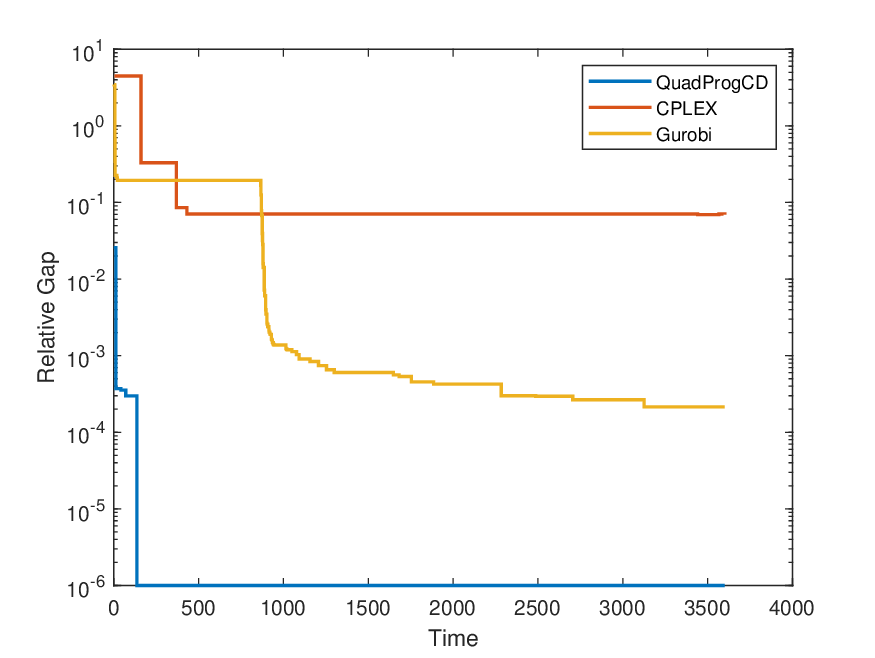}
        \caption{BioData instance 141783.}
        \label{fig:time-relgap-biodata-141783}
    \end{subfigure}
    \begin{subfigure}[t]{0.49\textwidth}
        \includegraphics[width=\textwidth]{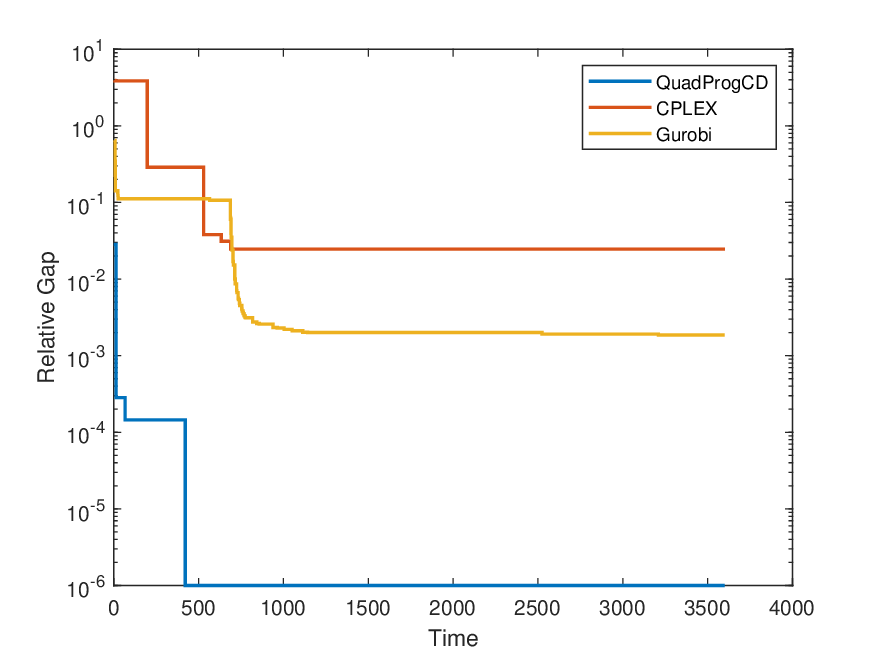}
        \caption{BioData instance 303256.}
        \label{fig:time-relgap-biodata-303256}
    \end{subfigure}
    \hfill
    \begin{subfigure}[t]{0.49\textwidth}
        \includegraphics[width=\textwidth]{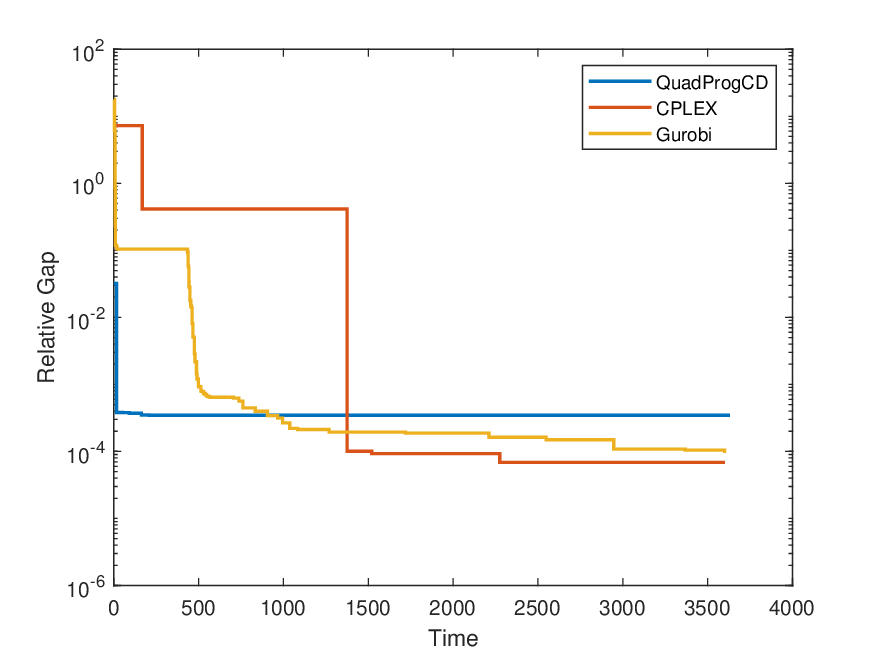}
        \caption{BioData instance 312988.}
        \label{fig:time-relgap-biodata-312988}
    \end{subfigure}
    \begin{subfigure}[t]{0.49\textwidth}
        \includegraphics[width=\textwidth]{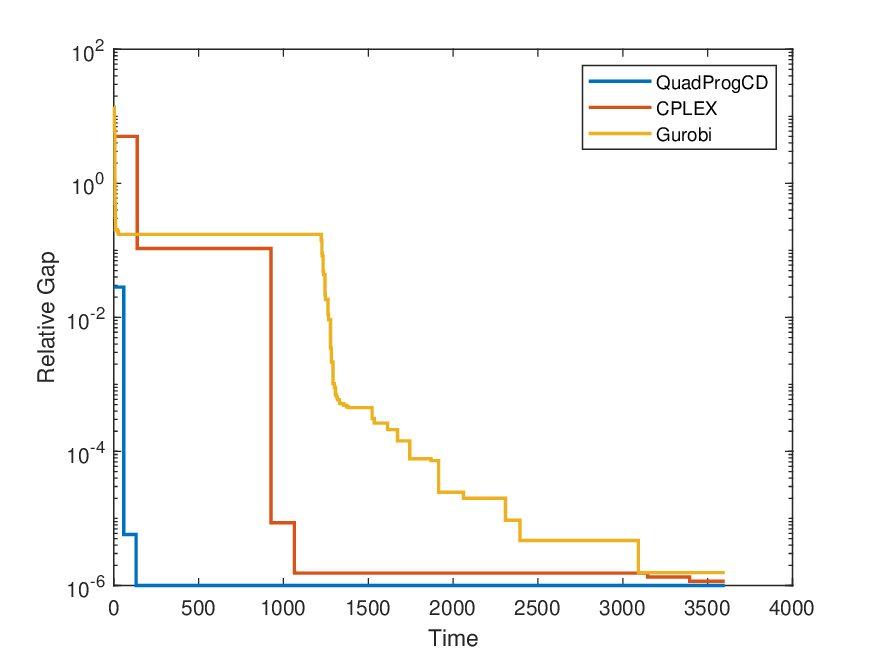}
        \caption{BioData instance 84896.}
        \label{fig:time-relgap-biodata-84896}
    \end{subfigure}
    \caption{Convergence behaviour of three algorithms, with relative gap ${\left(\bar v-\underline v\right)}/{|\underline v|}$ plotted in log scale.}
    \label{fig:time-relgap}
\end{figure}

More detailed information about \texttt{QuadProgCD}, including the number of iteration and the percentage of wall-clock running time spent on different sub-tasks, is summarized in~\Cref{sec:supp-numexp}.

\subsubsection{Comparison between Konno's cut and DNN cut}\label{subsec:comakc}

Recall that~\Cref{thm:phi-val-comp} states:
\begin{equation}
    \label{eq:phival-ineq}
    \bar\phi^K_{\tau,\theta}  \geq \bar\phi^L_{\tau,\theta}\geq \bar\phi^D_{\tau,\theta}\geq  \phi^*_{\tau,\theta},\enspace \forall \theta \geq 0.
\end{equation}
We calculate the relative improvement of the LP relaxation bound and the DNN relaxation bound, which are defined by
\[ \text{RI}_L(\theta_K): = \frac{\bar \phi^K_{\tau, \theta_K} - \bar \phi^L_{\tau, \theta_K}}{\bar \phi^K_{\tau, \theta_K}}, \quad \text{RI}(\theta_K) := \frac{\bar \phi^K_{\tau, \theta_K} - \bar{\phi}^D_{\tau, \theta_K}}{\bar \phi^K_{\tau, \theta_K}}. \]
Note that \eqref{eq:phival-ineq} implies that
\begin{align}\label{a:RI}0\leq \text{RI}_L(\theta_K)\leq \text{RI}(\theta_K) \leq 1.
\end{align}
Let us verify~\eqref{a:RI} with numerical experiments.
We randomly selected 20 instances from the real dataset BioData100 and run one iteration of \texttt{QuadProgCD}.
We display the values of $(\text{RI}_L(\theta_K), \text{RI}(\theta_K))$ in~\Cref{tab:konno-cut-phival-comparison} for those instances with nonempty feasible region  after adding Konno's cut.
\begin{center}
\begin{table}[htbp]%
    \footnotesize
    \centering
    \caption{Comparison of different relaxation bounds.}
    \label{tab:konno-cut-phival-comparison}
    \sisetup{round-mode=places, round-precision=4,
    detect-all}
    \begin{filecontents*}{tables/phival-log.csv}
Instance Index,$\text{RI}_L(\theta_K)$,$\text{RI}(\theta_K)$
465,0.0207,0.0208
1582,0.0143,0.0152
4455,0.0163,0.0164
4571,0.0163,0.0169
6664,0.0096,0.0156
8946,0.0172,0.0172
9217,0.0112,0.0142
10824,0.0182,0.0198
11114,0.0127,0.0142
11669,0.0142,0.0153
\end{filecontents*}
\csvreader[tabular=c c c,
    head=false,
    table head=\toprule,
    late after line=\\,
    late after first line=\\\midrule,
    table foot=\bottomrule,]{tables/phival-log.csv}{}{\csvcoli & \csvcolii & \csvcoliii}
    \normalsize
\end{table}
\end{center}
To demonstrate the  impact on generation of deeper cuts, we show in~\Cref{tab:two-time-deep-konno-cut-ar-rate}  the acceptance/rejection rate of the deepened cut $\theta=\theta_K/2$ under the criteria $\bar{\phi}^L_{\tau, \theta}\leq \nu_R-\delta$ and $\bar{\phi}^D_{\tau, \theta}\leq \nu_R-\delta$ over the dataset PCQMAX and CQMAX. The second column displays the number of instances with nonempty feasible region after adding Konno's cut in the first iteration of \texttt{QuadProgCD}. The third column displays the percentage of the instances which satisfy  $\bar{\phi}^L_{\tau, \theta}\leq \nu_R-\delta$ and the fourth column displays the percentage of the instances which satisfy
$ \bar{\phi}^D_{\tau, \theta} \leq \nu_R-\delta$.

\begin{table}[htbp]%
    \footnotesize
    \centering
    \caption{Acceptance/rejection rates of {the} deepened cut $\theta := \theta_K/2$ using different criteria. Here, a valid instance refers to an instance which is feasible after adding a Konno's cut.
    }
    \label{tab:two-time-deep-konno-cut-ar-rate}
    \begin{filecontents*}{tables/deepened-cut-accept-rate.csv}
Dataset,Valid instances,Valid LR Cut,Valid DNNR Cut
PCQMAX20,11,100.00%,100.00%
PCQMAX50,62,98.39%,100.00%
PCQMAX100,97,89.69%,85.57%
CQMAX20,19,89.47%,100.00%
CQMAX50,57,82.46%,91.23%
CQMAX100,75,64.00%,82.67%
\end{filecontents*}
\csvreader[tabular=c c c c,
    head=false,
    table head=\toprule,
    late after line=\\,
    late after first line=\\\midrule,
    table foot=\bottomrule,
    respect all]{tables/deepened-cut-accept-rate.csv}{}{\csvcoli & \csvcolii & \csvcoliii & \csvcoliv}
    \normalsize
\end{table}

\begin{remark}\label{rem:thm2-contrad}
In \Cref{tab:two-time-deep-konno-cut-ar-rate}, it can be observed that the percentage of valid LR cuts is higher than that of valid DNNR cuts for the dataset PCQMAX100, which appears to contradict the inequality $\bar{\phi}_{\tau,\theta}^D\leq \bar{\phi}^L_{\tau, \theta}$ established in~\Cref{thm:phi-val-comp}.
This is due to the fact that numerical solvers cannot return the exact values of $\bar{\phi}_{\tau,\theta}^D$.
To this end, we used the upper bound in \eqref{a:inexcs} instead of the exact value of $\bar{\phi}_{\tau,\theta}^D$ in our computation.
This upper bound of $\bar{\phi}^D_{\tau, \theta}$ may occasionally exceed the value of $\bar{\phi}^L_{\tau, \theta}$, resulting in LR cuts exhibiting better performance than DNNR cuts.
\end{remark}

Finally, we compare the performance of  Option I (Konno's cut)  with Option II (DNN cut).
We randomly selected 20 instances from the real dataset Biodata100
and set $\eta=1/10$ in this experiment. The termination criterion is the reach of a relative gap $\epsilon=10^{-6}$. The time limit is set to be 1000 seconds.
The result is summarized in~\Cref{tab:cut-walltime-comp}.
It can be seen that Option II is superior than Option I in most of the cases, which confirms the effectiveness of the proposed strategy of deepening Konno's cut through the LP and DNN relaxation.

\begin{table}[htb]
    \scriptsize
    \centering
    \caption{Comparison of Option I and with Option II in \texttt{QuadProgCD-R}.
    The time limit is set to 1000 seconds.
    The shortest wall-clock time in each row is in bold font.}
    \label{tab:cut-walltime-comp}
    \sisetup{round-mode=places, round-precision=2,
    detect-all}
    \begin{filecontents*}{tables/cut-walltime-comp.csv}
Instance,Deepened Cut Time,Konno Cut Time
252565,18.535313,32.9890385
280796,38.2098688,103.4470225
283147,40.6484552,474.7481889
196032,189.806749,288.479259
30238,21.8697627,35.7519193
86335,17.8006992,37.3972965
169534,21.1659221,107.6650842
299116,18.4597762,53.0808886
48861,1000,1000
150467,25.720523,40.6169392
248087,26.1445588,86.9992215
43985,24.2957334,105.7050299
130746,296.9249361,1000.0001
283879,347.4553617,1000.0001
245585,25.9726438,465.5396548
203280,67.0407985,241.201649
11071,21.9902779,58.8187933
263231,1000,1000
289538,24.4525654,39.2224097
210408,54.2621416,34.0716763
\end{filecontents*}
\csvreader[tabular=c c c,
    table head=\toprule
    Instance & Option I & Option II \\
    \midrule,
    head to column names,
    before line=\getmintwo,
    late after last line=\\\bottomrule]{tables/cut-walltime-comp.csv}{}{%
        \csvcoli   & \mynum{\csvcoliii} & \mynum{ \csvcolii}
    }
    \normalsize
\end{table}

\section{Conclusion}\label{sec:conclusion}
In this paper, we develop an efficient global solver for  the concave QP problem~\eqref{prob:original_max_qp0}, which is known to be an NP-hard problem. The concave QP problem finds a recent application in computational biology, where the associated reference value problem needs to be solved for a huge number (317,584) of instances of dimension 78. The existing QP solvers such as \texttt{CPLEX} or \texttt{Gurobi} seem to require years of computational time for this particular task. By revisiting the classical cutting plane method proposed by Tuy and Konno, we propose to construct deeper cuts through the DNN relaxation. We prove that the  DNN relaxation is equivalent to the Shor relaxation of an equivalent QCQP problem. This allows us to write down an SDP formulation of the DNN relaxation which satisfies Slater's condition, which is crucial for the robustness of applying the existing SDP solvers to compute the DNN bound. We also provide the explicit formula for obtaining a valid upper bound from any approximate primal and dual solution. The proposed algorithm is tested on a variety of real and synthetic instances and outperforms  \texttt{CPLEX}, \texttt{Gurobi} and \texttt{quadprogIP} in most of the cases.  In particular,
our method successfully solved all the 317,584 instances within 3 days on the HKU HPC cluster using 32 processors. Moreover, our algorithm demonstrates superior performance on large-scale instances of dimension up to {819}.

\begin{acknowledgements}\label{sec:ack}

We are deeply grateful to the anonymous referees for their careful reading and their many constructive comments, which helped us improve this paper.
The authors would like to thank Prof. Stephen S.-T.  Yau  for introducing the application of concave QP in computational biology and providing the dataset. We also thank Ms. Xinyuan Zhang for the proofreading work.
Parts of computations were performed using research computing facilities offered by Information Technology Services, the University of Hong Kong.

\end{acknowledgements}

\noindent\small{\textbf{Funding}
The work of the authors was supported by NSFC Young Scientist Fund grant 12001458 and Hong Kong
Research Grants Council grant 17317122.
}

\vspace{1em}

\noindent\small{\textbf{Author contributions}
All authors contributed to the study conception and design. Implementation, data generation and numerical experiments were performed by Tianyou Zeng and  Yuchen Lou. The first draft of the manuscript was written by Zheng Qu and all authors commented on previous versions of the manuscript. All authors read and approved the final manuscript.
}

\vspace{1em}

\noindent\small{\textbf{Data Availability Statement}
The datasets generated during and/or analysed during the current study are available in the repository \url{https://github.com/tianyouzeng/QuadProgCD}.
}

\vspace{1em}

\noindent\small{\textbf{Code Availability}
The full code was reviewed and is available on GitHub at the following address: \url{https://github.com/tianyouzeng/QuadProgCD}.
}

\vspace{1em}

\noindent\small{\textbf{Conflict of interest}
The authors have no relevant financial or non-financial interests to disclose.
}

\bibliographystyle{spmpsci}      
\bibliography{references.bib}   

\begin{thebibliography}{10}
\providecommand{\url}[1]{{#1}}
\providecommand{\urlprefix}{URL }
\expandafter\ifx\csname urlstyle\endcsname\relax
  \providecommand{\doi}[1]{DOI~\discretionary{}{}{}#1}\else
  \providecommand{\doi}{DOI~\discretionary{}{}{}\begingroup \urlstyle{rm}\Url}\fi

\bibitem{baesens2003benchmarking}
Baesens, B., Van~Gestel, T., Viaene, S., Stepanova, M., Suykens, J., Vanthienen, J.: Benchmarking state-of-the-art classification algorithms for credit scoring.
\newblock Journal of the Operational Research Society \textbf{54}(6), 627--635 (2003)

\bibitem{Burer2012}
Burer, S.: Copositive programming.
\newblock In: M.F. Anjos, J.B. Lasserre (eds.) Handbook on Semidefinite, Conic and Polynomial Optimization, pp. 201--218. Springer New York, NY (2012)

\bibitem{burer2008finite}
Burer, S., Vandenbussche, D.: A finite branch-and-bound algorithm for nonconvex quadratic programming via semidefinite relaxations.
\newblock Mathematical Programming \textbf{113}(2), 259--282 (2008)

\bibitem{burkard1998quadratic}
Burkard, R.E., Cela, E., Pardalos, P.M., Pitsoulis, L.S.: The quadratic assignment problem.
\newblock In: D.Z. Du, P.M. Pardalos (eds.) Handbook of Combinatorial Optimization, pp. 1713--1809. Springer (1998)

\bibitem{cabot1970solving}
Cabot, A.V., Francis, R.L.: Solving certain nonconvex quadratic minimization problems by ranking the extreme points.
\newblock Operations Research \textbf{18}(1), 82--86 (1970)

\bibitem{chen2012globally}
Chen, J., Burer, S.: Globally solving nonconvex quadratic programming problems via completely positive programming.
\newblock Mathematical Programming Computation \textbf{4}(1), 33--52 (2012)

\bibitem{fung2003disputed}
Fung, G.: The disputed federalist papers: {SVM} feature selection via concave minimization.
\newblock In: Proceedings of the 2003 Conference on Diversity in Computing, pp. 42--46 (2003)

\bibitem{gondzio2021global}
Gondzio, J., Y{\i}ld{\i}r{\i}m, E.A.: Global solutions of nonconvex standard quadratic programs via mixed integer linear programming reformulations.
\newblock Journal of Global Optimization \textbf{81}(2), 293--321 (2021)

\bibitem{guisewite1990minimum}
Guisewite, G.M., Pardalos, P.M.: Minimum concave-cost network flow problems: Applications, complexity, and algorithms.
\newblock Annals of Operations Research \textbf{25}(1), 75--99 (1990)

\bibitem{HladikMilan20}
Hlad{\'i}k, M., Hartman, D.: Maximization of a convex quadratic form on a polytope: Factorization and the chebyshev norm bounds.
\newblock In: H.A. Le~Thi, H.M. Le, T.~Pham~Dinh (eds.) Optimization of Complex Systems: Theory, Models, Algorithms and Applications, pp. 119--127. Springer International Publishing, Cham (2020)

\bibitem{HladikMilanZamani21}
Hlad{\'\i}k, M., Hartman, D., Zamani, M.: Maximization of a {PSD} quadratic form and factorization.
\newblock Optimization Letters \textbf{15}(7), 2515--2528 (2021)

\bibitem{horst2013handbook}
Horst, R., Pardalos, P.M.: Handbook of Global Optimization, vol.~2.
\newblock Springer Science \& Business Media (2013)

\bibitem{horst2013global}
Horst, R., Tuy, H.: Global Optimization: Deterministic Approaches.
\newblock Springer Science \& Business Media (2013)

\bibitem{jacobsen1981convergence}
Jacobsen, S.E.: Convergence of a {T}uy-type algorithm for concave minimization subject to linear inequality constraints.
\newblock Applied Mathematics and Optimization \textbf{7}, 1--9 (1981)

\bibitem{JIAO2021559}
Jiao, X., Pei, S., Sun, Z., Kang, J., Yau, S.S.T.: Determination of the nucleotide or amino acid composition of genome or protein sequences by using natural vector method and convex hull principle.
\newblock Fundamental Research \textbf{1}(5), 559--564 (2021)

\bibitem{KimKojimaToh16}
Kim, S., Kojima, M., Toh, K.C.: A {L}agrangian--{DNN} relaxation: a fast method for computing tight lower bounds for a class of quadratic optimization problems.
\newblock Mathematical Programming \textbf{156}(1), 161--187 (2016)

\bibitem{konno1976cutting}
Konno, H.: A cutting plane algorithm for solving bilinear programs.
\newblock Mathematical Programming \textbf{11}(1), 14--27 (1976)

\bibitem{konno1976maximization}
Konno, H.: Maximization of a convex quadratic function under linear constraints.
\newblock Mathematical Programming \textbf{11}(1), 117--127 (1976)

\bibitem{liuzzi2022computational}
Liuzzi, G., Locatelli, M., Piccialli, V.: A computational study on {QP} problems with general linear constraints.
\newblock Optimization Letters \textbf{16}(6), 1633--1647 (2022)

\bibitem{luo2010semidefinite}
Luo, Z.Q., Ma, W.K., So, A.M.C., Ye, Y., Zhang, S.: Semidefinite relaxation of quadratic optimization problems.
\newblock IEEE Signal Processing Magazine \textbf{27}(3), 20--34 (2010)

\bibitem{mangasarian1995breast}
Mangasarian, O.L., Street, W.N., Wolberg, W.H.: Breast cancer diagnosis and prognosis via linear programming.
\newblock Operations Research \textbf{43}(4), 570--577 (1995)

\bibitem{mccormick1976computability}
McCormick, G.P.: Computability of global solutions to factorable nonconvex programs: Part {I}—{Convex} underestimating problems.
\newblock Mathematical Programming \textbf{10}(1), 147--175 (1976)

\bibitem{momoh1995economic}
Momoh, J., Dias, L., Guo, S., Adapa, R.: Economic operation and planning of multi-area interconnected power systems.
\newblock IEEE Transactions on Power Systems \textbf{10}(2), 1044--1053 (1995)

\bibitem{nesterov1998semidefinite}
Nesterov, Y.: Semidefinite relaxation and nonconvex quadratic optimization.
\newblock Optimization Methods and Software \textbf{9}(1-3), 141--160 (1998)

\bibitem{nohra2021spectral}
Nohra, C.J., Raghunathan, A.U., Sahinidis, N.: Spectral relaxations and branching strategies for global optimization of mixed-integer quadratic programs.
\newblock SIAM Journal on Optimization \textbf{31}(1), 142--171 (2021)

\bibitem{pardalos1986methods}
Pardalos, P.M., Rosen, J.B.: Methods for global concave minimization: a bibliographic survey.
\newblock SIAM Review \textbf{28}(3), 367--379 (1986)

\bibitem{pardalos1991quadratic}
Pardalos, P.M., Vavasis, S.A.: Quadratic programming with one negative eigenvalue is {NP}-hard.
\newblock Journal of Global Optimization \textbf{1}(1), 15--22 (1991)

\bibitem{Pataki19}
Pataki, G.: Characterizing bad semidefinite programs: Normal forms and short proofs.
\newblock SIAM Review \textbf{61}(4), 839--859 (2019)

\bibitem{Porembski01}
Porembski, M.: Finitely convergent cutting planes for concave minimization.
\newblock Journal of Global Optimization \textbf{20}(2), 109--132 (2001)

\bibitem{sun2020sdpnal+}
Sun, D., Toh, K.C., Yuan, Y., Zhao, X.Y.: {SDPNAL}+: {A} {Matlab} software for semidefinite programming with bound constraints (version 1.0).
\newblock Optimization Methods and Software \textbf{35}(1), 87--115 (2020)

\bibitem{telli2020successive}
Telli, M., Bentobache, M., Mokhtari, A.: A successive linear approximation algorithm for the global minimization of a concave quadratic program.
\newblock Computational and Applied Mathematics \textbf{39}(4), 1--28 (2020)

\bibitem{TohToddTutuSDPT3}
Toh, K.C., Todd, M.J., Tütüncü, R.H.: {SDPT3} — {A} {Matlab} software package for semidefinite programming, version 1.3.
\newblock Optimization Methods and Software \textbf{11}(1-4), 545--581 (1999)

\bibitem{Tuy1964concave}
Tuy, H.: Concave programming under linear constraints.
\newblock Soviet Math. \textbf{5}, 1437--1440 (1964)

\bibitem{tuy2016nonconvex}
Tuy, H.: Nonconvex quadratic programming.
\newblock In: Convex Analysis and Global Optimization, pp. 337--390. Springer (2016)

\bibitem{VanTuy80}
Van~Thoai, N., Tuy, H.: Convergent algorithms for minimizing a concave function.
\newblock Mathematics of Operations Research \textbf{5}(4), 556--566 (1980)

\bibitem{WenGoldfardYin10}
Wen, Z., Goldfarb, D., Yin, W.: Alternating direction augmented {Lagrangian} methods for semidefinite programming.
\newblock Mathematical Programming Computation \textbf{2}(3), 203--230 (2010)

\bibitem{xia2020globally}
Xia, W., Vera, J.C., Zuluaga, L.F.: Globally solving nonconvex quadratic programs via linear integer programming techniques.
\newblock INFORMS Journal on Computing \textbf{32}(1), 40--56 (2020)

\bibitem{zamani2019new}
Zamani, M.: A new algorithm for concave quadratic programming.
\newblock Journal of Global Optimization \textbf{75}(3), 655--681 (2019)

\bibitem{zhao2020new}
Zhao, R., Pei, S., Yau, S.S.T.: New genome sequence detection via natural vector convex hull method.
\newblock IEEE/ACM Transactions on Computational Biology and Bioinformatics  (2020)

\end{thebibliography}

\appendix

\section{Search of KKT Vertex}
\label{sec:KKT}

Consider problem \eqref{prob:original_max_qp0}.
We first define the bilinear function $\Psi: \R^{{k}}\times \R^{{k}}\rightarrow \R$ as follows:
\begin{equation}\label{eq:bilinear-form-psi}
    \Psi(y, \tilde y):= y^\top Q \tilde y+ d^\top y+ d^\top  \tilde y + \nu,\enspace \forall y, \tilde y\in \R^k.
\end{equation}
By simple algebra,
\begin{align}\label{a:psiphi}
    \Psi(y, \tilde y)=\Phi(y)+\frac{1}{2}\<\nabla \Phi(y),  \tilde  y-y>=\Phi( \tilde y)+\frac{1}{2}\<\nabla \Phi( \tilde y), y- \tilde y>,\enspace \forall y, \tilde y \in \R^k.
\end{align}
It then follows that
\begin{align}\label{a:phisdfs}
    \Psi(y, \tilde y)=\frac{\Phi(y)+\Phi(\tilde  y)}{2}-{\frac{1}{4}\<\nabla\Phi(y)-\nabla\Phi( \tilde y), y- \tilde y>}
    \leq \frac{\Phi(y)+\Phi(\tilde y)}{2},\enspace \forall y, \tilde y \in \R^k,
\end{align}
where the inequality follows from the convexity of  $\Phi$. Based on~\eqref{a:phisdfs}, we immediately have the following result; see~\cite[Thm 2.2]{konno1976maximization}.
\begin{theorem}[\cite{konno1976maximization}]\label{thm:konno1}
    For any subset $\cY \subset \R^k$, we have
    $$
    \max_{y\in \cY} \Phi(y)=\max_{y\in \cY,  \tilde y\in \cY} \Psi(y,\tilde  y).
    $$
\end{theorem}\noindent
\Cref{thm:konno1} implies that~\eqref{prob:original_max_qp0} is equivalent to the following bilinear program.
\begin{equation}
    \label{prob:bilinear2}
    \begin{aligned}
        \max  &  \qquad \Psi(y, \tilde y) \\
        ~ \textrm{s.t.} &\qquad  y\in \cF,\quad \tilde  y \in \cF.
    \end{aligned}
\end{equation}
In~\cite{konno1976cutting},
Konno proposed a \textit{mountain climbing} algorithm for~\eqref{prob:bilinear2}, which corresponds to alternatively maximizing over $y$ and $\tilde y$.
We define $V(\cF)$ as the set of vertices of $\mathcal{F}$.

\begin{algorithm}[ht]
    \caption{Mountain Climbing Algorithm~\cite{konno1976cutting} (\texttt{Search\_of\_KKT\_Point})}
    \begin{algorithmic}[1]
        \Require {Initial point  $y^{0}\in \R^k$}, feasible region $\cF:=\{y\in \R^k: Fy\leq w, y\geq 0\}$, bilinear function $\Psi(y, \tilde y):= y^\top Q \tilde y+ d^\top y+ d^\top  \tilde y + \nu$.
        \State $\bar y\leftarrow y^{0}$
        \While{true}
        \State ${\tilde y \in  V(\cF)}\cap \argmax \{ \Psi({ y, \bar y}):  { y\in \cF}\} $ \label{alg:tildey}
        \If {$\Psi\left({\tilde y, \tilde y}\right)=\Psi\left({\bar y, \bar y}\right)$}
        \State break;
        \EndIf
        \State $\bar y\leftarrow \tilde y$
        \EndWhile
        \Ensure ${\bar y}\in V\left(\cF\right)$ such that
        \begin{align}\label{a:sdfswerwea}
            {\bar y}\in  V(\cF)\cap \argmax \{ \Psi({\bar y, y}):  {y}\in \cF\}.
        \end{align}
    \end{algorithmic}
    \label{alg:MC}
\end{algorithm}

Since $V(\cF)$ is a finite set, ~\Cref{alg:MC} terminates in finitely many iterations.
The main computational step of~\Cref{alg:MC} is line 3, where one needs to solve an LP problem.
Also note that the starting point {$y^0$} is not required to be a feasible solution.
In view of~\eqref{a:psiphi}, \eqref{a:sdfswerwea} is equivalent to~\eqref{eq:sdfsew1}.
Therefore, the output {$\bar y$} of~\Cref{alg:MC}  is a KKT vertex of problem~\eqref{prob:original_max_qp0}.

\section{Change of Coordinate}\label{sec:mp}

In this section, we discuss the reformulation of the original QP~\eqref{prob:original_max_qp0} so that it satisfies the assumptions~\eqref{esdf}.
Let $\bar{y}$ be any KKT vertex of the original problem~\eqref{prob:original_max_qp0}.
Without loss of generality we assume $\Phi(\bar{y}) < \nu_R$, otherwise the reference value problem is solved.

We lift the original problem~\eqref{prob:original_max_qp0} into $\R^n$ as in~\eqref{prob:original_max_qps}, which can be written as
\begin{equation}\label{eq:lifted}
\begin{aligned}
          \max_{x\in \R^n}  &  \qquad   x^\top H x + 2 p^\top x + \nu  \\
        ~ \textrm{s.t.} & \qquad  Ax = b\\ &  \qquad  x \geq 0.
    \end{aligned}
\end{equation}
where $H:=\begin{pmatrix}
            Q & \bzero  \\
            \bzero^\top & 0
        \end{pmatrix}$, $p:=\begin{pmatrix} d \\ \bzero\end{pmatrix}$, $A:=\begin{pmatrix} F^\top & \bI\end{pmatrix} $ and $b=w$.
Since $\bar{y}$ is a KKT vertex of~\eqref{prob:original_max_qp0},  $\bar{x} := \begin{pmatrix}\bar{y}\\ w - F^\top \bar{y}\end{pmatrix}$ is a KKT point of~\eqref{eq:lifted} and also a vertex of the feasible region of~\eqref{eq:lifted}.
Let $c=H\bar x+p$.
Then $\bar x$ is a basic feasible solution as well as an optimal solution of the following LP problem:
\begin{equation}
    \label{prob:localLP}
    \begin{aligned}
      \max_{x\in \R^n} \quad & c^\top x \\
        \textrm{s.t.}
        \quad &  Ax=b \\
        \quad & x \geq 0.
    \end{aligned}
\end{equation}
Let $\{i_1,\ldots,i_m\}$ be the indices of basic variables of $\bar x$ and $\{j_1,\ldots,j_{n-m}\}$ be the indices of nonbasic variables.
Note that the choice of basic variables is not unique if $\bar x$ is a degenerate vertex.
Let  $B:=\begin{pmatrix}A_{i_1} & \cdots & A_{i_m}
\end{pmatrix}$ be the basis matrix and $N:=\begin{pmatrix}A_{j_1} & \cdots & A_{j_{n-m}}
\end{pmatrix}$ be the nonbasis matrix.
Following the standard notations in linear programming, for any vector $x\in\R^n$ we denote by $x_B$ the subvector $(x_{i_1},\ldots,x_{i_m})\in \R^{m}$ and by $x_N$ the subvector $(x_{j_1},\ldots,x_{j_{n-m}})\in\R^{n-m}$.
Further, we partition accordingly the matrix $H$ and vector $p$ into blocks:
\begin{equation}\label{eq:psdfsdfwe}
    \begin{pmatrix}
        H_{BB} & H_{BN} \\ H^\top_{BN} & H_{NN}
    \end{pmatrix}, \quad \begin{pmatrix}
        p_B \\
            p_N
    \end{pmatrix}.
\end{equation}
By the feasibility and optimality of $\bar x$ as a solution of~\eqref{prob:localLP},  there exists a basis matrix $B$ associated with $\bar x$ such that
\begin{subequations}\label{sub:qeset}
\begin{align}
&\bar x_B=B^{-1}b\geq 0, \enspace \bar x_N=0, \\
&c^\top_B B^{-1}N- c_N^\top \geq 0. \label{a:sub2ebf}
\end{align}
\end{subequations}
Given such choice of basis $B$, for any feasible solution $x$ of~\eqref{eq:lifted}, we have
\begin{equation}\label{eq:sederrr}
\begin{aligned}
&x^\top Hx+2p^\top x +\nu \\& =
\begin{pmatrix} x_B^\top & x_N^\top
\end{pmatrix} \begin{pmatrix}
        H_{BB} & H_{BN} \\ H^\top_{BN} & H_{NN}
    \end{pmatrix} \begin{pmatrix} x_B \\ x_N
\end{pmatrix} +2 \begin{pmatrix} p_B^\top & p_N^\top
\end{pmatrix} \begin{pmatrix} x_B \\  x_N
\end{pmatrix}+\nu  \\
&=\begin{pmatrix} \left(B^{-1}(b-Nx_N)\right)^\top & x_N^\top
\end{pmatrix} \begin{pmatrix}
        H_{BB} & H_{BN} \\ H^\top_{BN} & H_{NN}
    \end{pmatrix} \begin{pmatrix} B^{-1}(b-Nx_N) \\ x_N
\end{pmatrix} \\
& \qquad +2 \begin{pmatrix} p_B^\top & p_N^\top
\end{pmatrix} \begin{pmatrix} B^{-1}(b-Nx_N) \\  x_N
\end{pmatrix}+\nu\\
&=x_N^\top \tilde Q x_N +2 \tilde d^\top x_N+\tilde \nu
\end{aligned}
\end{equation}
where
\begin{equation}\label{eq:symb_def}
    \begin{aligned}
        & \tilde Q = H_{NN} + \left(B^{-1} N\right)^\top H_{BB} B^{-1} N -   \left(B^{-1} N\right)^\top H_{BN}-H^\top_{BN} B^{-1} N,\\
        & \tilde d = p_N+H^\top_{BN}B^{-1} b - \left(B^{-1} N\right)^\top p_B -  \left(B^{-1} N\right)^\top H_{BB}B^{-1} b ,\\
        & \tilde \nu=(B^{-1}b)^\top H_{BB} (B^{-1}b)+2p_B^\top (B^{-1}b)+\nu .
    \end{aligned}
\end{equation}
Then~\eqref{eq:lifted} can be rewritten equivalently as follows:
\begin{equation}\label{eq:liftedr2}
\begin{aligned}
          \max_{y\in \R^k}  &  \qquad   y^\top \tilde Q y + 2 \tilde d^\top y + \tilde \nu  \\
        ~ \textrm{s.t.} & \qquad  \tilde F ^\top y \leq  \tilde w\\ &  \qquad  y \geq 0
    \end{aligned}
\end{equation}
where \begin{align}\label{a:qerf}\tilde F:=(B^{-1}N)^\top,\enspace \tilde w:=B^{-1}b\geq 0. \end{align}
Recall that
$$
 c_B=(H\bar x+p)_{B}=H_{BB}B^{-1}b+p_B,\enspace c_N=(H\bar x+p)_{N}=H^\top_{BN} B^{-1}b+p_N.
 $$
 Thus
 $$
 c_B^\top B^{-1}N-c_N^\top=\left(H_{BB}B^{-1}b+p_B\right)^\top B^{-1}N-
\left( H^\top_{BN} B^{-1}b+p_N\right)^\top=(-\tilde d)^\top,$$
and it follows from~\eqref{a:sub2ebf} that $\tilde d\leq 0$.
We also know that
\begin{align}
\Phi(\bar y)= \bar x^\top H \bar x
+2 p^\top \bar x+\nu = \tilde \nu,\end{align}
where the first equality can be checked from the definition of $H$ and $\bar x$, the second equality follows from~\eqref{eq:sederrr} and the fact that $\bar x_N=0$. Hence we have $\tilde \nu<\nu_R$. Therefore~\eqref{eq:liftedr2} is an equivalent model of~\eqref{prob:original_max_qp0} and it satisfies $\tilde w\geq 0,\enspace \tilde d\leq 0, \enspace \tilde \nu<\nu_R$.
The above transformation procedure is summarized in~\Cref{alg:minimalprogram}.

\begin{algorithm}
    \caption{\texttt{{Change\_of\_Coordinate}}}
    \begin{algorithmic}[1]
        \Require  KKT vertex $\bar y\in \R^k$,  matrix $Q\in \cS^k$, vector $d\in \R^k$, constant $\nu\in \R$, matrix $F\in \R^{k\times m}$, vector $w\in \R^m$.
\State $H\leftarrow \begin{pmatrix}
            Q & \bzero  \\
            \bzero^\top & 0
        \end{pmatrix}$, $p\leftarrow \begin{pmatrix} d \\ \bzero\end{pmatrix}$, $A \leftarrow \begin{pmatrix} F^\top & \bI\end{pmatrix} $, $b\leftarrow w$.
        \State $\bar x\leftarrow \begin{pmatrix}\bar y\\  w-F^\top \bar y\end{pmatrix}$.
        \State $c\leftarrow H\bar x+p$.
\State Choose basis matrix $B$ and non-basis matrix $N$ associated with $\bar x$ such that~\eqref{sub:qeset} holds.
\State Compute {$(\tilde{Q},\tilde{d},\tilde{F},\tilde{w},\tilde{\nu})$} according to~\eqref{eq:symb_def} and~\eqref{a:qerf}.
        \Ensure {$(\tilde{Q},\tilde{d},\tilde{F},\tilde{w},\tilde{\nu})$} .
    \end{algorithmic}
    \label{alg:minimalprogram}
\end{algorithm}

\section{Proof of \texorpdfstring{\Cref{l:posdf}}{Lemma 7}}\label{sec:proofof}
\begin{proof}[proof of~\Cref{l:posdf}]
Note that $\tau_i=0$ if and only if $Q_{ii}=0$.
Therefore, we have $Q_{11}=\cdots=Q_{\ell \ell}=0$.
Since $Q\succeq 0$, the first $\ell$ columns of $Q$ are all zero vectors, i.e., $Q_1=\cdots=Q_\ell=\bf{0}$.
It follows that for any scalar $\eta \geq 0$ we have
$$
\begin{aligned}
    g(\eta^{-1} e_i) & = \max\{(\eta^{-1} Q_i+d)^\top \tilde y+d_i+\nu: \tilde y\in \cF_\tau\} \\
    & = \max\{d^\top \tilde y+ d_i +\nu: \tilde y\in \cF_\tau\} \\
    & \leq \nu \leq \nu_R-\delta', \qquad \forall i = 1, \ldots, \ell.
\end{aligned}
$$
Here, the first inequality follows from $d\leq 0$ and $\tilde y\geq 0$ for any $\tilde y\in \cF_\tau$.
For $i \in \{\ell + 1, \ldots, k\}$, using the function $\Psi$ defined in~\eqref{eq:bilinear-form-psi}, we obtain
\begin{equation*}
    g(\tau_i^{-1} e_i) = \Psi(\tau_i^{-1} e_i, y^i) \leq \frac{\Phi(y^i) + \Phi(\tau_i^{-1} e_i)}{2} \leq \frac{\nu_R - \delta + \omega}{2}\leq \frac{\nu_R+\omega}{2}.
\end{equation*}
where the first inequality follows from~\eqref{a:phisdfs}, and the second is due to the definition of Tuy's cut and $\omega$.
\qed
\end{proof}

\section{Further Discussion on Konno's Cut}\label{sec:konnocut}
Recall in Section~\ref{subsec:konno}, we have discussed the generation of the Konno's cut by computing $\theta_K$, and the elements $\{(\theta_K)_i: i=1,\ldots, k\}$ are defined in~\eqref{a:deflamdai}.
In this section, we argue that Konno's cut exists and is computable by solving $k$ LP problems, given that $\delta<\nu_R-\nu$.

We first address~\Cref{rem:compute_alpha_as_LP}. It was shown in
\cite{konno1976maximization}  that for each $i\in \{1,\ldots,k\}$,
$(\theta_K)_i$ can be computed by solving the following LP:
\begin{equation}
    \label{prob:konnocuttheta2}
    \begin{aligned}
   { (\theta_K)_i^{-1}} =  \min \quad & -d^\top z+(\nu_R-\delta-\nu)z_0 \\
        \textrm{s.t.}
        \quad &  -{F^\top} z+w z_0  \geq 0 \\
        \quad &  \tau^\top z-z_0 \geq 0 \\
  \quad & (Q_{i})^\top z+d_i z_0=1\\
        \quad & z \geq 0,\enspace z_0\geq 0 .
    \end{aligned}
\end{equation}
We refer interested readers to~\cite{konno1976maximization,tuy2016nonconvex} for more details.
The pseudo code for the generation of Konno's cut is summarized in Algorithm~\ref{alg:Konnocut} below.

Next, we demonstrate that even if LP~\eqref{prob:konnocuttheta2} is pathological, $(\theta_K)_i$ is still well-defined under Assumption~\ref{ass:bound_and_interior}.
The first case is when~\eqref{prob:konnocuttheta2} is infeasible, then the optimal value of~\eqref{prob:konnocuttheta2} is $+\infty$, and thus $(\theta_K)_i=0$.
The other pathological case is when the optimal value of~\eqref{prob:konnocuttheta2} is $0$, in which case $(\theta_K)_i=+\infty$ and Konno's cut does not exist.
However, as long as $\delta<\nu_R-\nu$, such case never occurs.
We first note that by $\nu_R-\delta-\nu>0$, $d\leq 0$, and the constraints $z\geq 0$, $z_0\geq 0$, the objective of~\eqref{prob:konnocuttheta2} is always nonnegative.
Having an optimal value equal to 0 thus implies that there exists an optimal  solution $(z^*, z_0^*)$ satisfying $z_0^*=0$. It follows that  $F^\top z^*\leq 0$, $z^*\geq 0$ and $(Q_i)^\top z^*=1 $. In particular,
 $z^*\neq 0$ is a direction of $\cF$, which contradicts with the boundedness of $\cF$ in Assumption~\ref{ass:bound_and_interior}.
We then conclude that each $(\theta_K)_i$ is well-defined for any $\delta<\nu_R-\nu$ under Assumption~\ref{ass:bound_and_interior}.

\begin{algorithm}
    \caption{\texttt{Konno\_Cut}}
    \begin{algorithmic}[1]
        \Require $Q,d, \nu, F,w,\nu_R,\tau,\delta$.
        \For{$i=1,\ldots,k$}
        \State Compute $(\theta_K)_i$ by solving the LP problem~\eqref{prob:konnocuttheta2}.
        \EndFor
        \Ensure $\theta_K \in \R_+^k$ such that $\bar\phi^K_{\tau, \theta_K}\leq \nu_R-\delta$.
    \end{algorithmic}
    \label{alg:Konnocut}
\end{algorithm}

\section{Generation of Deeper Cuts}
\label{sec:bisec-cut}

In~\Cref{sec:cuts} it is mentioned that deeper cuts can be generated using a bisection method.
Below in Algorithm~\ref{alg:ikc} we attach the pseudo code for generation of those deeper cuts.

\begin{algorithm}[ht!]
    \caption{\texttt{DNN\_Cut}}
    \begin{algorithmic}[1]
        \Require $Q,d, \nu, F,w,\nu_R,\tau,\theta_0,\eta, \delta$.
 \State $ \theta\leftarrow\theta_0$;
    \State Compute $\bar\phi^L_{{\tau}, \eta\theta_0}$ based on~\eqref{prob:konnoQP2e};
        \If {$\bar\phi^L_{\tau, \eta\theta_0}\leq \nu_R-\delta$}
        \State $ \theta\leftarrow \eta\theta_0$;
        \Else
        \State Compute $\bar\phi^D_{\tau, \eta\theta_0}$ based on~\eqref{prob:SQP3Shor3ae}.
  \If {$\bar\phi^D_{\tau, \eta\theta_0}\leq \nu_R-\delta$}
  \State $ \theta\leftarrow \eta\theta_0$;
  \EndIf
        \EndIf
        \Ensure $\theta\leq \theta_0$ such that $\phi^*_{\tau, \theta} \leq \nu_R-\delta$.
    \end{algorithmic}
    \label{alg:ikc}
\end{algorithm}

\section{Supplementary Numerical Results}\label{sec:supp-numexp}

In addition to the numerical experiments presented in~\Cref{sec:experiments}, we provide more detailed results for \texttt{QuadProgCD-G} in this section.
\Cref{tab:supp-numexp-biodata-100} supplements~\Cref{tab:conv-comparison-biodata} by offering additional information about the algorithm performance on the instances from the real dataset BioData100.
Specifically, the table includes the number of iteration performed by \texttt{QuadProgCD-G} before convergence or reaching the time limit, and the proportion of wallclock time spent on different tasks.
The tasks include the search of KKT points, computation of DNN bounds, performing Konno's cut (\Cref{alg:Konnocut}), and performing DNN cuts (\Cref{alg:ikc}).
The same information for datasets PCQMAX and CQMAX is presented in~\Cref{tab:supp-numexp-synthetic}.
According to the tables, the majority of computation efforts is on computing DNN bounds, and generating DNN cuts takes the second place.
This is not surprising since these two tasks requires solving several semidefinite programs.

We also test the per-instance variation of {the} lower and upper objective bounds computed by \texttt{QuadProgCD-G} with respect to iteration, as shown in~\Cref{fig:time-lbub}.
The test instances are chosen to be the same ones as in~\Cref{fig:time-relgap}.
It can be observed that the upper bound for objective value is \emph{not} decreased for each iteration, meaning that adding a single cut may not reduce the upper bound.
Therefore, keeping adding various cuts to the feasible region is necessary for the convergence.

\begin{table}[htpb]
    \centering
    \caption{Detailed numerical results for \texttt{QuadProgCD} on real dataset BioData100.
    A red instance index means that \texttt{QuadProgCD} failed to solve the corresponding instance to the prescribed tolerance ($\epsilon = 10^{-6}$) in the time limit ($3600$ seconds).
    The percentages in each row may not add up to $100\%$ due to time spent on pre-processing and post-processing.}
    \begin{tabular}{c c c c c c }
        \toprule
         & & \multicolumn{4}{c}{Time} \\
        \cmidrule(l{1pt}r{2pt}){3-6}
        Instance & Iteration & KKT Points & DNN Bound & Konno's Cut & DNN Cut \\
        \midrule
        257424 & 5 & 0.23\% & 61.05\% & 0.64\% & 38.05\% \\
        84896 & 3 & 0.28\% & 93.15\% & 0.80\% & 5.74\% \\
        141783 & 3 & 0.21\% & 68.57\% & 0.57\% & 30.64\% \\
        123044 & 2 & 0.29\% & 96.16\% & 0.60\% & 2.93\% \\
        303256 & 2 & 0.32\% & 93.66\% & 0.72\% & 5.28\% \\
        {312988} & 37 & 0.08\% & 60.27\% & 0.33\% & 39.32\% \\
        229644 & 1 & 0.21\% & 96.47\% & 1.09\% & 2.17\% \\
        259490 & 2 & 0.28\% & 93.29\% & 0.68\% & 5.73\% \\
        143339 & 2 & 0.27\% & 95.16\% & 0.60\% & 3.95\% \\
        95618 & 12 & 0.15\% & 62.00\% & 0.56\% & 37.28\% \\
        23659 & 2 & 0.27\% & 95.36\% & 0.63\% & 3.71\% \\
        284341 & 2 & 0.24\% & 95.11\% & 0.56\% & 4.07\% \\
        79834 & 2 & 0.28\% & 94.65\% & 0.61\% & 4.44\% \\
        88552 & 2 & 0.25\% & 98.25\% & 0.60\% & 0.88\% \\
        305293 & 6 & 0.20\% & 70.23\% & 0.64\% & 28.93\% \\
        149864 & 1 & 0.27\% & 98.40\% & 1.29\% & 0.00\% \\
        81426 & 4 & 0.19\% & 73.45\% & 0.69\% & 25.65\% \\
        274863 & 4 & 0.17\% & 66.34\% & 0.60\% & 32.88\% \\
        82857 & 2 & 0.27\% & 93.97\% & 0.60\% & 5.15\% \\
        211213 & 3 & 0.26\% & 92.89\% & 0.74\% & 6.09\% \\
        \bottomrule
    \end{tabular}
    \label{tab:supp-numexp-biodata-100}
\end{table}

\begin{table}[htpb]
    \centering
    \caption{Detailed numerical results for \texttt{QuadProgCD} on synthetic datasets PCQMAX and CQMAX.
    The percentages in each row may not add up to $100\%$ due to time spent on pre-processing and post-processing.}
    \fontsize{7pt}{7pt}\selectfont
    \begin{tabular}{c c c c c c }
        \toprule
         & & \multicolumn{4}{c}{Time} \\
        \cmidrule(l{1pt}r{2pt}){3-6}
        Instance & Iteration & KKT Points & DNN Bound & Konno's Cut & DNN Cut \\
        \midrule
        pcqmax20-1 & 1 & 44.92\% & 29.96\% & 16.78\% & 0.00\% \\
        pcqmax20-2 & 1 & 44.06\% & 16.39\% & 36.23\% & 0.00\% \\
        pcqmax20-3 & 2 & 19.06\% & 44.73\% & 23.06\% & 7.19\% \\
        pcqmax20-4 & 1 & 11.13\% & 36.80\% & 38.89\% & 5.45\% \\
        pcqmax20-5 & 1 & 25.83\% & 62.66\% & 0.00\% & 0.00\% \\
        pcqmax20-6 & 1 & 5.86\% & 36.42\% & 52.58\% & 0.00\% \\
        pcqmax20-7 & 2 & 16.57\% & 41.01\% & 33.90\% & 3.23\% \\
        pcqmax20-8 & 1 & 12.17\% & 37.85\% & 44.41\% & 0.00\% \\
        pcqmax20-9 & 1 & 11.94\% & 35.74\% & 47.68\% & 0.00\% \\
        pcqmax20-10 & 1 & 11.29\% & 26.05\% & 58.34\% & 0.00\% \\
        pcqmax50-1 & 1 & 1.98\% & 83.72\% & 13.72\% & 0.00\% \\
        pcqmax50-2 & 1 & 4.93\% & 83.86\% & 10.60\% & 0.00\% \\
        pcqmax50-3 & 2 & 3.43\% & 85.03\% & 5.77\% & 5.09\% \\
        pcqmax50-4 & 1 & 2.17\% & 75.51\% & 10.14\% & 11.18\% \\
        pcqmax50-5 & 2 & 3.33\% & 85.05\% & 6.21\% & 4.83\% \\
        pcqmax50-6 & 2 & 5.70\% & 83.94\% & 5.74\% & 4.09\% \\
        pcqmax50-7 & 1 & 4.29\% & 95.08\% & 0.00\% & 0.00\% \\
        pcqmax50-8 & 1 & 2.06\% & 85.78\% & 11.58\% & 0.00\% \\
        pcqmax50-9 & 1 & 11.63\% & 64.98\% & 22.78\% & 0.00\% \\
        pcqmax50-10 & 1 & 2.47\% & 67.91\% & 16.22\% & 12.46\% \\
        pcqmax100-1 & 2 & 0.30\% & 93.86\% & 0.89\% & 4.92\% \\
        pcqmax100-2 & 1 & 0.28\% & 93.06\% & 1.36\% & 5.22\% \\
        pcqmax100-3 & 2 & 0.62\% & 94.02\% & 0.80\% & 4.54\% \\
        pcqmax100-4 & 2 & 0.65\% & 91.62\% & 0.88\% & 6.83\% \\
        pcqmax100-5 & 2 & 0.48\% & 94.33\% & 0.81\% & 4.36\% \\
        pcqmax100-6 & 2 & 0.69\% & 95.71\% & 0.62\% & 2.95\% \\
        pcqmax100-7 & 3 & 0.58\% & 91.75\% & 1.00\% & 6.64\% \\
        pcqmax100-8 & 2 & 0.24\% & 91.24\% & 0.87\% & 7.63\% \\
        pcqmax100-9 & 2 & 0.67\% & 96.21\% & 0.63\% & 2.46\% \\
        pcqmax100-10 & 2 & 0.24\% & 92.82\% & 0.82\% & 6.10\% \\
        \midrule
        cqmax20-1 & 1 & 28.60\% & 31.66\% & 35.14\% & 0.00\% \\
        cqmax20-2 & 1 & 19.73\% & 36.85\% & 39.32\% & 0.00\% \\
        cqmax20-3 & 1 & 27.12\% & 31.25\% & 37.68\% & 0.00\% \\
        cqmax20-4 & 1 & 26.27\% & 26.16\% & 43.73\% & 0.00\% \\
        cqmax20-5 & 1 & 7.36\% & 35.53\% & 51.42\% & 0.00\% \\
        cqmax20-6 & 1 & 14.37\% & 39.98\% & 41.18\% & 0.00\% \\
        cqmax20-7 & 1 & 16.09\% & 35.26\% & 35.99\% & 6.16\% \\
        cqmax20-8 & 1 & 14.59\% & 37.36\% & 43.46\% & 0.00\% \\
        cqmax20-9 & 1 & 21.29\% & 35.86\% & 38.69\% & 0.00\% \\
        cqmax20-10 & 1 & 8.44\% & 41.69\% & 45.49\% & 0.00\% \\
        cqmax50-1 & 1 & 2.58\% & 83.90\% & 12.82\% & 0.00\% \\
        cqmax50-2 & 1 & 2.97\% & 76.89\% & 12.49\% & 6.49\% \\
        cqmax50-3 & 1 & 4.91\% & 76.77\% & 11.38\% & 5.93\% \\
        cqmax50-4 & 1 & 3.65\% & 83.82\% & 11.79\% & 0.00\% \\
        cqmax50-5 & 2 & 4.32\% & 87.41\% & 4.48\% & 3.24\% \\
        cqmax50-6 & 2 & 4.83\% & 86.43\% & 5.05\% & 3.16\% \\
        cqmax50-7 & 2 & 3.59\% & 88.31\% & 5.85\% & 1.71\% \\
        cqmax50-8 & 2 & 4.75\% & 84.07\% & 5.79\% & 4.92\% \\
        cqmax50-9 & 1 & 4.36\% & 84.58\% & 10.48\% & 0.00\% \\
        cqmax50-10 & 1 & 2.87\% & 82.76\% & 13.64\% & 0.00\% \\
        cqmax100-1 & 2 & 0.77\% & 96.69\% & 0.52\% & 1.98\% \\
        cqmax100-2 & 1 & 1.70\% & 97.04\% & 1.20\% & 0.00\% \\
        cqmax100-3 & 2 & 0.78\% & 97.11\% & 0.48\% & 1.58\% \\
        cqmax100-4 & 2 & 0.49\% & 96.93\% & 0.53\% & 2.00\% \\
        cqmax100-5 & 1 & 1.72\% & 96.86\% & 1.34\% & 0.00\% \\
        cqmax100-6 & 3 & 0.62\% & 83.99\% & 0.54\% & 14.81\% \\
        cqmax100-7 & 2 & 1.01\% & 96.54\% & 0.49\% & 1.93\% \\
        cqmax100-8 & 3 & 0.92\% & 76.25\% & 0.45\% & 22.35\% \\
        cqmax100-9 & 3 & 0.59\% & 73.98\% & 0.42\% & 24.99\% \\
        cqmax100-10 & 2 & 1.06\% & 97.14\% & 0.45\% & 1.30\% \\
        \bottomrule
    \end{tabular}
    \label{tab:supp-numexp-synthetic}
\end{table}

\begin{figure}[htbp]
    \centering
    \begin{subfigure}[t]{0.49\textwidth}
        \includegraphics[width=\textwidth]{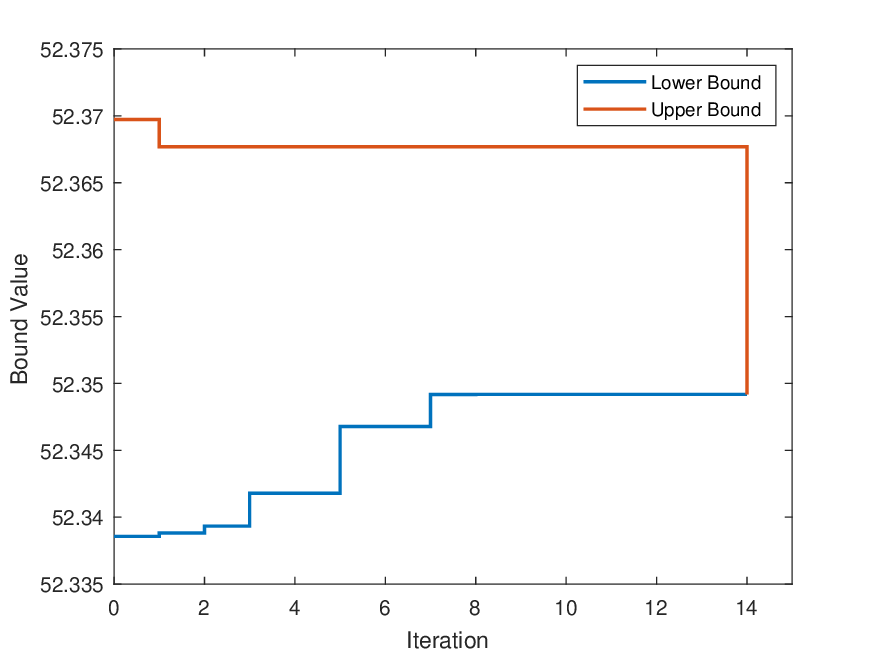}
        \caption{BioData instance 257424.}
        \label{fig:time-lbub-biodata-257424}
    \end{subfigure}
    \begin{subfigure}[t]{0.49\textwidth}
        \includegraphics[width=\textwidth]{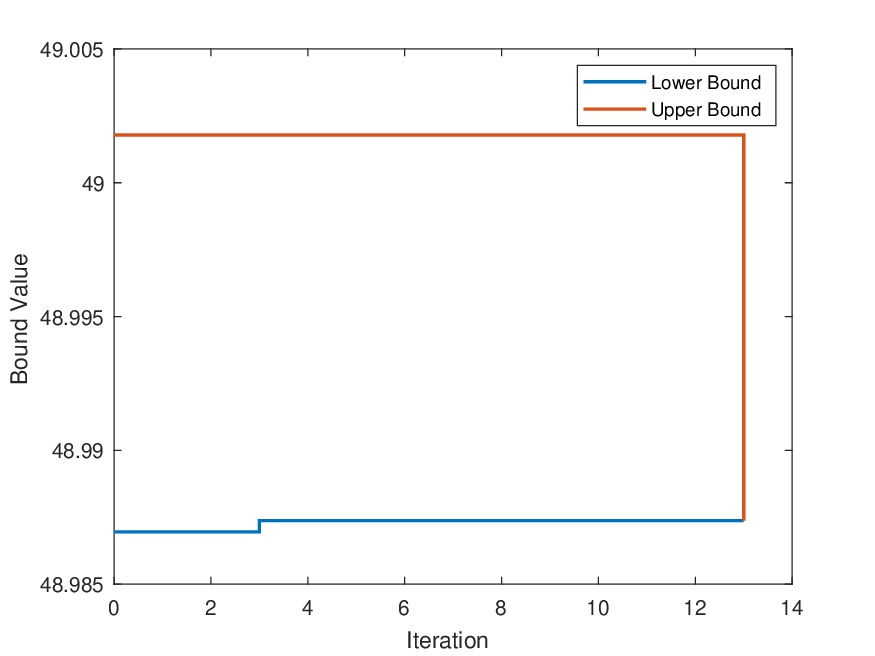}
        \caption{BioData instance 305293.}
        \label{fig:time-lbub-biodata-305293}
    \end{subfigure}
    \hfill
    \begin{subfigure}[t]{0.49\textwidth}
        \includegraphics[width=\textwidth]{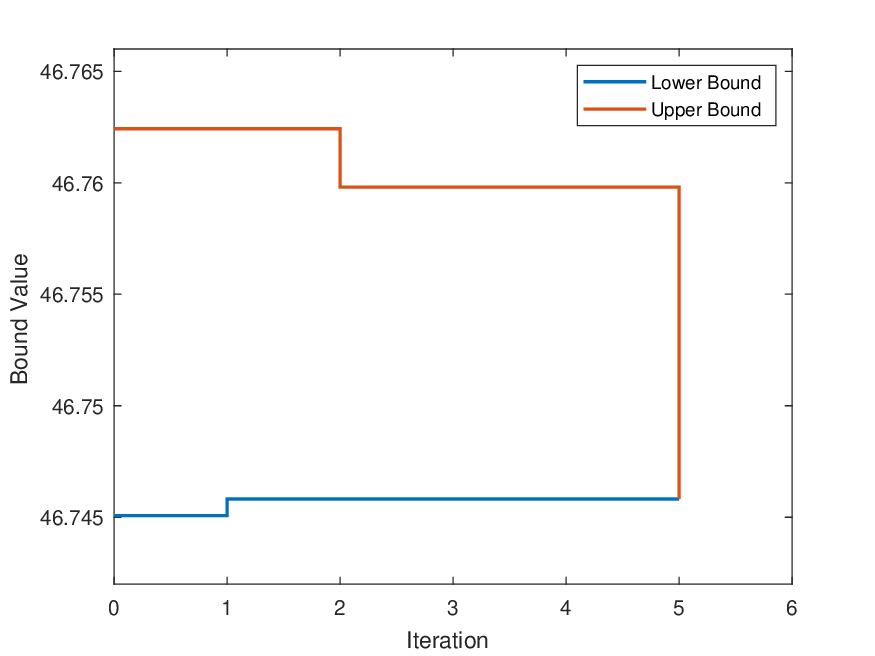}
        \caption{BioData instance 141783.}
        \label{fig:time-lbub-biodata-141783}
    \end{subfigure}
    \begin{subfigure}[t]{0.49\textwidth}
        \includegraphics[width=\textwidth]{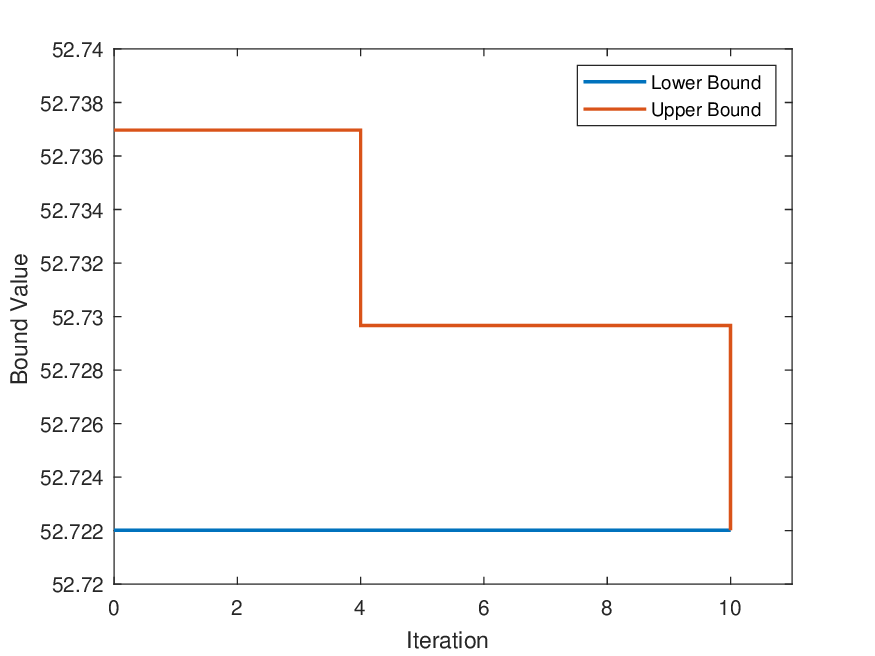}
        \caption{BioData instance 303256.}
        \label{fig:time-lbub-biodata-303256}
    \end{subfigure}
    \hfill
    \begin{subfigure}[t]{0.49\textwidth}
        \includegraphics[width=\textwidth]{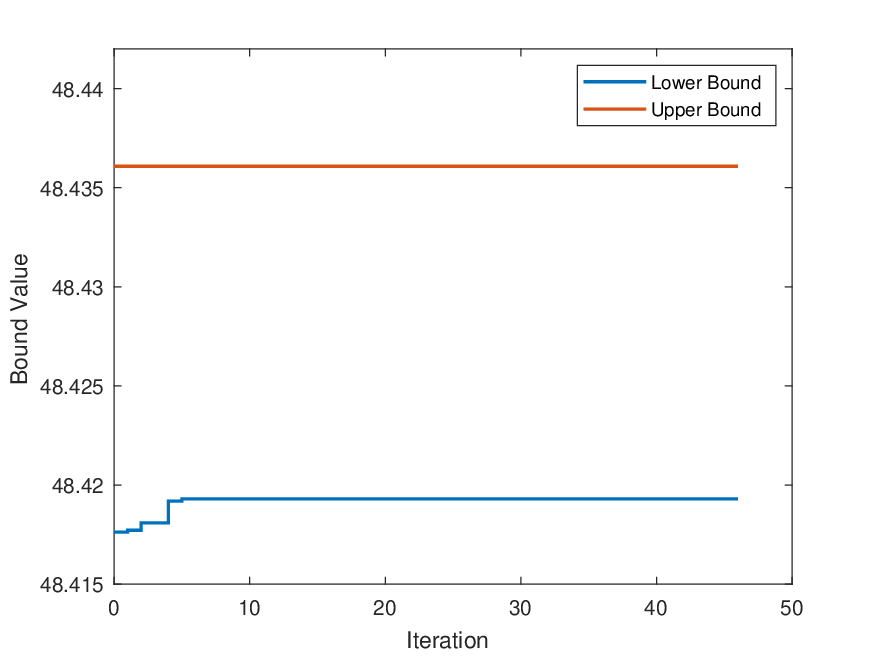}
        \caption{BioData instance 312988.}
        \label{fig:time-lbub-biodata-312988}
    \end{subfigure}
    \begin{subfigure}[t]{0.49\textwidth}
        \includegraphics[width=\textwidth]{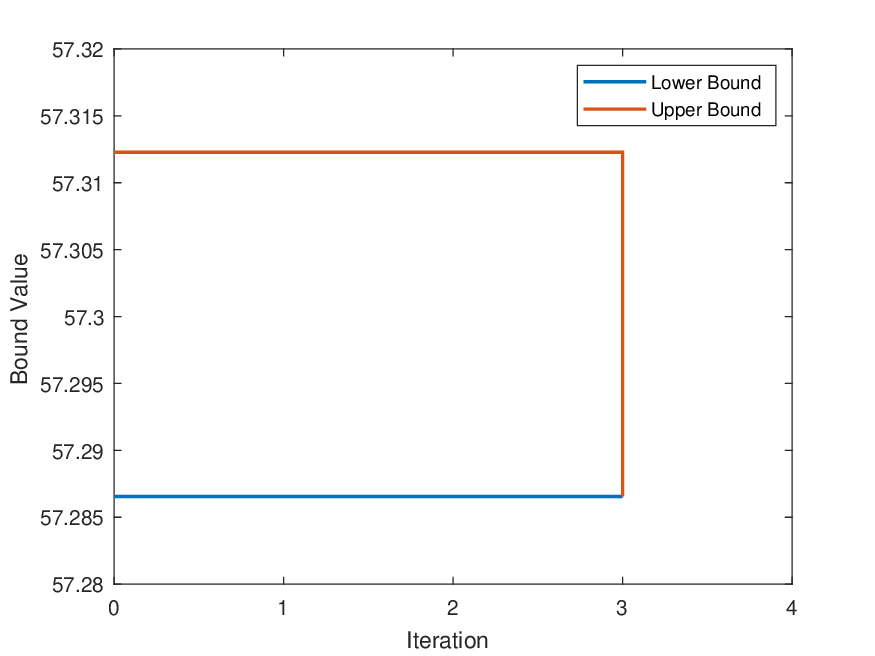}
        \caption{BioData instance 84896.}
        \label{fig:time-lbub-biodata-84896}
    \end{subfigure}
    \caption{Computed lower and upper bounds in each iteration of \texttt{QuadProgCD}.}
    \label{fig:time-lbub}
\end{figure}

\end{document}